\documentclass[smallextended,envcountsame]{svjour3}
\usepackage{trsty-springer}
\smartqed

\newcommand{\kry}{g}  % Maximum Krylov space size
\newcommand{\dm}{n}  % Problem dimension

\newcommand{\Rdd}{\mathbb{R}^{\dm\times \dm}}
\newcommand{\Sstat}{\mathrm{S}}
\newcommand{\Zstat}{\mathrm{Z}}

\title{Statistical Properties of the Probabilistic Numeric\\ Linear Solver BayesCG}
\author{Tim W. Reid, Ilse C. F. Ipsen,\\ Jon Cockayne, and Chris J. Oates \thanks{The work was supported in part by NSF grant DMS-1745654 
(TWR, ICFI), NSF grant DMS-1760374 and
DOE grant DE-SC0022085
(ICFI), and the Lloyd's Register Foundation Programme on Data Centric
Engineering at the Alan Turing Institute (CJO).}}
\institute{Tim W. Reid \at North Carolina State University  \\ Department of Mathematics \\ Raleigh, NC 27695-8205, US \\ twreid@alumni.ncsu.edu \and Ilse C. F. Ipsen \at  North Carolina State University  \\ Department of Mathematics \\ Raleigh, NC 27695-8205, US  \\ ipsen@ncsu.edu \and Jon Cockayne \at  University of Southampton \\ Department of Mathematical Sciences \\ Southampton, SO17 1BJ, UK \\ jon.cockayne@soton.ac.uk \and Chris J. Oates \at  Newcastle University \\ School of Mathematics and Statistics \\ Newcastle-upon-Tyne, NE1 7RU, UK \\ chris.oates@ncl.ac.uk}
\date{\today}

\begin{document}

\maketitle

\abstract{
We analyse the calibration of BayesCG under the Krylov prior, a
probabilistic numeric extension of the Conjugate Gradient (CG) method for 
solving systems of linear equations
 with symmetric positive definite coefficient matrix. Calibration refers
  to the statistical quality of the posterior covariances produced by a solver.
  Since BayesCG is not calibrated in the strict existing notion, we propose instead two test statistics that are necessary but not sufficient for calibration:  the $Z$-statistic and the new $S$-statistic. 
  We show analytically and experimentally that under low-rank approximate Krylov posteriors, BayesCG exhibits desirable properties of a calibrated solver, 
  is only slightly optimistic, and  is computationally competitive with CG.
  }

\section{Introduction}
We present a rigorous analysis of
the probabilistic numeric solver  BayesCG under the Krylov prior \cite{Cockayne:BCG,RICO21} 
for 
solving systems of linear equations
\begin{equation}
  \label{Eq:Axb}
  \Amat\xvec_* = \bvec,
\end{equation}
with symmetric positive definite coefficient matrix $\Amat\in\Rnn$.

\paragraph{Probabilistic numerics.} 
This area \cite{Cockayne:BPNM,HOG15,Oates} seeks 
to quantify the uncertainty due to limited computational resources, and to propagate these uncertainties through computational pipelines---sequences of computations where the output of one computation is the input for the next.
At the core of many computational pipelines are iterative linear solvers
\cite{CIOR20,HBH20,NW06,PZSHG12,SHB21},
whose computational resources are limited by the impracticality of running the solver to completion.
The premature termination generates uncertainty in the computed solution.

\paragraph{Probabilistic numeric linear solvers.} 
Probabilistic numeric extensions of
Kry\--lov space and stationary iterative 
methods
\cite{Bartels,CIOR20,Cockayne:BCG,Fanaskov21,Hennig,RICO21,WH20}
model the `epistemic uncertainty' in 
a quantity of interest, which can be the matrix inverse $\Amat^{-1}$ \cite{Hennig,Bartels,WH20} or the solution $\xvec_*$ \cite{Cockayne:BCG,Bartels,CIOR20,Fanaskov21}. Our
quantity of interest is the solution $\xvec_*$, and the `epistemic uncertainty' is the uncertainty in the user's knowledge of the true value of $\xvec_*$. 

The probabilistic solver takes as input a \emph{prior distribution} which models the initial uncertainty in $\xvec_*$ and then computes \emph{posterior distributions} which model the uncertainty remaining after each iteration.
\figref{F:PriorPosterior} depicts a prior and posterior distribution for the solution $\xvec_*$ of 2-dimensional linear system. 

\begin{figure}
  \centering
  \includegraphics[scale=.35]{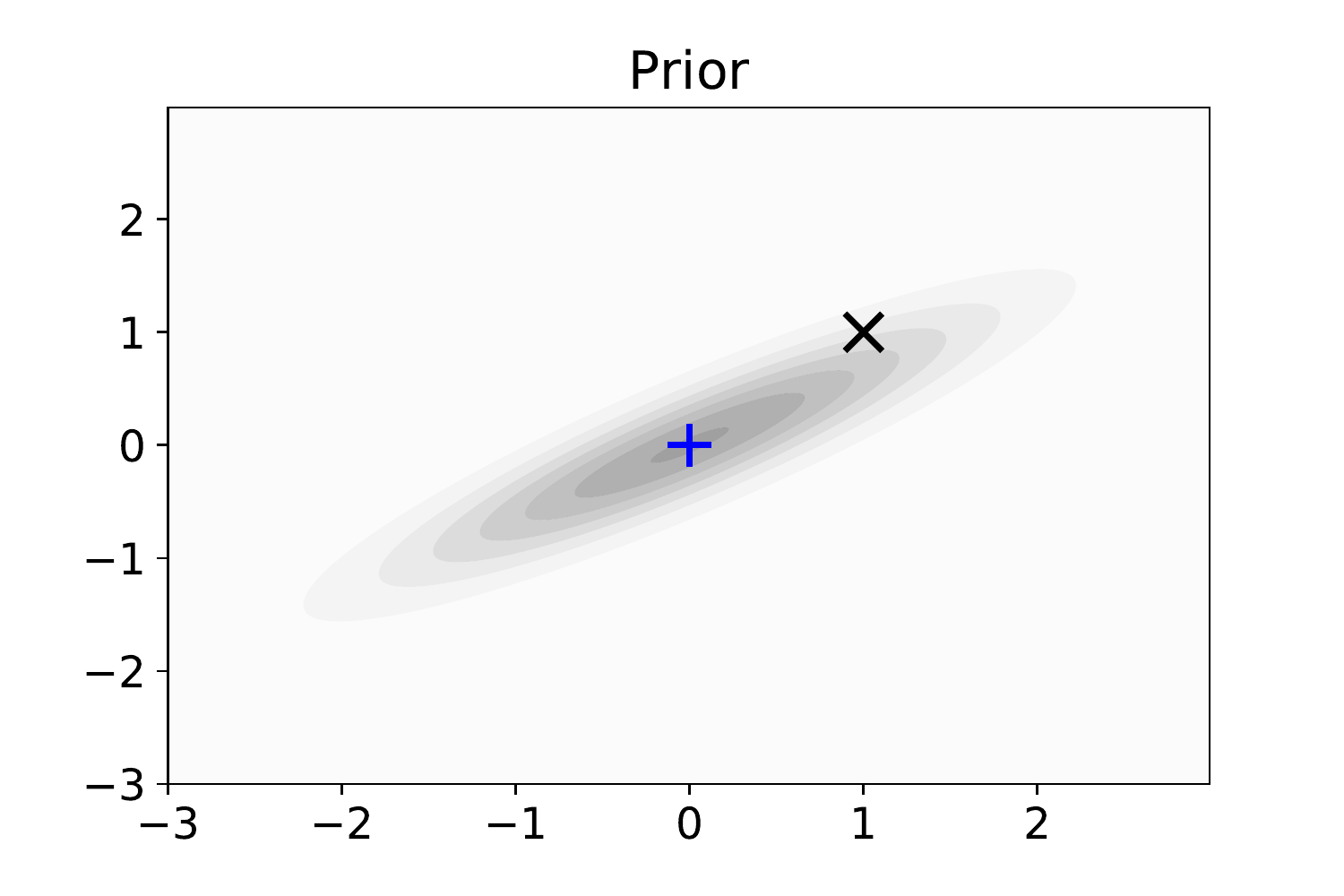} \\
  \includegraphics[scale=.35]{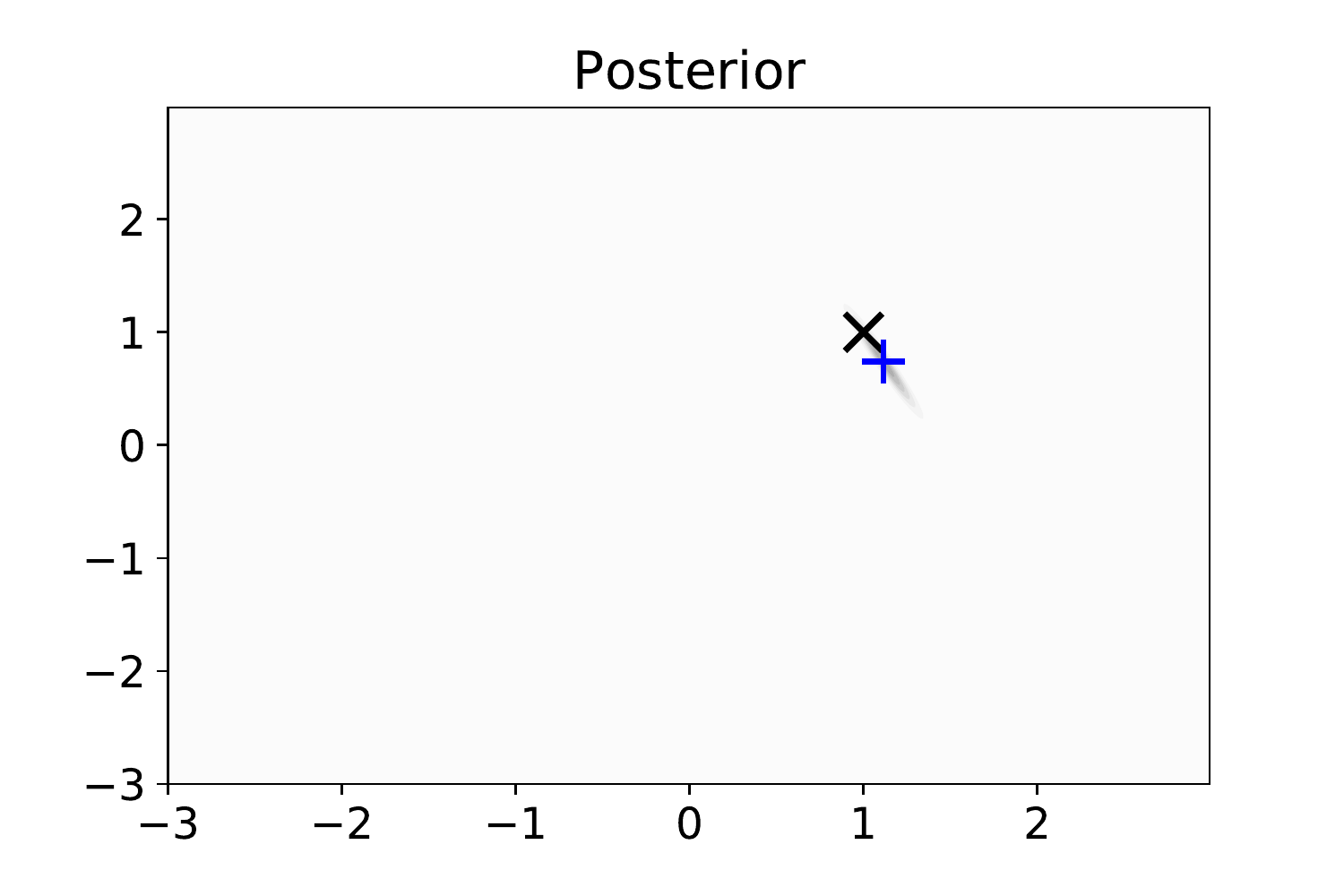}
  \includegraphics[scale=.35]{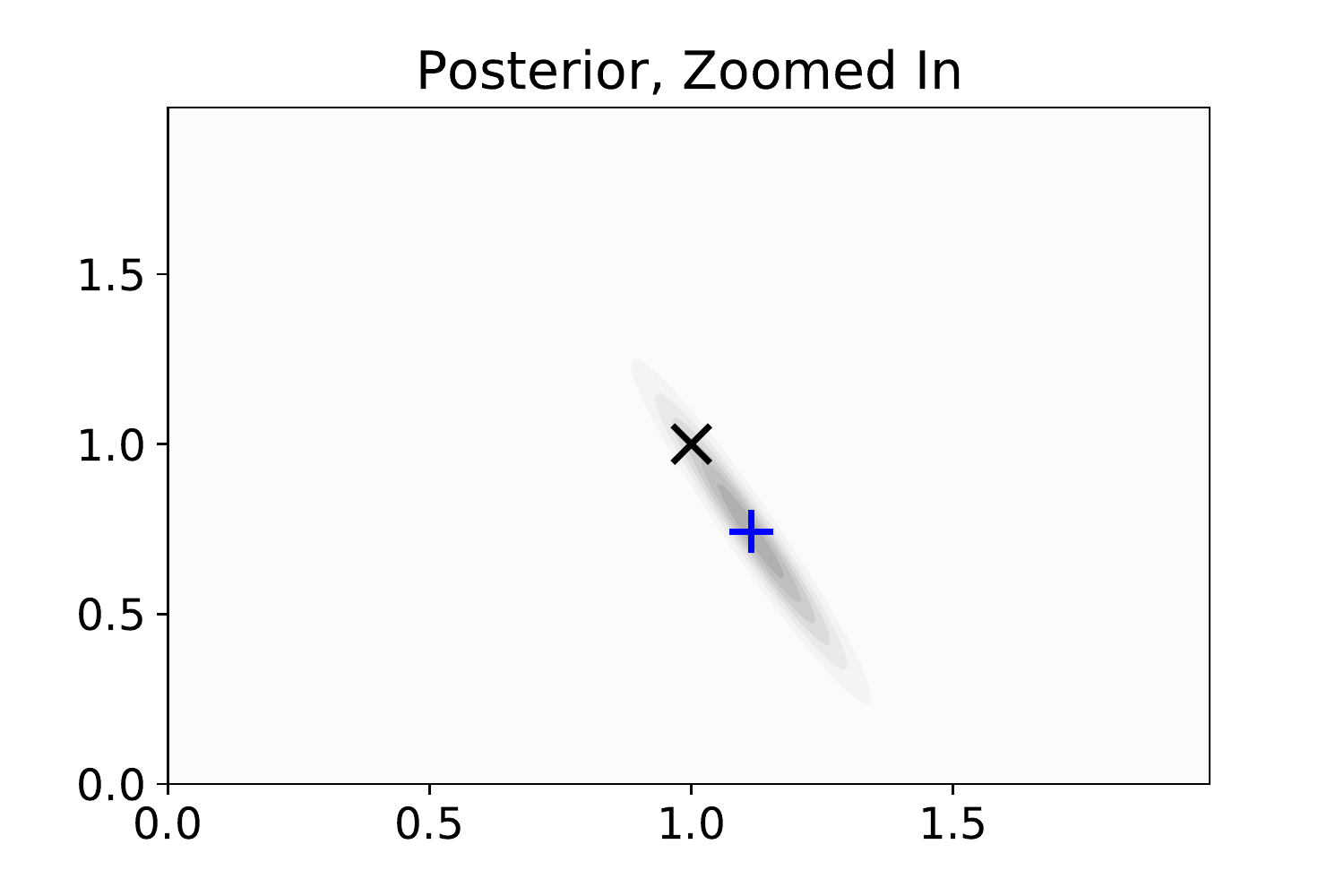}
  \caption{Prior and posterior distributions for a linear system (\ref{Eq:Axb}) with $n=2$. Top plot: prior distribution. Bottom plots: posterior distributions, where the bottom right is a zoomed in version of the bottom left. The gray shaded contours represent the areas in which the distributions are concentrated, the symbol `$\times$' represents the solution, and the symbol `\textcolor{DarkBlue}{+}' the mean of the prior or posterior.}
  \label{F:PriorPosterior}
\end{figure}

\paragraph{Calibration.} An important criterion of probabilistic  solvers is the statistical quality of their posterior distributions.
A solver is `calibrated' if its posterior distributions accurately model the uncertainty in $\xvec_*$ \cite[Section 6.1]{Cockayne:BCG}. Probabilistic Krylov solvers are not always calibrated because their posterior distributions tend to be \emph{pessimistic}. This means, the posteriors imply that the error is larger than it actually is
\cite[Section 6.4]{Bartels},
\cite[Section 6.1]{Cockayne:BCG}.  Previous efforts for improving calibration have focused on scaling the posterior covariances 
\cite[Section 4.2]{Cockayne:BCG},
\cite[Section 7]{Fanaskov21},
\cite[Section 3]{WH20}.

\paragraph{BayesCG.}
We analyze the calibration of BayesCG
under the Krylov prior \cite{Cockayne:BCG,RICO21}.
BayesCG was introduced in \cite{Cockayne:BCG}
as a probabilistic numeric extension of the Conjugate Gradient (CG) method \cite{Hestenes}  for solving
(\ref{Eq:Axb}). The 
Krylov prior proposed in
\cite{RICO21} makes BayesCG competitive
with CG. The numerical properties of BayesCG under the Krylov prior are
analysed in \cite{RICO21}, while here
we analyse its statistical properties.

\begin{comment}
In addition to BayesCG \cite{Cockayne:BCG,RICO21}, probabilistic extensions of  CG \cite{Hennig}, GMRES \cite{Bartels}, and stationary iterative methods \cite{CIOR20} have been developed. 

is a solver that iteratively computes an approximate solution to \eref{Eq:Axb}, and it is often used when $\Amat$ is large and sparse or is only available through matrix vector products. The Bayesian Conjugate Gradient method (BayesCG) is a probabilistic generalization of CG that computes an approximate solution to \eref{Eq:Axb} and also computes statistical model of the error between the computed solution and the true solution.

Probabilistic numerical methods have also been developed for many problems in computational mathematics, such as Bayesian optimization \cite{Mockus75},  numerical integration \cite{BOGOS19,GKH20,KOS18}, sparse Cholesky decompositions \cite{SSO21}, least-squares regression \cite{BH2016}, and solution of ordinary and partial differential equations \cite{COSG17,MM21,OCAG19,TKSH19}. A review of the history of probabilistic numerics can be found in \cite{Oates}, and a review of the modern foundations of Bayesian probabilistic numerical methods can be found in \cite{Cockayne:BPNM}.

Probabilistic solvers produce a statistical model of the error by viewing linear system solution as a statistical inference problem. 
\end{comment}

\subsection{Contributions and Overview}

\paragraph{Overall Conclusion.}
BayesCG under the Krylov prior is not calibrated in the strict sense, but has the desirable properties of a calibrated solver. Under the efficient approximate Krylov posteriors, BayesCG is only slightly optimistic and competitive with CG.

\paragraph{Background (\sref{S:Background}).}
We present a short review of 
BayesCG, and the Krylov prior and posteriors.

\paragraph{Approximate Krylov posteriors (\sref{S:KrylovProperties}).}
  We define the $\Amat$-Wasserstein distance 
(Definition  \ref{D:BWasserstein}, Theorem \ref{L:WassersteinNorm});
  determine the error between Krylov posteriors and their low-rank approximations in the $\Amat$-Wasserstein distance (\tref{T:Wasserstein}); and  present a statistical interpretation of a Krylov prior as an empirical Bayesian procedure (Theorem~\ref{T:KrylovEmp},
  Remark~\ref{R:DetRandom}). 
  
\paragraph{Calibration (\sref{S:Calibration}).}
We review the strict notion of calibration for probabilistic solvers 
(Definition~\ref{D:StrongCalibration},
Lemma~\ref{L:CalibrationAlt}), and show that it does not apply to BayesCG under the Krylov prior (Remark~\ref{r_41}).

We relax the strict notion above and propose
as an alternative form of assessment  
 two test statistics that are necessary but not sufficient for calibration:  the $Z$-statistic (\tref{T:ZStat}) and the new $S$-statistic (\tref{T:CalibratedTrace}, \dref{D:SStat}). We present implementations for both (Algorithms \ref{A:ZStat} and
\ref{A:SStat}); and apply a Kolmogorov-Smirnov statistic
(Definition~\ref{D:KSTest})
for evaluating the quality of samples from the $Z$-statistic.

The $Z$-statistic
is inconclusive about the calibration of BayesCG under the Krylov prior (\tref{T:KrylovZ}), while
the $S$-statistic indicates that it is not calibrated (Section~\ref{S:BayesCGS}).

\paragraph{Numerical Experiments (\sref{S:Experiments}).}
We create a calibrated but slowly converging version of BayesCG
which has random search directions, and use it as a baseline for comparison with
two BayesCG versions that both replicate CG: BayesCG under the inverse and the Krylov priors. 

We assess calibration with the $Z$- and $S$-statistics for  BayesCG with random search directions
(Algorithms \ref{A:ALanczos}
and~\ref{A:BayesCGC});
BayesCG under the inverse prior
(Algorithms \ref{A:BayesCG} and~\ref{A:BayesCGF});
and BayesCG under the Krylov prior with full posteriors (Algorithm~\ref{A:BayesCGKFull})
and approximate posteriors
(Algorithm \ref{A:BayesCGWithoutBayesCG}).
  
Both, $Z$- and $S$ statistics indicate that BayesCG with random search directions is indeed a calibrated solver, while BayesCG under the inverse prior is pessimistic.

The $S$-statistic indicates that 
BayesCG under full Krylov posteriors mimics a calibrated solver, and that
BayesCG under rank-50 approximate posteriors does as well, 
although not as much, since it is slightly optimistic.  

\subsection{Notation}
Matrices are represented in 
bold uppercase, such as $\Amat$; vectors in bold lowercase, such as $\bvec$; and scalars in lowercase, such as $m$.

The $m\times m$ identity matrix is $\Imat_m$, or just $\Imat$ if the dimension is clear. The Moore-Penrose inverse of a matrix $\Amat$ is $\Amat^\dagger$, and the matrix square root is $\Amat^{1/2}$ \cite[Chapter 6]{Higham08}.

Probability distributions are represented in lowercase Greek, such as $\mu_m$; and random variables in  uppercase, such as $\Xrv$. 
The random variable $\Xrv$ having distribution $\mu$ is represented by
$\Xrv\sim\mu$, and its expectation by $\Exp[\Xrv]$. 

The Gaussian distribution with mean $\xvec\in\Rn$ and covariance $\Sigmat\in\Rnn$ is denoted by $\N(\xvec,\Sigmat)$, and the chi-squared distribution with $f$ degrees of freedom by $\chi_f^2$.

\section{Review of Existing Work}
\label{S:Background}

We review BayesCG (\sref{S:BayesCG}), 
the ideal Krylov prior (\sref{S:KrylovBackground}),
and practical approximations for Krylov posteriors (\sref{S:approxKrylov}).
All statements in this section hold in exact arithmetic.

\subsection{BayesCG}
\label{S:BayesCG}

We review the computation of posterior distributions
for BayesCG under general priors (Theorem~\ref{T:BayesCG}), 
and present a pseudo code for BayesCG (Algorithm~\ref{A:BayesCG}).

Given an initial guess $\xvec_0$, BayesCG \cite{Cockayne:BCG} solves symmetric positive definite linear systems~(\ref{Eq:Axb})
by computing iterates $\xvec_m$ that converge to the solution
$\xvec_*$. In addition, BayesCG 
computes probability distributions that quantify
the uncertainty about the solution at each iteration~$m$.
Specifically, for a user-specified Gaussian prior $\mu_0 \equiv \N(\xvec_0,\Sigmat_0)$, 
BayesCG computes posterior distributions $\mu_m \equiv \N(\xvec_m,\Sigmat_m)$, by conditioning a random variable $\Xrv\sim\mu_0$ on  information from $m$ search directions $\Smat_m$.

\begin{theorem}[{\cite[Proposition 1]{Cockayne:BCG}, \cite[Theorem 2.1]{RICO21}}]
  \label{T:BayesCG}
Let $\Amat\xvec_*=\bvec$ be a linear system where $\Amat\in\Rnn$
is symmetric positive definite.
  Let $\mu_0 \equiv \N(\xvec_0,\Sigmat_0)$ be a prior with symmetric positive semi-definite covariance
  $\Sigmat_0\in\Rnn$, and initial residual $\rvec_0\equiv\bvec_0-\Amat\xvec_0$. 
  
  Pick $m\leq n$ so that  $\Smat_ m\equiv \begin{bmatrix} \svec_1 & \svec_2 & \cdots & \svec_m \end{bmatrix}\in\Real^{\dm\times m}$ has $\rank(\Smat_m) = m$ and $\Lammat_m \equiv \Smat_m^T\Amat\Sigmat_0\Amat\Smat$ is non-singular.
  Then, the BayesCG posterior 
  $\mu_m\equiv  \N(\xvec_m,\Sigmat_m)$ has mean and covariance
  \begin{align}
    \xvec_m &= \xvec_0 + \Sigmat_0 \Amat\Smat_m \Lammat_m^{-1} \Smat_m^T\rvec_0 \label{Eq:XmTheory}\\
    \Sigmat_m &= \Sigmat_0 -  \Sigmat_0 \Amat\Smat_m \Lammat_m^{-1} \Smat_m^T \Amat \Sigmat_0. \label{Eq:SigmTheory}
  \end{align}
\end{theorem}

\aref{A:BayesCG} represents the
 iterative computation of the posteriors from  \cite[Propositions 6 and 7]{Cockayne:BCG}, \cite[Theorem 2.7]{RICO21}. 
To illustrate the resemblance of BayesCG and the Conjugate Gradient
method, we present the most common implementation of CG
in \aref{A:CG}. 

BayesCG (\aref{A:BayesCG}) computes specific search directions $\Smat_m$ 
with two
additional  properties:
\begin{enumerate}
\item They are $\Amat\Sigmat_0\Amat$-orthogonal, which means
that  $\Lammat_m = \Smat_m^T\Amat\Sigmat_0\Amat\Smat_m$ is diagonal \cite[Section 2.3]{Cockayne:BCG}, 
thus easy to invert.
\item They form a basis for the Krylov space  \cite[Proposition S4]{BCG:Supp}
  \begin{equation*}
    \range(\Smat_m) = \Kcal_m(\Amat\Sigmat_0\Amat,\rvec_0) \equiv\spn\{\rvec_0, \Amat\Sigmat_0\Amat\rvec_0, \ldots, (\Amat\Sigmat_0\Amat)^{m-1}\rvec_0\}.
  \end{equation*}
 \end{enumerate}
  
\begin{remark}  
The additional requirement $\xvec_*-\xvec_0 \in \range(\Sigmat_0)$ 
in \aref{A:BayesCG} ensures the nonsingularity of $\Lammat_m$
as required by \tref{T:BayesCG},
even for singular prior covariance matrices $\Sigmat_0$ \cite[Theorem 2.7]{RICO21}.
\end{remark}

\begin{algorithm}
\caption{BayesCG
\cite[Algorithm 2.1]{RICO21}}
\label{A:BayesCG}
\begin{algorithmic}[1]
\State \textbf{Input:} spd $\Amat\in\Rnn$, $\bvec\in\Rn$, prior $\mu_0 = \N(\xvec_0,\Sigmat_0)$ \Comment{with $\xvec_*-\xvec_0\in\range(\Sigmat_0)$}
\State{$  \rvec_0 =  \bvec- \Amat\xvec_0$} \Comment{Initial residual}
\State{$ \svec_1 =  \rvec_0$} \Comment{Initial search direction}
\State{$m=0$} \Comment{Initial iteration count}
\While{not converged}
\State{$m = m+1$}\Comment{Increment iteration count}
\State{$\alpha_m = \left( \rvec_{m-1}^T  \rvec_{m-1}\right)\big/\left( \svec_m^T \Amat\Sigmat_{0}  \Amat\svec_m\right)$}
\State{$ \xvec_m =  \xvec_{m-1} + \alpha_m  \Sigmat_0 \Amat  \svec_m $} \Comment{Next posterior mean} \label{A:Line:x}
\State{$\Sigmat_m = \Sigmat_{m-1} - \Sigmat_0\Amat\svec_m\left(\Sigmat_0\Amat\svec_m\right)^T\big/(\svec_m^T\Amat\Sigmat_0\Amat\svec_m)$} \Comment{Next posterior covariance}
\State{$ \rvec_m =  \rvec_{m-1} - \alpha_m \Amat\Sigmat_0\Amat\svec_m$}
\Comment{Next residual}
\State{$\beta_m = \left( \rvec_m^T \rvec_m\right)\big/\left( \rvec_{m-1}^T\rvec_{m-1}\right)$}
\State{$ \svec_{m+1} =  \rvec_m+\beta_m  \svec_m$} \Comment{Next $\Amat\Sigmat_0\Amat$-orthogonal search direction} \label{A:Line:s}
\EndWhile
\State \textbf{Output:} $\mu_m = \N(\xvec_m,\Sigmat_m)$\Comment{Final posterior}
\end{algorithmic}
\end{algorithm}

\begin{algorithm}
\caption{Conjugate Gradient Method (CG) \cite[Section 3]{Hestenes}}
\label{A:CG}
\begin{algorithmic}[1]
  \State{\textbf{Input:} spd $\Amat\in\Rnn$, $\bvec\in\Rn$, $\xvec_0\in\Rn$}
  \State{$  \rvec_0 =  \bvec- \Amat\xvec_0$}   \Comment{Initial residual}
  \State{$\wvec_1 = \rvec_0$} \Comment{Initial search direction}
  \State{$m = 0$} \Comment{Initial iteration count}
  \While{Not converged}
  \State{$m = m + 1$} \Comment{Increment iteration count}
\State{$\gamma_m = (\rvec_{m-1}^T\rvec_{m-1})\big/ (\wvec_m^T\Amat\wvec_m)$}
\Comment{Next step size}
\State{$ \xvec_m =  \xvec_{m-1} +   \gamma_m\wvec_m $}\Comment{Next iterate}
\State{$\rvec_m = \rvec_{m-1} - \gamma_m\Amat\wvec_m$}\Comment{Next residual}
\State{$\delta_m = (\rvec_m^T\rvec_m)\big/ (\rvec_{m-1}^T\rvec_{m-1})$}
\State{$ \wvec_{m+1} =  \rvec_m+  \delta_m\wvec_m$} 
\Comment{Next search direction}
\EndWhile
\State \textbf{Output:} $\xvec_m$ \Comment{Final approximation for $\xvec_*$}
\end{algorithmic}
\end{algorithm}

\subsection{The ideal Krylov Prior}
\label{S:KrylovBackground}
After defining the 
Krylov space of maximal dimension (Definition~\ref{e_kmax}),
we review the ideal but impractical Krylov prior 
(Definition~\ref{D:KrylovPrior}), and discuss its construction (Lemma~\ref{L:KrylovCG}) and properties (Theorem~\ref{T:KrylovPrior}).

\begin{definition}\label{e_kmax}
The Krylov space of \textit{maximal dimension} for \aref{A:CG} is
\begin{equation*}
  \Kcal_\kry(\Amat,\rvec_0) \equiv \spn\{\rvec_0,\Amat\rvec_0, \ldots, \Amat^{\kry-1}\rvec_0\}.
\end{equation*}
Here $\kry\leq n$ represents the \textit{grade} of $\rvec_0$ with respect to $\Amat\in\mathbb{R}^{n\times n}$ \cite[Definition 4.2.1]{Liesen}, or the \textit{invariance index} for $(\Amat, \rvec_0)$ \cite[Section 2]{Berljafa},
which is the minimum value where
\begin{align*}
\Kcal_\kry(\Amat,\rvec_0) = \Kcal_{\kry+i} (\Amat,\rvec_0), \qquad i\geq 1.
\end{align*}
\end{definition}

The Krylov prior is a Gaussian distribution whose covariance 
is constructed
from a basis for the maximal dimensional CG Krylov space.

\begin{definition}[{\cite[Definition 3.1]{RICO21}}]
\label{D:KrylovPrior}
The \emph{ideal Krylov prior} for $\Amat\xvec_*=\bvec$ 
is $\eta_0 \equiv \N(\xvec_0,\Gammat_0)$ with
symmetric positive semi-definite covariance
\begin{equation}
\label{Eq:KrylovPrior}
\Gammat_0 \equiv \Vmat \Phimat \Vmat^T \in\Rnn.
\end{equation}
The columns of $ \Vmat \equiv \begin{bmatrix} \vvec_1 & \vvec_2 & \cdots & \vvec_\kry \end{bmatrix} \in\Real^{n\times\kry}$ are an $\Amat$-orthonormal basis 
for $\Kcal_\kry(\Amat,\rvec_0)$, which means that
\begin{equation*}
  \Vmat^T\Amat\Vmat = \Imat_g \quad \text{and} \quad \spn\{\vvec_1, \ldots, \vvec_i\} = \Kcal_i(\Amat,\rvec_0), \quad 1\leq i \leq \kry.
\end{equation*}
The diagonal matrix $\Phimat \equiv \diag\begin{pmatrix} \phi_1 &\cdots & \phi_\kry\end{pmatrix}\in\Real^{\kry\times \kry}$ 
has diagonal elements
\begin{equation}
\label{Eq:PhiDef}
  \phi_i = (\vvec_i^T\rvec_0)^2 , \qquad 1\leq i \leq \kry.
\end{equation}
\end{definition}

\begin{remark}
The Krylov prior covariance satisfies the requirement of \aref{A:BayesCG} that $\xvec_*-\xvec_0 \in \range(\Gammat_0)$. This follows from \cite[Section 5.6]{Liesen},
\begin{align*}
\xvec_* \in \xvec_0 + \Kcal_\kry(\Amat,\rvec_0)=\range(\Gammat_0).
\end{align*}
If the maximal Krylov space $\Kcal_\kry(\Amat,\rvec_0)$ has dimension
$\kry < n$, then $\Gammat_0$ is singular.
\end{remark}

\begin{lemma}[{\cite[Remark SM2.1]{RICO21}}]
  \label{L:KrylovCG}
  The Krylov prior $\Gammat_0$ can be constructed from quantities computed by CG (\aref{A:CG}),
  \begin{equation*}
    \vvec_i \equiv \wvec_i / (\wvec_i^T\Amat\wvec_i), \quad \text{and} \quad \phi_i \equiv \gamma_i\|\rvec_{i-1}\|_2^2, \qquad 1\leq i \leq \kry.
  \end{equation*}
\end{lemma}

The posterior distributions from BayesCG under the Krylov prior depend on submatrices of $\Vmat$ and $\Phimat$,
\begin{equation}
\begin{split}\label{Eq:VPhiSubmatrices}
  \Vmat_{i:j} &\equiv \begin{bmatrix} \vvec_i &  \cdots & \vvec_j \end{bmatrix}  \\
  \Phimat_{i:j} &\equiv \diag\begin{pmatrix} \phi_i &  \cdots & \phi_j \end{pmatrix}, \qquad 1 \leq i \leq j \leq \kry,
  \end{split}
\end{equation}
where $\Vmat_{1:g}=\Vmat$, $\Phimat_{1:g}=\Phimat$, and $\Vmat_{j+1:j} =
\Phimat_{j+1:j}=\zerovec$, $1\leq j\leq n$.

Under suitable assumptions, BayesCG (Algorithm~\ref{A:BayesCG}) produces
the same iterates as CG (Algorithm~\ref{A:CG}).

\begin{theorem}[{\cite[Theorem 3.3]{RICO21}}]
  \label{T:KrylovPrior}
  Let $\xvec_0$ be the starting vector for  CG (Algorithm~\ref{A:CG}). 
Then BayesCG (Algorithm~\ref{A:BayesCG}) under the Krylov prior
$\eta_0 \equiv \N(\xvec_0,\Gammat_0)$ produces
Krylov posteriors $\eta_m\equiv \mathcal{N}(\xvec_m,\Gammat_m)$ 
whose mean vectors 
\begin{align*}
\xvec_m=\xvec_0+\Vmat_{1:m}\Vmat_{1:m}^T\rvec_0,
\qquad 1\leq m\leq g,
\end{align*}
are identical to the iterates in CG (Algorithm~\ref{A:CG}), and whose covariance matrices 
  \begin{align}
\label{Eq:GammaN}
\Gammat_m = \Vmat_{m+1:\kry}\Phimat_{m+1:\kry}\Vmat_{m+1:\kry}^T, \qquad 1\leq m < g,
  \end{align}
satisfy
\begin{equation}
  \label{Eq:PostTrace}
 \trace(\Amat\Gammat_m) = \trace(\Phimat_{m+1:\kry}) = \|\xvec_*-\xvec_m\|_\Amat^2.
\end{equation}
\end{theorem}

Explicit construction of the ideal Krylov prior, followed by explicit computation
of the Krylov posteriors in Algorithm~\ref{A:BayesCG}
is impractical, because it is more expensive than solving
the linear system (\ref{Eq:Axb}) in the first place. That is the reason
for introducing practical, approximate Krylov posteriors.

\subsection{Practical Krylov posteriors}\label{S:approxKrylov}
 We dispense with the explicit computation of the Krylov prior, and instead
compute  a low-rank approximation 
of the final posterior (Definition~\ref{D:KrylovApprox})
by running $d$ additional iterations. The corresponding CG-based implementation of
BayesCG under approximate Krylov posteriors 
is relegated to \aref{A:BayesCGWithoutBayesCG} in \appref{S:Algs}.

\begin{definition}[{\cite[Definition 3.4]{RICO21}}]
  \label{D:KrylovApprox}
Given the  Krylov prior $\eta_0\equiv\mathcal{N}(\xvec_0,\Gammat_0)$
with posteriors $\eta_m\equiv\mathcal{N}(\xvec_m,\Gammat_m)$, pick some $d\geq 1$. 
The rank-$d$ approximation of $\eta_m$
 is a Gaussian distribution 
 $\widehat{\eta}_m \equiv \N(\xvec_m,\widehat{\Gammat}_m)$ with
 the same mean $\xvec_m$ as~$\eta_m$, and a rank-$d$ covariance
  \begin{equation*}
    \widehat{\Gammat}_m \equiv \Vmat_{m+1:m+d} \Phimat_{m+1:m+d} \,\Vmat_{m+1:m+d}^T, \qquad 1\leq m< g-d,
  \end{equation*}
that consists of the leading $d$ columns of $\Vmat_{m+1:g}$.
\end{definition}

In contrast to the full Krylov posteriors, which reproduce the error as 
in~(\ref{Eq:PostTrace}), approximate Krylov posteriors underestimate the error \cite[Section 3.4]{RICO21},
\begin{equation}
  \label{Eq:StrakosTichy}
  \trace(\Amat\widehat\Gammat_m) = \trace(\Phimat_{m+1:m+d}) =  \|\xvec_*-\xvec_m\|_\Amat^2 - \|\xvec_*-\xvec_{m+d}\|_\Amat^2,
\end{equation}
where $\|\xvec_*-\xvec_{m+d}\|_\Amat^2$ is the error after $m+d$ iterations of CG. The error underestimate $\trace(\Amat\widehat\Gammat_m)$ is equal to \cite[Equation(4.9)]{StrakosTichy}, and it is more accurate when convergence is fast. Fast convergence makes $\trace(\Amat\widehat\Gammat_m)$ a more accurate estimate because fast convergence implies that  $\|\xvec_*-\xvec_{m+d}\|_\Amat^2\ll\|\xvec_*-\xvec_m\|_\Amat^2$, and this, along with \eref{Eq:StrakosTichy}, implies that $\trace(\Amat\widehat\Gammat_m)\approx\|\xvec_*-\xvec_m\|_\Amat^2$ \cite[Section 4]{StrakosTichy}.

\section{Approximate Krylov Posteriors}
\label{S:KrylovProperties}
We determine the error in approximate Krylov posteriors 
(Section~\ref{S:KrylovApprox}),
and interpret the Krylov prior as an empirical Bayesian method (Section~\ref{S:KrylovStat}).

\subsection{Error in Approximate Krylov Posteriors}
\label{S:KrylovApprox}
We review the $p$-Wasserstein distance
(Definition~\ref{D:Wasserstein}), extend the 2-Wasserstein distance 
to the $\Amat$-Wasserstein distance weighted by a symmetric positive definite 
matrix $\Amat$ (Theorem~\ref{L:WassersteinNorm}), 
and derive the $\Amat$-Wasserstein distance between approximate 
and full Krylov posteriors
(Theorem~\ref{T:Wasserstein}).

The $p$-Wasserstein distance is a metric on the set of probability distributions. 

\begin{definition}[{\cite[Definition 2.1]{KLMU20}, \cite[Definition 6.1]{Villani09}}]
  \label{D:Wasserstein}
The $p$-Wasser\--stein distance between probability 
  distributions $\mu$ and $\nu$ on $\Rn$ is
\begin{equation}
  \label{Eq:WassersteinInt}
  W_p(\mu,\nu) \equiv \left(\inf_{\pi\in\Pi(\mu,\nu)}\int_{\Rn\times\Rn}\|\Mrv-\Nrv\|_2^p \ d\pi(\Mrv,\Nrv)\right)^{1/p}, \quad p\geq 1,
\end{equation}
where $\Pi(\mu,\nu)$ is the set of couplings between $\mu$ and $\nu$, 
that is, the set of probability distributions on $\Rn\times\Rn$ that have 
$\mu$ and $\nu$ as marginal distributions.
\end{definition}

In the special case $p=2$, the $2$-Wasserstein or \textit{Fr\'{e}chet distance} between two Gaussian distributions admits an explicit expression.

\begin{lemma}[{\cite[Theorem 2.1]{Gelbrich90}}]
  \label{L:2Wasserstein}
  The 2-Wasserstein distance between Gaussian distributions
  $\mu \equiv \N(\xvec_\mu,\Sigmat_\mu)$ and $\nu\equiv\N(\xvec_\nu,\Sigmat_\nu)$
  on $\Rn$ is
  \begin{equation*}
    \left(W_2(\mu,\nu)\right)^2 = \|\xvec_\mu-\xvec_\nu\|_2^2 + 
    \trace\left(\Sigmat_\mu+\Sigmat_\nu - 2\left(\Sigmat_\mu^{1/2}\Sigmat_\nu\Sigmat_\mu^{1/2}\right)^{1/2}\right).
  \end{equation*}
\end{lemma}

We generalize the 2-Wasserstein distance 
to the $\Amat$-Wasserstein distance weighted by a symmetric positive definite  matrix~$\Amat$.

 \begin{definition}
  \label{D:BWasserstein}
  The two-norm of $\xvec\in\Rn$ weighted by a symmetric positive definite  $\Amat\in\Rdd$ is
\begin{equation}
    \label{Eq:BNorm}
    \|\xvec\|_\Amat \equiv \|\Amat^{1/2}\xvec\|_2.
  \end{equation}
The $\Amat$-Wasserstein distance between Gaussian distributions
  $\mu \equiv \N(\xvec_\mu,\Sigmat_\mu)$ and $\nu\equiv\N(\xvec_\nu,\Sigmat_\nu)$
  on $\Rn$ is
   \begin{align}
    \label{Eq:L:WassersteinNorm:Exp}
    W_\Amat (\mu,\nu) \equiv \left(\inf_{\pi\in\Pi(\mu,\nu)}\int_{\Rn\times\Rn}\|\Mrv-\Nrv\|_\Amat^2 \ d\pi(\Mrv,\Nrv)\right)^{1/2},
  \end{align}
where $\Pi(\mu,\nu)$ is the set of couplings between $\mu$ and $\nu$.
\end{definition}

We derive an explicit expression for the $\Amat$-Wasserstein distance
analogous to the one for the 2-Wasserstein distance in \lref{L:2Wasserstein}.

\begin{theorem}
  \label{L:WassersteinNorm}
For symmetric positive definite $\Amat\in\Rdd$,
the $\Amat$-Wasserstein distance between Gaussian distributions 
$\mu \equiv \N(\xvec_\mu,\Sigmat_\mu)$ and 
$\nu \equiv \N(\xvec_\nu,\Sigmat_\nu)$ on $\Rn$ is
  \begin{align}
    \label{Eq:WassersteinNormA}
    \left(W_\Amat(\mu,\nu)\right)^2 &= \|\xvec_\mu-\xvec_\nu\|_\Amat^2 + \trace(\widetilde{\Sigmat}_\mu)+\trace(\widetilde{\Sigmat}_\nu) \notag \\
    & \qquad -2
    \trace\left((\widetilde{\Sigmat}_\mu^{1/2}\widetilde{\Sigmat}_\nu
 \widetilde{\Sigmat}_\mu^{1/2})^{1/2}\right),
  \end{align}
  where $\widetilde{\Sigmat}_{\mu}\equiv\Amat^{1/2}\Sigmat_\mu \Amat^{1/2}$
  and $\widetilde{\Sigmat}_{\nu}\equiv\Amat^{1/2}\Sigmat_\nu \Amat^{1/2}$.
\end{theorem}

\begin{proof}
First express the $\Amat$-Wasserstein distance as a 2-Wasserstein distance, by
substituting \eref{Eq:BNorm} into \eref{Eq:L:WassersteinNorm:Exp},
  \begin{equation}
    \label{Eq:BWassersteinExp}
\left( W_\Amat(\mu,\nu)\right)^2 = \inf_{\pi\in\Pi(\mu,\nu)}\int_{\Rn\times\Rn}\|\Amat^{1/2}\Mrv - \Amat^{1/2}\Nrv\|_2^2 \ d\pi(\Mrv,\Nrv).
  \end{equation}
\lref{L:GaussStability} in Appendix~\ref{S:Aux} implies that
 $\Amat^{1/2}\Mrv$ and $\Amat^{1/2}\Nrv$ are 
again Gaussian random variables with respective means and covariances
  \begin{equation*}
    \tilde{\mu} \equiv\mathcal{N}(\Amat^{1/2}\xvec_\mu,\,
    \underbrace{\Amat^{1/2}\Sigmat_\mu\Amat^{1/2}}_{\widetilde{\Sigmat}_{\mu}}),\qquad 
    \tilde{\nu} \equiv \mathcal{N}(\Amat^{1/2}\xvec_\nu,\,\underbrace{\Amat^{1/2}\Sigmat_\nu\Amat^{1/2}}_{\widetilde{\Sigmat}_{\nu}}).
  \end{equation*}
  Thus \eref{Eq:BWassersteinExp} is equal to the 2-Wasserstein distance
  \begin{align}
    \label{Eq:B2Wasserstein}
\left(W_\Amat(\mu,\nu)\right)^2 = \inf_{\pi\in\Pi(\tilde{\mu},\tilde{\nu})}\int_{\Rn\times\Rn}\|\widetilde{\Mrv} - \widetilde{\Nrv}\|_2^2 \ d\pi(\widetilde{\Mrv},\widetilde{\Nrv})
    = \left(W_2(\tilde{\mu},\tilde{\nu})\right)^2.
  \end{align}
 At last, apply \lref{L:2Wasserstein} and the linearity of the trace.\qed
  \end{proof}

We are ready to derive the $\Amat$-Wasserstein distance between approximate and full Krylov posteriors.  

%The following theorem relies on a relationship between orthogonal and 
%generalized orthogonal matrices which we have placed in 

\begin{theorem}
  \label{T:Wasserstein}
  Let $\eta_m \equiv \N(\xvec_m,\Gammat_m)$ 
  be a Krylov posterior from Theorem~\ref{T:KrylovPrior}, and
 for some $d\geq 1$ let $\widehat{\eta}_m \equiv \N(\xvec_m,\widehat{\Gammat}_m)$ be a rank-$d$ approximation from Definition~\ref{D:KrylovApprox}.
The $\Amat$-Wasserstein distance between $\eta_m$ and $\widehat{\eta}_m$ is
  \begin{equation}
    W_\Amat(\eta_m,\widehat{\eta}_m) = \left( \sum_{i=m+d+1}^\kry \phi_i\right)^{1/2}.
  \end{equation}
\end{theorem}

\begin{proof}
We factor the covariances into square factors, 
to obtain an eigenvalue decomposition for the congruence
transformations of the covariances in \eref{Eq:WassersteinNormA}.

Expand the column dimension of $\Vmat_{m+1:\kry}$ from $g-m$ to $n$ by adding
an $\Amat$-orthogonal complement
$\Vmat_m^\perp\in\Real^{n\times (\dm -\kry+m)}$ to create
 an $\Amat$-orthogonal matrix
  \begin{align*}
    \label{Eq:BlockFull}
    \widetilde{\Vmat} &\equiv \begin{bmatrix} \Vmat_{m+1:\kry} & \Vmat_m^\perp \end{bmatrix}\in\Rnn
\end{align*}    
with $\widetilde{\Vmat}^T\Amat\widetilde{\Vmat}=\Imat_n$.
Analogously
expand the dimension of the diagonal matrices by padding with trailing zeros,
\begin{align*}   
    \widetilde{\Phimat}_{m+1:\kry} &\equiv\diag\begin{pmatrix}\phi_{m+1}&\cdots&\phi_\kry &
\vzero_{1\times (n-g+m)}\end{pmatrix} \in\Rnn,  \notag \\
    \widetilde{\Phimat}_{m+1:m+d} &\equiv  \diag\begin{pmatrix}\phi_{m+1} &\cdots &\phi_{m+d} & \vzero_{1\times (n-d)}\end{pmatrix} \in\Rnn.
  \end{align*}
  Factor the covariances in terms of the above square matrices,
  \begin{equation*}
    \label{Eq:BlockCov}
    \Gammat_m = \widetilde{\Vmat}\widetilde{\Phimat}_{m+1:\kry}\widetilde{\Vmat}^T \quad \text{and} \quad \widehat{\Gammat}_m = \widetilde{\Vmat}\widetilde{\Phimat}_{m+1:m+d}\widetilde{\Vmat}^T.
  \end{equation*}
  Substitute the factorizations into \eref{Eq:WassersteinNormA}, and compute the $\Amat$-Wasserstein distance between $\eta_m$ and~$\widehat{\eta}_m$ as
  \begin{equation}
    \label{Eq:WassersteinApprox}
   \left( W_\Amat(\eta_m,\widehat{\eta}_m) \right)^2= \trace(\Gmat) + \trace(\Jmat)
    -2\trace\left((\Gmat^{1/2}\,\Jmat \,\Gmat^{1/2})^{1/2}\right),
  \end{equation}
where the congruence transformations of 
  $\Gammat_m$ and $\widehat{\Gammat}_m$
  are again Hermitian,
  \begin{align*}
    \label{Eq:GJ}
    \Gmat &\equiv\Amat^{1/2}\underbrace{\widetilde{\Vmat}\widetilde{\Phimat}_{m+1:\kry}\widetilde{\Vmat}^T}_{\Gammat_m}\Amat^{1/2}=\Umat\widetilde{\Phimat}_{m+1:\kry}\Umat^T,\qquad \Umat\equiv \Amat^{1/2}\widetilde{\Vmat}\\
    \Jmat &\equiv \Amat^{1/2}\underbrace{\widetilde{\Vmat}\widetilde{\Phimat}_{m+1:m+d}\widetilde{\Vmat}^T}_{\widehat{\Gammat}_m}\Amat^{1/2}=\Umat\widetilde{\Phimat}_{m+1:d}\Umat^T.
  \end{align*}
\lref{L:BOrth} implies that $\Umat$ is an orthogonal matrix, so that the second factorizations of $\Gmat$ and $\Jmat$  represent eigenvalue decompositions. Commutativity
of the trace implies
\begin{align*}
\trace(\Gmat)&=\trace(\widetilde{\Phimat}_{m+1:\kry})=\sum_{i=m+1}^g{\phi_{i}}\\
\trace(\Jmat)&=\trace(\widetilde{\Phimat}_{m+1:m+d})=\sum_{i=m+1}^{m+d}{\phi_{i}}.
\end{align*}
Since $\Gmat$ and $\Jmat$ 
have the same eigenvector matrix, they commute, and so do diagonal matrices,
\begin{align*}
\Gmat^{1/2}\Jmat\Gmat^{1/2}&=
\Umat\widetilde{\Phimat}_{m+1:g}
\widetilde{\Phimat}_{m+1:m+d}\Umat^T\\
&=\Umat\diag\begin{pmatrix}\phi_{m+1}^2& \cdots \phi_{m+d}^2 & \vzero_{1\times (n-d)}\end{pmatrix}\Umat^T
   \end{align*}
where the last equality follows from the fact that $\widetilde{\Phimat}_{m+1:g}$ and 
$\widetilde{\Phimat}_{m+1:m+d}$
share the leading $d$ diagonal elements. Thus
\begin{align*}
   \trace\left((\Gmat^{1/2}\,\Jmat \,\Gmat^{1/2})^{1/2}\right)=
   \sum_{i=m+1}^{m+d} \phi_i.
\end{align*}
Substituting the above expressions
into \eref{Eq:WassersteinApprox} gives
  \begin{align*}
\left( W_\Amat(\eta_m,\widehat{\eta}_m)\right)^2 = \sum_{i=m+1}^{\kry}\phi_i + \sum_{i=m+1}^{m+d} \phi_i- 2 \sum_{i=m+1}^{m+d} \phi_i=
\sum_{i=m+d+1}^{g}{\phi_i}.
  \end{align*}\qed
  \end{proof}

\tref{T:Wasserstein} implies that the $\Amat$-Wasserstein distance between 
approximate and full Krylov posteriors is the sum of the CG steps sizes skipped by the approximate posterior, and this, as seen in \eref{Eq:StrakosTichy} and \cite[Equation (4.4)]{StrakosTichy}, is equal to the distance between the error estimate $\trace(\Amat\widehat\Gammat_m)$ and the true error $\|\xvec_*-\xvec_m\|_\Amat^2$. As a consequence, the approximation error decreases as the convergence of the posterior mean accelerates, or the rank~$d$ of the approximation increases.

\begin{remark}
The distance in
\tref{T:Wasserstein} is a special case of the 2-Wasserstein distance between two distributions whose covariance matrices commute \cite[Corollary 2.4]{KLMU20}.

To see this, consider the $\Amat$-Wasserstein distance between $\eta_m$ and $\widehat{\eta}_m$ from Theorem~\ref{T:Wasserstein}, and the 2-Wasserstein distance between $\nu_m \equiv \N(\xvec_m,\Amat^{1/2}\Gammat\Amat^{1/2})$ and $\widehat\nu_m \equiv  \N(\xvec_m,\Amat^{1/2}\widehat\Gammat\Amat^{1/2})$.
Then \eref{Eq:B2Wasserstein} implies that the $\Amat$-Wasserstein
distance is equal to the 2-Wassterstein distance of a congruence transformation,
\begin{equation*}
  W_\Amat(\eta_m,\widehat\eta_m) = W_2(\nu_m,\widehat\nu_m).
\end{equation*}
The covariance matrices $\Amat^{1/2}\Gammat_m\Amat^{1/2}$ and 
$\Amat^{1/2}\widehat{\Gammat}_m\Amat^{1/2}$ associated with the 2-Wasserstein distance 
commute because they are both diagonalized by the same orthogonal matrix $\Amat^{1/2}\widetilde\Vmat$.
\end{remark}
%We chose to prove \tref{T:Wasserstein} from first principles, rather than invoking 
%\cite[Corollary 2.4]{KLMU20}, because it would not have shortened this self-contained proof.

\subsection{Probabilistic  Interpretation of the Krylov Prior}
\label{S:KrylovStat}

We interpret the Krylov prior as an `empirical Bayesian procedure' 
(Theorem~\ref{T:KrylovEmp}), and elucidate the connection between the
random variables and the deterministic solution (Remark~\ref{R:DetRandom}).

An \textit{empirical Bayesian procedure} estimates the prior from data \cite[Section 4.5]{Berger1985}.
Our `data' are the pairs 
of normalized search directions $\vvec_i$ and step sizes~$\phi_i$,
$1\leq i\leq m+d$, from $m+d$ iterations of CG.
In contrast, the usual data for BayesCG are the inner products $\vvec_i^T \bvec$, $1\leq i \leq m$.
However, if we augment the usual data with the search directions, which is natural due to their dependence on $\xvec_*$, then $\phi_i$ is just a function of the data.

From these data we construct a prior in an empirical Bayesian fashion,
starting with a random variable 
\begin{equation*}
  \Xrv = \xvec_0 + \sum_{i=1}^{m+d}  \sqrt{\phi_i} \vvec_i \Qrv_i\in\Rn,
\end{equation*}
where $\Qrv_i \sim \N(0,1)$ are independent and identically distributed scalar Gaussian random variables, $1\leq i\leq m+d$. Due to the independence of the $\Qrv_i$, the above sum is the matrix vector product
\begin{equation}
  \label{Eq:RandEmp}
  \Xrv = \xvec_0 + \Vmat_{1:m+d}\, \Phimat_{1:m+d}^{1/2}\, \Qrv
\end{equation} 
where $\Qrv \sim \mathcal{N}(\zerovec, \Imat_{m+d})$ is a vector-valued Gaussian random variable. 

The distribution of $\Xrv$ is the \textit{empirical prior}, while the distribution of $\Xrv$ conditioned on the random variable $\Yrv \equiv \Vmat_{1:m}^T\Amat\Xrv$ taking the value $\Vmat_{1:m}^T\bvec$ is the \textit{empirical posterior}. We relate
these distributions to the Krylov prior.

\begin{theorem}
\label{T:KrylovEmp}
Under the assumptions of Theorem~\ref{T:KrylovPrior}, the
random variable $\Xrv$ in \eref{Eq:RandEmp} is 
distributed according to the empirical prior
  \begin{equation*}
  \N \left(\xvec_0, \Vmat_{1:m+d}\Phimat_{1:m+d} \Vmat_{1:m+d}^T\right),
\end{equation*}
which is the rank-$(m+d)$ approximation of the Krylov prior $\Gammat_0$. The variable $\Xrv$ conditioned on $\Yrv \equiv \Vmat_{1:m}^T\Amat\Xrv$ taking the value $\Vmat_{1:m}^T\bvec$ is distributed according to the empirical posterior
\begin{equation*}
  \N \left(\xvec_m, \Vmat_{m+1:m+d}\Phimat_{m+1:m+d} \Vmat_{m+1:m+d}^T\right)=\N\left(\xvec_m,\widehat{\Gammat}_m\right),
\end{equation*}
which, in turn, is the rank-$d$ approximation of the Krylov posterior.
\end{theorem}

\begin{proof}
As in the proof of \tref{T:BayesCG} in \cite[Proof of Proposition 1]{BCG:Supp}, we exploit the stability and conjugacy of Gaussian distributions in Lemmas \ref{L:GaussStability} and~\ref{L:GaussConjugacy} in Appendix~\ref{S:Aux}.
  
  \paragraph{Prior.}
\lref{L:GaussStability} implies that $\Xrv$ in \eref{Eq:RandEmp} is a Gaussian random variable with mean and covariance
  \begin{equation}\label{e_33a}
    \Xrv \sim \mathcal{N}
    \left(\xvec_0, \Vmat_{1:m+d}\Phimat_{1:m+d} \Vmat_{1:m+d}^T\right).
  \end{equation}
Thus, the approximate Krylov prior is an empirical Bayesian prior.
  
 \paragraph{Posterior.}
 From (\ref{e_33a}) follows that 
$\Xrv$ and $\Yrv\equiv\Vmat_{1:m}^T\Amat X$ have the joint distribution
  \begin{equation}
    \label{Eq:JointDist}
    \begin{bmatrix} \Xrv \\ \Yrv \end{bmatrix} \sim \N \left(\begin{bmatrix} \xvec_0 \\ \Exp[\Yrv] \end{bmatrix},\begin{bmatrix} \Vmat_{1:m+d}\Phimat_{1:m+d}\Vmat_{1:m+d}^T & \Cov(\Xrv,\Yrv) \\ \Cov(\Xrv,\Yrv)^T & \Cov(\Yrv,\Yrv) \end{bmatrix}\right)
  \end{equation}
and that $\Exp[\Yrv] = \Vmat_{1:m}^T\Amat\xvec_0$.
This, together with the linearity 
of the expectation and the $\Amat$-orthonormality of $\Vmat$ implies 
  \begin{align*}
    \Cov(\Yrv,\Yrv) &=
 \Exp\left[(\Yrv-\Exp[\Yrv])(\Yrv-\Exp[\Yrv])^T\right]\\
& =\Vmat_{1:m}^T\Amat\,\Exp\left[(\Xrv -\xvec_0)
 (\Xrv - \xvec_0)^T\right]\Amat\Vmat_{1:m}\\  
 &=\Vmat_{1:m}^T\Amat\>\left(\Vmat_{1:m+d}\Phimat_{1:m+d}\Vmat_{1:m+d}^T\right)\>\Amat\Vmat_{1:m}\\
 &= \begin{bmatrix}\Imat_m&\vzero\end{bmatrix}
 \Phimat_{1:m+d}\begin{bmatrix}\Imat_m\\ \vzero\end{bmatrix}=
 \Phimat_{1:m}.
  \end{align*}
 Analogously,
  \begin{align*}
    \Cov(\Xrv,\Yrv) &= \Exp[(\Xrv-\xvec_0)(\Yrv - \Exp[\Yrv])^T]
    = \Exp[(\Xrv-\xvec_0)(\Yrv - \Vmat_{1:m}^T\Amat\xvec_0)^T]  \\
 &=  \Exp[(\Xrv-\xvec_0)(\Xrv-\xvec_0)^T]\Amat\Vmat_{1:m}
=    \Vmat_{1:m+d}\Phimat_{1:m+d}\Vmat_{1:m+d}^T\Amat\Vmat_{1:m}\\
&=\Vmat_{1:m+d}\Phimat_{1:m+d}\begin{bmatrix}\Imat_m& \vzero\end{bmatrix}
=\Vmat_{1:m}\Phimat_{1:m}.
  \end{align*}
From \cite[Theorem 6.20]{Stuart:BayesInverse} follows 
 the expression for the posterior mean,
\begin{align*}
\xvec_m &= 
\xvec_0+\Cov(\Xrv,\Yrv)\Cov(\Yrv,\Yrv)^{-1}\left(\Vmat_{1:m}^T\bvec-\Vmat_{1:m}^T\Amat\xvec_0\right)\\
&=\xvec_0 + \Vmat_{1:m}\Phimat_{1:m}\Phimat_{1:m}^{-1}\Vmat_{1:m}^T\rvec_0 = \xvec_0+\Vmat_{1:m}\Vmat_{1:m}^T\rvec_0,
  \end{align*}
  and for the posterior covariance 
  \begin{align*}
\widehat\Gammat_m &= \Vmat_{1:m+d}\Phimat_{1:m+d}\Vmat_{1:m+d}^T
-\Cov(\Xrv,\Yrv)\Cov(\Yrv,\Yrv)^{-1}\Cov(\Xrv,\Yrv)^T,
\end{align*}
where
\begin{align*}
\Cov(\Xrv,\Yrv)\Cov(\Yrv,\Yrv)^{-1}\Cov(\Xrv,\Yrv)^T&=
\Vmat_{1:m}\Phimat_{1:m}\Phimat_{1:m}^{-1}\Phimat_{1:m}\Vmat_{1:m}^T\\
&=\Vmat_{1:m}\Phimat_{1:m}\Vmat_{1:m}^T.
\end{align*}
Substituting this into $\widehat{\Gammat}_m$ gives the expression for the posterior covariance
\begin{align*}
\widehat{\Gammat}_m
&=\Vmat_{1:m+d}\Phimat_{1:m+d}\Vmat_{1:m+d}^T-\Vmat_{1:m}\Phimat_{1:m}\Vmat_{1:m}^T\\
&= \Vmat_{m+1:m+d}\Phimat_{m+1:m+d}\Vmat_{m+1:m+d}^T.
  \end{align*}
  Thus, the posterior mean $\xvec_m$ is equal to the one in 
Theorem~\ref{T:KrylovPrior}, and the posterior covariance
$\widehat{\Gammat}_m$ 
is equal to the rank-$d$ approximate Krylov posterior in Definition~\ref{D:KrylovApprox}.\qed
\end{proof}

\begin{remark}
\label{R:DetRandom}
The random variable $\Xrv$ in Theorem~\ref{T:KrylovEmp}
is a surrogate for the unknown solution $\xvec_*$. The solution $\xvec_*$ is a deterministic quantity, but prior to  solving the linear system (\ref{Eq:Axb}),
we are uncertain of $\xvec_*$, and the prior models this uncertainty. 

During the course of the BayesCG iterations, we acquire information about~$\xvec_*$, and the posterior distributions $\mu_m$, $1\leq m \leq n$ incorporate our increasing knowledge and, consequently, our
diminishing uncertainty. 
\end{remark}

\section{Calibration of BayesCG Under the Krylov Prior}
\label{S:Calibration}

We review the notion of calibration for probabilistic solvers,
and show that this notion does not apply to BayesCG under the Krylov prior (\sref{S:CalibrationDef}). 
Then we relax this notion and analyze BayesCG with two
test statistics that are necessary but not sufficient for calibration: the $Z$-statistic (\sref{S:UQZ}) and the $S$-statistic (\sref{S:UQS}).

\subsection{Calibration}
\label{S:CalibrationDef}
We review the definition of calibration for probabilistic linear solvers (Definition \ref{D:StrongCalibration}, Lemma~\ref{L:CalibrationAlt}),
discuss the difference between certain random variables (Remark~\ref{r_rv}),
present two illustrations (Examples \ref{Ex:Story} and~\ref{Ex:Pictures}), 
and explain why this notion
of calibration does not apply to BayesCG under the Krylov prior
(Remark~\ref{r_41}).

Informally, a probabilistic numerical solver is calibrated if its posterior 
distributions accurately model the uncertainty in the solution \cite{CGOS21,CIOR20}. 

\begin{definition}[{\cite[Definition 6]{CIOR20}}]
  \label{D:StrongCalibration}
Let $\Amat\Xrv_* = \Brv$ be a class of linear systems where $\Amat\in\Real^{n\times n}$ is symmetric positive definite,
and the random right hand sides $\Brv\in\Real^n$ are defined by random solutions 
  $\Xrv_*\sim \mu_0 \equiv \N(\xvec_0,\Sigmat_0)$.
 
  Assume that a probabilistic linear solver under the prior $\mu_0$ and applied to a system  $\Amat\Xrv_*=\Brv$  computes posteriors $\mu_m  \equiv \N(\xvec_m,\Sigmat_m)$, $1 \leq m \leq n$. 
    Let $\rank(\Sigmat_m)=p_m$, 
  and let $\Sigmat_m$ have an orthogonal eigenvector matrix
  $\Umat=\begin{bmatrix}\Umat_m&\Umat_m^{\perp}\end{bmatrix}\in\Real^{n\times n}$
  where
    $\Umat_m\in\Real^{n\times p_m}$ and 
 $\Umat_m^{\perp}\in\Real^{n \times (n-p_m)}$ satisfy
  \begin{equation*}
    \range(\Umat_m) = \range(\Sigmat_m), \qquad \quad \range(\Umat_m^{\perp}) = \ker(\Sigmat_m).
  \end{equation*}
The probabilistic solver is \textit{calibrated}
if all  posterior covariances $\Sigmat_m$  are independent of $\Brv$ and satisfy
  \begin{align} \label{Eq:StrongCalibration}
    \begin{split}
    (\Umat_m^T\Sigmat_m\Umat_m)^{-1/2}\Umat_m^T(\Xrv_*-\xvec_m) &\sim \N(\zerovec, \Imat_{p_m}),\\
    (\Umat_m^{\perp})^T(\Xrv_*-\xvec_m) &= \zerovec, \qquad 1\leq m\leq n.
    \end{split}
  \end{align}
\end{definition}

Alternatively, one can think of a probabilistic linear solver as calibrated if and only if the solutions $\Xrv_*$ are distributed according to the posteriors.

\begin{lemma}
  \label{L:CalibrationAlt}
Under the conditions of \dref{D:StrongCalibration}, a probabilistic linear solver is calibrated, if and only if
 \begin{equation*}
\Xrv_*\sim\N(\xvec_m,\Sigmat_m), \qquad 1\leq m\leq n.
\end{equation*}
\end{lemma}  

\begin{proof}
Let $\Sigmat_m=\Umat\Dmat\Umat^T$ be an eigendecomposition where the eigenvalue
matrix
$\Dmat = \diag\begin{pmatrix} \Dmat_m & \zerovec \end{pmatrix}$
is commensurately partitioned with  $\Umat$ in Definition~\ref{D:StrongCalibration}.
Multiply the first equation of \eref{Eq:StrongCalibration} on the left by $\Dmat_m = \Umat_m^T\Sigmat_m\Umat_m$,
  \begin{equation*}
    \Umat_m^T(\Xrv_*-\xvec_m) \sim \N(\zerovec, \Dmat_m),
  \end{equation*}
combine the result with the second equation in~\eref{Eq:StrongCalibration},
  \begin{equation*}
    \Umat^T(\Xrv_*-\xvec_m) \sim \N\left(\zerovec,\Dmat \right).
  \end{equation*}
and multiply by $\Umat$ on the left,
  \begin{equation*}
    (\Xrv_*-\xvec_m) \sim \N(\zerovec, \Umat\Dmat\Umat^T), \qquad 1\leq m\leq n.
  \end{equation*}
At last, substitute $\Sigmat_m = \Umat\Dmat\Umat^T$ and subtract $\xvec_m$.
  \qed
\end{proof}

Since the covariance matrix $\Sigmat_m$  is singular, its
probability density function is zero on the subspace of $\Real^n$
where the solver has eliminated the uncertainty about $\Xrv_*$. From (\ref{Eq:StrongCalibration}) follows that $\Xrv_* = \xvec_m$ in $\ker(\Sigmat_m)$.
Hence, this subspace must be $\ker(\Sigmat_m)$, and
any remaining uncertainty about $\Xrv_*$ lies in $\range(\Sigmat_m)$.

\begin{remark}\label{r_rv}
We discuss the difference between the random variable $\Xrv_*$ in 
 \dref{D:StrongCalibration} and the random variable $\Xrv$ in \tref{T:KrylovEmp}.

In the context of calibration, the random variable $\Xrv_*\sim\mu_0$ 
represents the set of \textit{all possible} solutions that are accurately modeled by the 
prior $\mu_0$. If the solver is calibrated, then \lref{L:CalibrationAlt} shows that $\Xrv_*\sim\mu_m$. Thus, solutions accurately modeled by the prior~$\mu_0$ are also accurately modeled by all posteriors~$\mu_m$.

By contrast, in the context of a deterministic linear system 
$\Amat\xvec_* = \bvec$,
the random variable $\Xrv$ represents a surrogate for the \textit{particular} solution $\xvec_*$ and can be viewed as an abbreviation for $\Xrv \mid \Xrv_* = \xvec_*$.
The prior $\mu_0$ models the uncertainty in the user's initial knowledge of $\xvec_*$, and the posteriors $\mu_m$ model the uncertainty remaining after $m$ iterations of the solver.
\end{remark}

The following two examples illustrate \dref{D:StrongCalibration}.

\begin{example}
  \label{Ex:Story}
  Suppose there are three people: Alice, Bob, and Carol.
  \begin{enumerate}
\item Alice samples $\xvec_*$ from the prior $\mu_0$ and computes 
the matrix vector product $\bvec = \Amat\xvec_*$.  
\item Bob receives $\mu_0$, $\bvec$, and $\Amat$ from Alice.
He estimates $\xvec_*$ by solving the linear system with a probabilistic solver under the prior $\mu_0$,
and then samples $\yvec$ from a posterior $\mu_m$. 
\item Carol receives $\mu_m$, $\xvec_*$ and $\yvec$, but she is not told which vector is $\xvec_*$ and which is $\yvec$. Carol then attempts to determine which one of $\xvec_*$ or $\yvec$ is the sample from $\mu_m$. If Carol cannot 
distinguish between $\xvec_*$ and $\yvec$, then the solver is calibrated.
   \end{enumerate}
\end{example}  

\begin{figure}
  \centering
  \includegraphics[scale = .35]{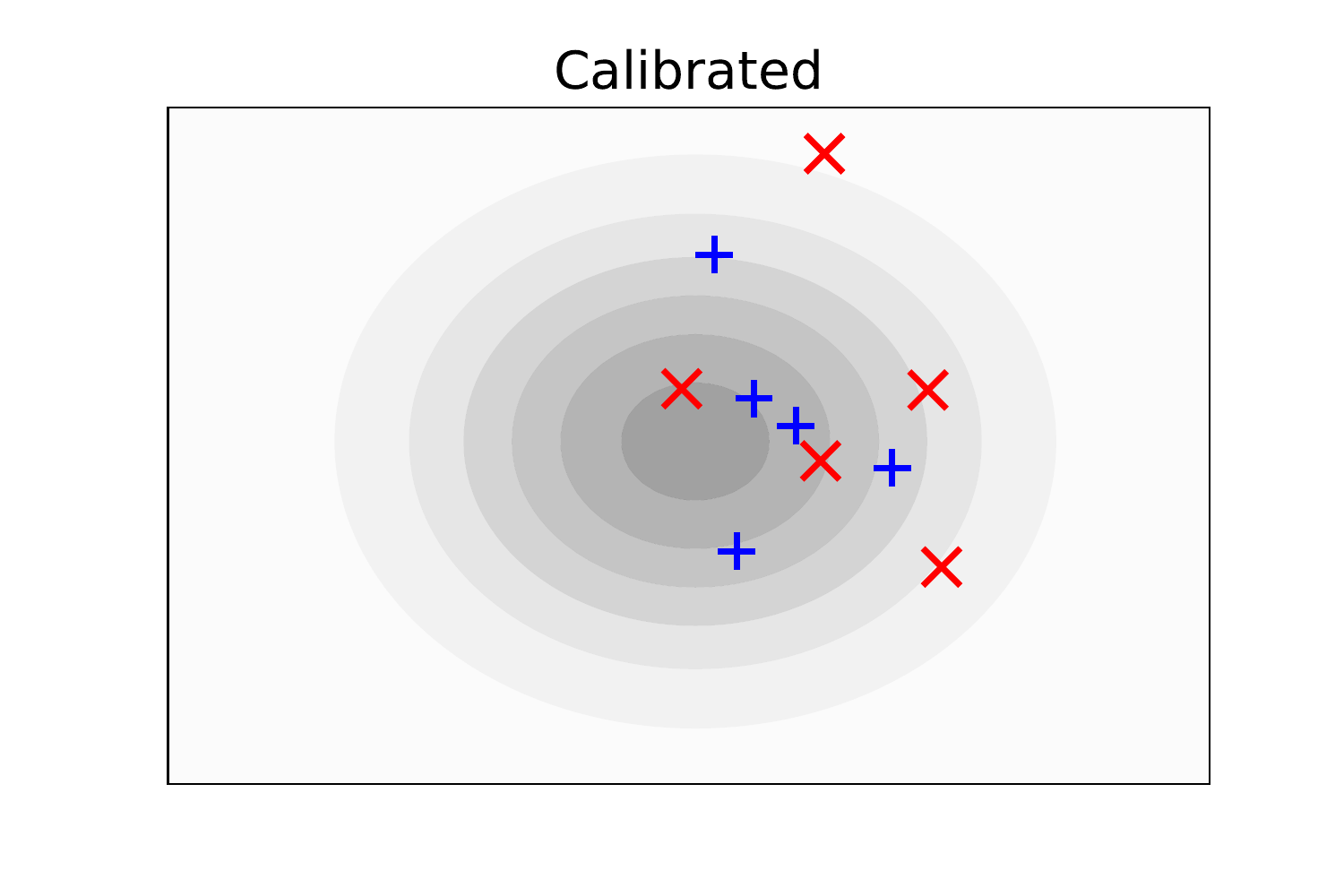} \\
  \includegraphics[scale = .35]{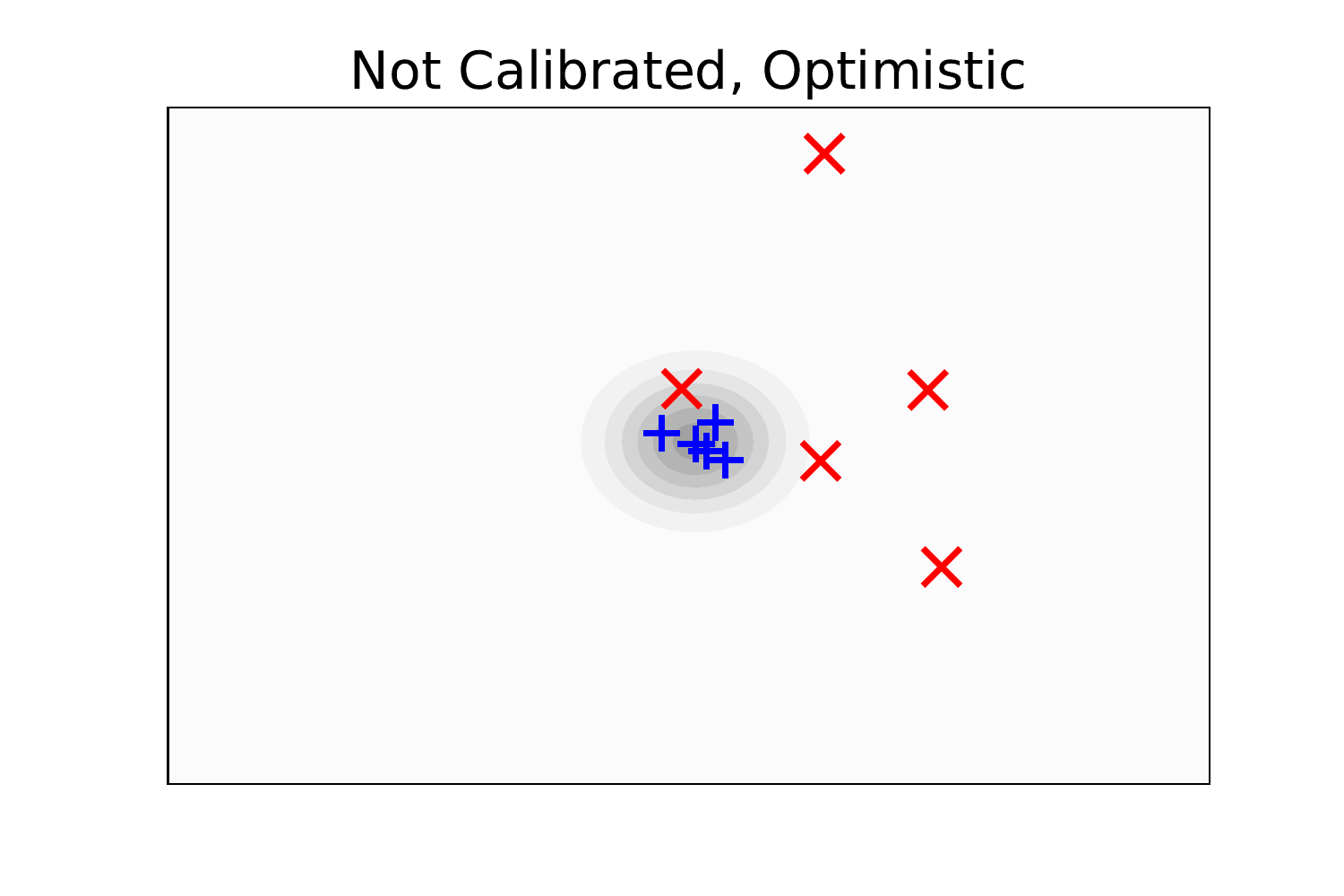}
  \includegraphics[scale = .35]{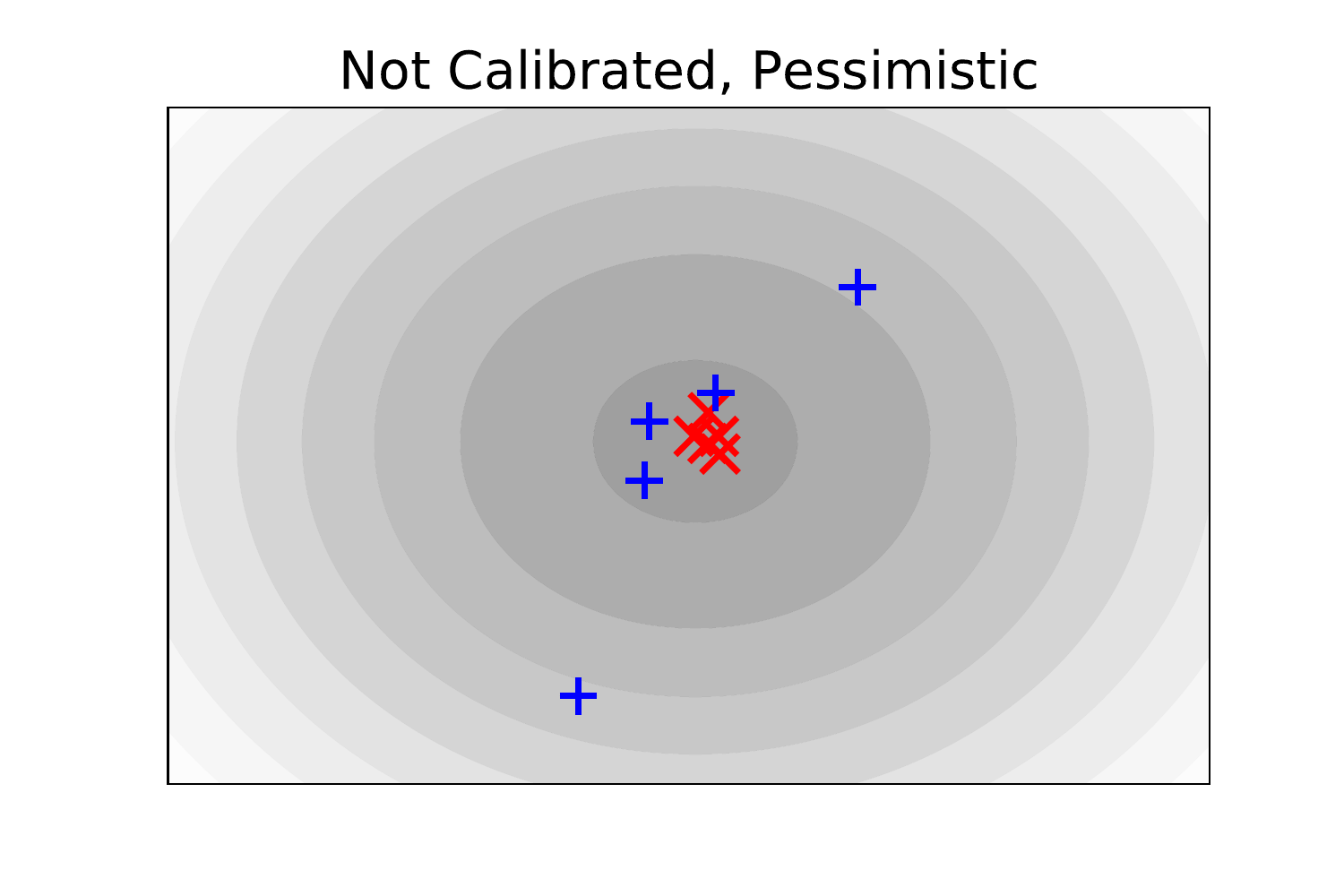}
  \caption{Posterior distributions and solutions from three different probabilistic solvers: calibrated (top), optimistic (bottom left), and pessimistic (bottom right). The gray contours represent the posterior distributions, the red symbols ``$\times$'' the solutions,
  and the blue symbols ``+''  samples from the posterior distributions.}
%  \color{red}{I realise this is supposed to be a cartoon but it's a bit jarring for me that there is only \emph{one} posterior distribution depicted. Given what we've seen so far there should be one for each red cross.}
 % \textcolor{blue}{I will think of ways to edit the picture}} 
  \label{F:Calibration}
\end{figure}

\begin{example}
  \label{Ex:Pictures}
  This is the visual equivalent of Example~\ref{Ex:Story},
  where Carol receives the images in  \figref{F:Calibration} of three different
  probabilistic solvers, but 
  without any identification of the solutions and posterior samples.
      
\begin{itemize}
\item Top plot.
This solver is calibrated because the solutions 
 look indistinguishable from the samples of the posterior distribution.
\item Bottom left plot.
This solver  is not calibrated because the solutions are unlikely 
to be samples from the posterior distribution.\\
  The solver is \emph{optimistic} because the posterior distribution is concentrated in an area of $\Rn$ that is too small to cover the solutions.
  \item Bottom right plot.
The solver is not calibrated. Although the solutions could plausibly be sampled from the posterior, they are concentrated too close to the center of the distribution. \\
The solver is \emph{pessimistic} because the area covered
by the posterior distribution is much larger than the area containing the solutions.
\end{itemize}
\end{example}
  
\begin{remark}
\label{r_41}
The posterior means and covariances from a probabilistic solver can depend on the solution $\xvec_*$, as is the case for BayesCG. If a solver is applied to a random linear system in \dref{D:StrongCalibration} and if the posterior means and covariances depend on the solution $\Xrv_*$, then the posterior means and covariances are also random variables.

\dref{D:StrongCalibration} prevents the posterior covariances from being random variables by forcing them to be independent of the random right hand side $\Brv$.
Although this is a realistic constraint for the stationary iterative solvers in  
\cite{RICO21},
it does not apply to BayesCG under the Krylov prior, because 
Krylov  posterior covariances depend non-linearly on the right-hand side. In Sections~\ref{S:UQZ} and~\ref{S:UQS}, we present a remedy for BayesCG in the form
of test statistics that are motivated by \dref{D:StrongCalibration} and \lref{L:CalibrationAlt}.
\end{remark}

\subsection{The $Z$-statistic}
\label{S:UQZ}
We assess BayesCG under the Krylov prior with
an existing test statistic, the $Z$-statistic, which is a necessary condition for calibration and can be viewed as a weaker normwise version of criterion (\ref{Eq:StrongCalibration}).
We review the $Z$-statistic (Section~\ref{S:UQZDef}), and apply 
it to BayesCG under the Krylov prior (Section~\ref{S:UQZKrylov}).

\subsubsection{Review of the $Z$-statistic}
\label{S:UQZDef}
We review the $Z$-statistic (Definition~\ref{d:ZStat}), and the 
chi-square distribution (Definition~\ref{D:ChiSquare}), which links the $Z$-statistic to calibration (Theorem~\ref{T:ZStat}).
Then we discuss how to generate samples of the $Z$-statistic  (Algorithm~\ref{A:ZStat}), how to use them for the assessment of  calibration (Remark~\ref{r_zeval}), and then present
the Kolmogorov-Smirnov statistic as a computationally inexpensive estimate for the 
difference between two general distributions (Definition~\ref{D:KSTest}).
 
The $Z$-statistic was introduced in \cite[Section 6.1]{Cockayne:BCG} as a means  to assess the calibration of BayesCG, and has subsequently been applied to other probabilistic linear solvers \cite[Section 6.4]{Bartels}, \cite[Section 9]{Fanaskov21}. 

\begin{definition}[{\cite[Section 6.1]{Cockayne:BCG}}]\label{d:ZStat}
 Let $\Amat\Xrv_* = \Brv$ be a class of linear systems where $\Amat\in\Real^{n\times n}$ is symmetric positive definite, and $\Xrv_*\sim\mu_0 \equiv \N(\xvec_0,\Sigmat_0)$.
 Let $\mu_m \equiv \N(\xvec_m,\Sigmat_m)$, $1\leq m \leq n$, be the posterior
 distributions from a probabilistic solver under the prior $\mu_0$ applied to $\Amat\Xrv_* = \Brv$. The $Z$-statistic is
 \begin{equation}
    \label{Eq:ZStat}
    \Zstat_m(\Xrv_*) \equiv (\Xrv_*-\xvec_m)^T\Sigmat_m^\dagger(\Xrv_*-\xvec_m),
    \qquad 1\leq m\leq n.
  \end{equation}
\end{definition} 
     
     The chi-squared distribution below furnishes the link from $Z$-statistic to calibration.

\begin{definition}[{\cite[Definition 2.2]{Ross07}}]
  \label{D:ChiSquare}
If $\Xrv_1,\ldots, \Xrv_f\in\mathcal{N}(0,1)$  are independent
random normal variables, then $\sum_{j=1}^{f}{\Xrv_j^2}$
  is distributed according to the chi-squared distribution $\chi_f^2$ with $f$ degrees of freedom and mean~$f$.\\
In other words, if $\Xrv\sim\N(\zerovec,\Imat_f)$, then
$\Xrv^T\Xrv\sim\chi_f^2$ and $\Exp[\Xrv^T\Xrv]=f$.
\end{definition}

We show that the $Z$-statistic is a necessary condition for calibration. That is:
If a probabilistic solver is calibrated, then  
the $Z$-statistic is distributed according to a chi-squared distribution.

\begin{theorem}[{\cite[Proposition 1]{BCG:Supp}}]
    \label{T:ZStat}
    Let $\Amat\Xrv_* = \Brv$ be a class of linear systems where $\Amat\in\Real^{n\times n}$ is symmetric positive definite, and $\Xrv_*\sim\mu_0 \equiv \N(\xvec_0,\Sigmat_0)$.
Assume that a probabilistic solver under the prior $\mu_0$ applied to $\Amat\Xrv_* = \Brv$
computes the posteriors $\mu_m \equiv \N(\xvec_m,\Sigmat_m)$ with $\rank(\Sigmat_m) = p_m$, $1\leq m\leq n$. 

If the solver is calibrated, then 
\begin{equation*}
\Zstat_m(\Xrv_*)\sim\chi_{p_m}^2, \qquad 1\leq m\leq n.
\end{equation*}
\end{theorem}

\begin{proof}
Write $\Zstat_m(\Xrv_*)= \Mrv_m^T\Mrv_m$, where
$\Mrv_m \equiv (\Sigmat_m^\dagger)^{1/2}(\Xrv_*-\xvec_m)$.
\lref{L:CalibrationAlt} implies that a calibrated solver produces posteriors with
\begin{equation*}
(\Xrv_*-\xvec_m)\sim\N(\zerovec,\Sigmat_m),\qquad 1\leq m\leq n.
\end{equation*}
With the eigenvector matrix $\Umat_m\in\Real^{n\times p_m}$ as in \dref{D:StrongCalibration},
\lref{L:GaussStability} in Appendix~\ref{S:Aux} implies
\begin{equation*}
    \Mrv_m\sim\N(\zerovec,\Umat_m\Umat_m^T),\qquad 1\leq m\leq n.
\end{equation*}
Since the covariance of $\Mrv_m$ is an orthogonal projector,
\lref{L:ChiProjector} implies
$\Zstat_m(\Xrv_*) = (\Mrv_m^T\Mrv_m) \sim \chi_{p_m}^2$.
\qed
\end{proof}

\tref{T:ZStat} implies that BayesCG is calibrated if the $Z$-statistic is distributed according to a chi-squared distribution with $p_m = \rank(\Sigmat_0) - m$ degrees of freedom. For the Krylov prior specifically, $p_m = \kry - m$.

\paragraph{Generating samples from the $Z$-statistic and assessing calibration.}
For a user-specified probabilistic linear \texttt{solver} and a symmetric positive definite matrix~$\Amat$,
Algorithm~\ref{A:ZStat} samples $N_{\text{test}}$
solutions $\xvec_*$ from the prior distribution $\mu_{0}$, defines the systems
$\bvec\equiv \Amat\xvec_*$,
runs $m$ iterations of the \texttt{solver} on $\bvec\equiv \Amat\xvec_*$,
and computes $\Zstat_m(\xvec_*)$ in (\ref{Eq:ZStat}).

The application of the Moore-Penrose inverse in Line~\ref{A:ZStatLstsq}
can be implemented by computing the minimal norm solution
$\qvec=\Sigmat_m^{\dagger}(\xvec_*-\xvec_m)$
of the least squares problem
\begin{equation}
\label{Eq:ZStatLstsq}
 \min_{\uvec\in\Rn} \|(\xvec_*-\xvec_m) - \Sigmat_m\uvec \|_2,
\end{equation}
followed by the inner product
$z_i = (\xvec_*-\xvec_m)^T\qvec$.

\begin{algorithm}
\caption{Sampling from the $Z$-statistic}
\label{A:ZStat}
\begin{algorithmic}[1]
  \State \textbf{Input:} spd $\Amat\in\Rnn$, $\mu_0 = \N(\xvec_0,\Sigmat_0)$, 
  \texttt{solver}, $m$, $N_{\text{test}}$
  \For{$i = 1:N_{\text{test}}$}
  \State{Sample $\xvec_*$ from prior distribution $\mu_{0}$} 
  \Comment{Sample solution vector}
  \State{$\bvec = \Amat\xvec_*$} \Comment{Define test problem}
  \State{$[\xvec_m,\Sigmat_m]= \texttt{solver}(\Amat,\bvec,\mu_0,m)$} \Comment{Compute posterior $\mu_m\equiv\N(\xvec_m,\Sigmat_m)$}
    \State{$ z_i = (\xvec_*-\xvec_m)^T\Sigmat_m^\dagger(\xvec_*-\xvec_m)$} \Comment{Compute $Z$-statistic sample} \label{A:ZStatLstsq}
  \EndFor
  \State{\textbf{Output:} $Z$-statistic samples $z_i$, $1\leq i \leq N_{test}$.}
\end{algorithmic}
\end{algorithm}

\begin{remark}\label{r_zeval}
We assess calibration of the \texttt{solver}  
by comparing the $Z$-statistic samples $z_i$ from Algorithm~\ref{A:ZStat}
to the chi-squared distribution $\chi_{p_m}^2$ with $p_m\equiv\rank(\Sigmat_0)-m$ degrees of freedom, 
based on the following criteria from \cite[Section 6.1]{Cockayne:BCG}.

\begin{description}
\item[\bf Calibrated:] If $z_i\sim\chi_{p_m}^2$, 
then $\xvec_*\sim\mu_m$ and the solutions $\xvec_*$ are  distributed according to 
the posteriors $\mu_m$.

\item[\bf Pessimistic:] If the $z_i$ are concentrated around smaller values than $\chi_{p_m}^2$, then the solutions $\xvec_*$ occupy a smaller area of $\Rn$ than predicted by 
$\mu_m$.

\item[\bf Optimistic:] If the $z_i$ are concentrated around larger values than $\chi_{p_m}^2$, then the solutions cover a larger area of $\Rn$ than predicted by $\mu_m$.
\end{description}
\end{remark}

In \cite[Section 6.1]{Cockayne:BCG} and \cite[Section 6.4]{Bartels}, the $Z$-statistic samples and the predicted chi-squared distribution are compared visually. In \sref{S:Experiments}, we make an additional quantitative comparison with the Kolmogorov-Smirnov test to estimate the difference between 
two probability distributions.

\begin{definition}[{\cite[Section 3.4.1]{Kaltenbach12}}] \label{D:KSTest}
  Given two distributions $\mu$ and $\nu$ on $\Rn$ with cumulative distribution functions $F_\mu$ and $F_\nu$, the Kolmogorov-Smirnov statistic is
\begin{equation*}
KS(\mu,\nu) = \sup_{x\in\Real} |F_\mu(x) - F_\nu(x)|,
\end{equation*}
where $0\leq KS(\mu,\nu)\leq 1$.

If $KS(\mu,\nu) = 0$, then $\mu$ and $\nu$  have the same cumulative distribution functions, $F_\mu=F_\nu$.
If $KS(\mu,\nu) = 1$, then $\mu$ and $\nu$ do not overlap.
In general, the lower $KS(\mu,\nu)$, the closer $\mu$ and $\nu$ are to each other.
\end{definition}

In contrast to the Wasserstein distance in Definition~\ref{D:Wasserstein},
the Kolmogorov-Smirnov statistic can be easier to 
estimate---especially if the distributions are not Gaussian---but it is not a metric.
Consequently, if  $\mu$ and $\nu$ do not overlap, then $KS(\mu,\nu) = 1$
regardless of how far $\mu$ and $\nu$ are apart, while the Wasserstein metric 
still gives information about the distance between $\mu$ and~$\nu$.

\subsubsection{$Z$-Statistic for BayesCG under the Krylov prior}
\label{S:UQZKrylov}
We apply the $Z$-statistic to BayesCG under the Krylov prior. We  start with an
expression for the Moore-Penrose inverse of the Krylov posterior covariances (Lemma~\ref{L:KrylovInv}). Then
we show that the $Z$-statistic for the full Krylov posteriors 
has the same \textit{mean} as the corresponding chi-squared distribution (\tref{T:KrylovZ}), but its \textit{distribution} is different. Therefore the $Z$-statistic is inconclusive about the calibration of BayesCG under the Krylov prior (Remark~\ref{R:ZBad}).

\begin{lemma}
  \label{L:KrylovInv}
In Definition~\ref{D:KrylovApprox}, 
abbreviate $\widehat{\Vmat} \equiv \Vmat_{m+1:m+d}$ and $\widehat{\Phimat} \equiv \Phimat_{m+1:m+d}$. The rank-$d$ 
approximate Krylov posterior covariances have the
Moore-Penrose inverse
  \begin{equation*}
    \widehat{\Gammat}_m^\dagger = 
\left(\widehat{\Vmat}\widehat{\Phimat}\widehat{\Vmat}^T\right)^{\dagger}=
    \widehat{\Vmat}(\widehat{\Vmat}^T\widehat{\Vmat})^{-1}\widehat{\Phimat}^{-1}(\widehat{\Vmat}^T\widehat{\Vmat})^{-1}\widehat{\Vmat}^T, \qquad 1\leq m\leq g-d.
  \end{equation*}
  \end{lemma}

\begin{proof}
We exploit the fact that all factors of $\widehat{\Gammat}_m$ have full column rank.

The factors $\widehat\Vmat$ and $\widehat\Vmat^T$ have full column and row rank, respectively, because $\Vmat$ has $\Amat$-orthonormal columns. Additionally, the diagonal matrix $\widehat{\Phimat}$ is nonsingular. Then
\lref{L:PseudoProduct} in Appendix~\ref{S:Aux} implies that the Moore-Penrose inverses can be expressed
in terms of the matrices proper,
  \begin{equation}
    \label{Eq:VInv}
    \widehat{\Vmat}^\dagger = (\widehat{\Vmat}^T\widehat{\Vmat})^{-1}\widehat{\Vmat}^T, \qquad (\widehat{\Vmat}^T)^\dagger = \widehat{\Vmat}(\widehat{\Vmat}^T\widehat{\Vmat})^{-1},
  \end{equation}
  and
  \begin{equation}
    \label{Eq:PhiVInv}
    (\widehat{\Phimat}\widehat{\Vmat}^T)^\dagger = (\widehat{\Vmat}^T)^\dagger\widehat{\Phimat}^{-1} = \widehat{\Vmat}(\widehat{\Vmat}^T\widehat{\Vmat})^{-1} \widehat{\Phimat}^{-1}.
  \end{equation}
Since $\widehat{\Phimat}\widehat{\Vmat}^T$ also has full row rank, apply \lref{L:PseudoProduct} to $\widehat{\Gammat}_m$,
  \begin{equation*}
    \widehat{\Gammat}_m^\dagger = (\widehat{\Phimat}\widehat{\Vmat}^T)^\dagger\widehat{\Vmat}^\dagger,
  \end{equation*}
and substitute \eref{Eq:VInv} and \eref{Eq:PhiVInv} into the above expression.\qed
    \end{proof}

We apply the $Z$-statistic to the full Krylov posteriors, and show
that $Z$-statistic samples reproduce the dimension of the unexplored Krylov space.

\begin{theorem}
  \label{T:KrylovZ}
Under the assumptions of Theorem~\ref{T:KrylovPrior}, let BayesCG
under the Krylov prior $\eta_0 \equiv \N(\xvec_0,\Gammat_0)$ produce full Krylov posteriors
$\eta_m \equiv \N(\xvec_m,\Gammat_m)$. Then the $Z$-statistic is equal to
  \begin{equation*}
    \Zstat_m(\xvec_*) = (\xvec_*-\xvec_m)^T\Gammat_m^\dagger(\xvec_*-\xvec_m) = \kry-m, \qquad 1\leq m\leq \kry.
  \end{equation*}
\end{theorem}

\begin{proof}
Express the error $\xvec_0-\xvec_m$ 
in terms of $\widehat{\Vmat} \equiv \Vmat_{m+1:m+d}$
by inserting 
\begin{align}
\xvec_*=\xvec_0+\Vmat_{1:\kry}\Vmat_{1:\kry}^T\rvec_0,\qquad
\xvec_m=\xvec_0+\Vmat_{1:m}\Vmat_{1:m}^T\rvec_0,
\quad 1\leq m\leq \kry, \label{Eq:DirectPosterior}
\end{align}
from Theorem~\ref{T:KrylovPrior} into 
\begin{align*}%\label{Eq:KrylovError}
  \xvec_* -\xvec_m = \Vmat_{m+1:\kry}\Vmat_{m+1:\kry}^T\rvec_{0}
=\widehat{\Vmat}\widehat{\Vmat}^T\,\rvec_0.
\end{align*}
This expression is identical to \cite[Equation (5.6.5)]{Liesen}, which relates the CG error to the search directions and step sizes of the remaining iterations.

With \lref{L:KrylovInv}, this implies for the $Z$-statistic in Theorem~\ref{T:ZStat}
  \begin{align*}
    \label{Eq:T:KrylovZ}
    \Zstat_m(\xvec_*) &= (\xvec_*-\xvec_m)^T\Gammat_m^\dagger\xvec_*-\xvec_m) \\
    &=\underbrace{\rvec_0^T\widehat{\Vmat}\widehat{\Vmat}^T}_{(\xvec_*-\xvec_m)^T} \underbrace{\widehat{\Vmat}(\widehat{\Vmat}^T\widehat{\Vmat})^{-1}\widehat{\Phimat}^{-1}(\widehat{\Vmat}^T\widehat{\Vmat})^{-1}\widehat{\Vmat}^T}_{\Gammat_m^\dagger} \underbrace{ \widehat{\Vmat}\widehat{\Vmat}^T\rvec_0}_{(\xvec_*-\xvec_m)}\\
    &=\rvec_0^T\widehat{\Vmat}\>
    \underbrace{(\widehat{\Vmat}^T\widehat{\Vmat})(\widehat{\Vmat}^T\widehat{\Vmat})^{-1}}_{\Imat}\>
    \widehat{\Phimat}^{-1}\>
    \underbrace{(\widehat{\Vmat}^T\widehat{\Vmat})^{-1}(\widehat{\Vmat}^T \widehat{\Vmat})}_{\Imat}\>\widehat{\Vmat}^T\rvec_0\\
&=\rvec_0^T\widehat{\Vmat}\widehat{\Phimat}^{-1}\widehat{\Vmat}^T\rvec_0.
  \end{align*}
  In other words,
  \begin{align*}
    \|\xvec_*-\xvec_m\|_{\widehat{\Gammat}_m^\dagger}^2 &= \left(\Vmat_{m+1:\kry}^T\rvec_0\right)^T \Phimat_{m+1:m+d}^{-1} \left(\Vmat_{m+1:\kry}^T\rvec_0\right)\\
    &= \sum_{j = m+1}^{\kry} \phi_j^{-1} (\vvec_j^T\rvec_0)^2=g-m,
    \qquad 0\leq m< \kry,
  \end{align*}
where the last inequality follows from $\phi_j = (\vvec_j^T\rvec_0)^2$ in
Definition~\ref{D:KrylovPrior}.\qed
  \end{proof}

\begin{remark}
  \label{R:ZBad}
 The $Z$-statistic is inconclusive about the calibration of BayesCG under the Krylov prior.
 
\tref{T:KrylovZ} shows that the $Z$-statistic is distributed according to a Dirac distribution at $\kry - m$. Thus,
the $Z$-statistic has the same mean as the chi-squared distribution  $\chi_{\kry-m}^2$, which suggests that BayesCG under the Krylov prior is neither optimistic or pessimistic.
However, although the means are the same, the distributions are not.
Hence, \tref{T:KrylovZ} does not imply that BayesCG under the 
Krylov prior is calibrated. 
\end{remark}

% I am not sure the remarks below are helpful, as they are forward references and in some sense digressions.

\begin{comment}
Furthermore, the $Z$-statistic gives no information about the optimism or pessimism
of BayesCG under the Krylov prior as defined in \rref{r_zeval}, because the  $Z$-statistic  is not concentrated around smaller or larger values than $\chi_{\kry-m}^2$. 

However, under rank-$d$ approximate Krylov posteriors, we expect BayesCG to be optimistic when compared to full posteriors. This is because approximate posteriors model the uncertainty about $\xvec_*$ in lower dimensional spaces than full posteriors.

Per Definition~\ref{D:StrongCalibration} of calibration, the solutions to
test problems must be distributed according to the prior $\mu_0$,
see \aref{A:ZStat}.
However, the Krylov prior depends
on the particular solution $\xvec_*$, hence is different for every test problem.

The remedy in \sref{S:Experiments} is to generate test problems with a separate reference distribution 
$\mu_{\text{ref}} \equiv \N(\zerovec,\Amat^{-1})$
when applying 
the $Z$-statistic to BayesCG under the Krylov prior. 
\end{comment}

\subsection{The $S$-statistic}
\label{S:UQS}

We introduce a new test statistic for assessing the calibration of probabilistic solvers,
the $S$-statistic.
After discussing the relation between calibration and error estimation (\sref{S:CalibrationError}), we define the $S$-statistic (Section~\ref{S:SStatDef}), compare the $S$-statistic to the $Z$-statistic (\sref{S:ZSComp}), and then apply the $S$-statistic to BayesCG under the Krylov prior (Section~\ref{S:BayesCGS}).

\subsubsection{Calibration and Error Estimation}
\label{S:CalibrationError}
We establish a relation between the error of the posterior means (approximations to the solution) and the trace of posterior covariances (Theorem~\ref{T:CalibratedTrace}).

\begin{theorem}
\label{T:CalibratedTrace}
Let $\Amat\Xrv_* = \Brv$ be a class of linear systems where $\Amat\in\Real^{n\times n}$ is symmetric positive definite and $\Xrv_*\sim\mu_0 \equiv\N(\xvec_0,\Sigmat_0)$. Let $\mu_m \equiv \N(\xvec_m,\Sigmat_m)$, $1\leq m \leq n$ be the posterior distributions from a probabilistic solver under the prior $\mu_0$ applied to $\Amat\Xrv_* = \Brv$.

If the solver is calibrated, then
\begin{equation}
\label{Eq:T:CalibratedTrace}
\Exp[\|\Xrv_*-\xvec_m\|_\Amat^2] = \trace(\Amat\Sigmat_m), \qquad 1\leq m\leq n.
\end{equation}
\end{theorem}

\begin{proof}
For a calibrated solver \lref{L:CalibrationAlt} implies that $\Xrv_*\sim\mu_m$.
Then apply \lref{L:ExpS} in Appendix~\ref{S:Aux} to the error $\|\Xrv_*-\xvec_m\|_\Amat^2$. \qed
\end{proof}

For a calibrated solver,
\tref{T:CalibratedTrace} implies that the equality $\|\xvec_*-\xvec_m\|_\Amat^2 = \trace(\Amat\Sigmat_m)$ 
holds \emph{on average}. This means the trace can overestimate 
the error for some solutions, 
while for others, it can underestimate the error. 

We explain how \tref{T:CalibratedTrace} relates the errors of a calibrated solver to the area in which its posteriors are concentrated. 

\begin{remark}
\label{R:SStatInterpretation}
The trace of a posterior covariance matrix quantifies the spread of its probability distribution---because the trace is the sum of the eigenvalues, which in the case of a covariance are the variances of the principal components \cite[Section 12.2.1]{JWHT21}. 

In analogy to viewing the $\Amat$-norm as the 2-norm weighted by $\Amat$, we can view $\trace(\Amat\Sigmat_m)$ as the trace of $\Sigmat_m$ weighted by $\Amat$. \tref{T:CalibratedTrace} shows that the $\Amat$-norm errors of a calibrated solver are equal to the weighted sum of the principal component variances from the posterior. Thus, the posterior means $\xvec_m$ and the areas in which the posteriors are concentrated both converge to the solution at the same speed, provided the solver is calibrated.
\end{remark}

%The relation between errors and variances of posteriors in \tref{T:CalibratedTrace} 
%allows
%for accurate error estimation based on sampling from posteriors \cite[Section 3.3.1]{RICO21}.
%Such sampling based error estimates are discussed further in \appref{S:WassersteinS}. 

\subsubsection{Definition of the $S$-statistic}
\label{S:SStatDef}
We introduce the $S$-statistic (\dref{D:SStat}), present an algorithm for generating
samples from the $S$-statistic (\aref{A:SStat}), and discuss their use
for assessing calibration of solvers (\rref{r_seval}).

The $S$-statistic represents a necessary condition for calibration,
as established in \tref{T:CalibratedTrace}.

\begin{definition}
  \label{D:SStat}
  Let $\Amat\Xrv_* = \Brv$ be a class of linear systems where $\Amat\in\Rnn$ is symmetric positive definite, and $\Xrv_*\sim\mu_0\equiv\N(\xvec_0,\Sigmat_0)$. 
Let $\mu_m \equiv \N(\xvec_m,\Sigmat_m)$, $1\leq m \leq n$, be the posterior
distributions from a probabilistic solver under the prior $\mu_0$ applied to $\Amat\Xrv_* = \Brv$. The $S$-statistic is
  \begin{equation}
      \label{Eq:SStat}
      \Sstat_m(\Xrv_*) \equiv \|\Xrv_*-\xvec_m\|_\Amat^2.
  \end{equation}
If the solver is calibrated then \tref{T:CalibratedTrace} implies 
  \begin{equation}
  \label{Eq:SStatTrace}
      \Exp[\Sstat(\Xrv_*)] = \trace(\Amat\Sigmat_m).
  \end{equation}
\end{definition}

\paragraph{Generating samples from the $S$-statistic and assessing calibration.}
For a user specified probabilistic linear \texttt{solver} and a symmetric positive definite matrix~$\Amat$, \aref{A:SStat} samples $N_{test}$ solutions $\xvec_*$ from the prior distribution $\mu_{0}$, defines the linear systems 
$\bvec = \Amat\xvec_*$, runs $m$ iterations of the solver on the system,
and computes $\Sstat_m(\xvec_*)$ and  $\trace(\Amat\Sigmat_m)$ from \eref{Eq:SStat}.

As with the $Z$-statistic, Algorithm~\ref{A:SStat} requires a separate reference
$\mu_{ref}$ when sampling solutions $\xvec_*$ for BayesCG under the Krylov prior.

\begin{algorithm}
\caption{Sampling from the $S$-statistic}
\label{A:SStat}
\begin{algorithmic}[1]
  \State \textbf{Input:} spd $\Amat\in\Rnn$, $\mu_0 = \N(\xvec_0,\Sigmat_0)$,  \texttt{solver}, $m$, $N_{\text{test}}$
  \For{$i = 1:N_{\text{test}}$}
  \State{Sample $\xvec_*$ from prior distribution $\mu_{0}$} 
  \Comment{Sample solution vector}
  \State{$\bvec = \Amat\xvec_*$} \Comment{Define test problem}
  \State{$[\xvec_m,\Sigmat_m]= \texttt{solver}(\Amat,\bvec,\mu_0,m)$} \Comment{Compute posterior $\mu_m\equiv\N(\xvec_m,\Sigmat_m)$}
  \State{$s_i = \|\xvec_*-\xvec_m\|_\Amat^2$} \Comment{Compute $S$-statistic for test problem}
  \State{$t_i = \trace(\Amat\Sigmat_m)$} \Comment{Compute trace for test problem}
  \EndFor
  \State{$h = (1/N_{test})\sum_{i = 1}^{N_\text{test}} s_i$} \Comment{Compute empirical mean of $S$-statistic samples}
  \State{\textbf{Output:} $S$-statistic samples $s_i$ and traces $t_i$, $1\leq i \leq N_{\text{test}}$; $S$-statistic mean $h$}
\end{algorithmic}
\end{algorithm}

\begin{remark}\label{r_seval}
We assess calibration of the solver by comparing the $S$-statistic
samples $s_i$ from Algorithm~\ref{A:SStat}
to the traces $t_i$, $1\leq i \leq N_{test}$. The following criteria are based on \tref{T:CalibratedTrace} and \rref{R:SStatInterpretation}.

\begin{description}
\item[\bf Calibrated:] If the solver is calibrated, the traces $t_i$ should all be equal to the empirical mean $h$ of the $S$-statistic samples $s_i$.
\item[\bf Pessimistic:] If the $s_i$ are concentrated around smaller values than the $t_i$, then the solutions $\xvec_*$ occupy a smaller area of $\Rn$ than predicted by the posteriors $\mu_m$.

\item[\bf Optimistic:] If the $s_i$ are concentrated around larger values than the $t_i$, then the solutions $\xvec_*$ occupy a larger area of $\Rn$ than predicted by $\mu_m$.
\end{description}

 We can also compare the empirical means of the $s_i$ and $t_i$, because a calibrated solver should produce $s_i$ and $t_i$ with the same mean.
Note that a comparison via the Kolmogorov-Smirnov statistic is not appropriate because 
the empirical distributions of $s_i$ and $t_i$ are generally different. 
\end{remark}

\subsubsection{Comparison of the $Z$- and $S$-statistics}
\label{S:ZSComp}

Both, $Z$- and $S$-statistic represent necessary conditions for calibration
((\ref{Eq:T:CalibratedTrace}) and (\ref{Eq:SStatTrace})); and both
 measure the norm of the error $\Xrv_*-\xvec_m$: The $Z$-statistic in the $\Sigmat_m^\dagger$-pseudo norm (Definition~\ref{d:ZStat}), and the
$S$-statistic in the $\Amat$-norm (Definition~\ref{D:SStat}).
Deeper down, though, the $Z$-statistic projects errors onto a single dimension 
(Theorem~\ref{T:ZStat}), while the $S$-statistic relates errors to the areas in which the posterior distributions are concentrated.

Due to its focus on the area of the posteriors, the $S$-statistic can give a \textit{false positive} for calibration. This occurs when the solution is not in the area of posterior concentration but the size of the posteriors is consistent with the errors. The $Z$-statistic is less likely to encounter this problem,
as illustrated in~\figref{F:ZSComparison}.

The $Z$-statistic is better at assessing calibration,
while the $S$ statistic produces accurate 
error estimates, which default to the traditional $\Amat$-norm estimates.
The $S$-statistic is also faster to compute because it does not require the solution of a least squares problem.

\begin{figure}
  \centering
  \includegraphics[scale = .35]{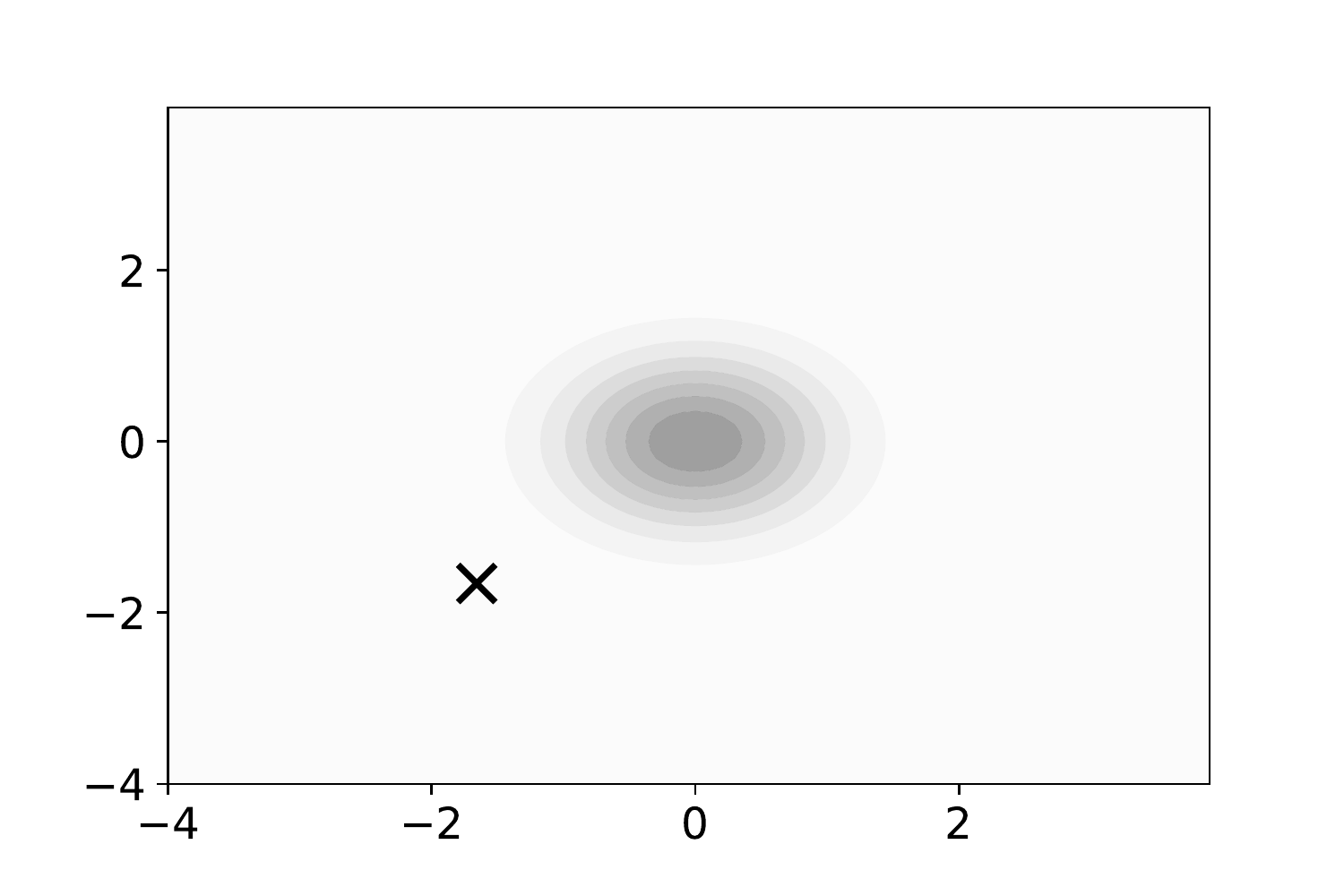}
  \includegraphics[scale = .35]{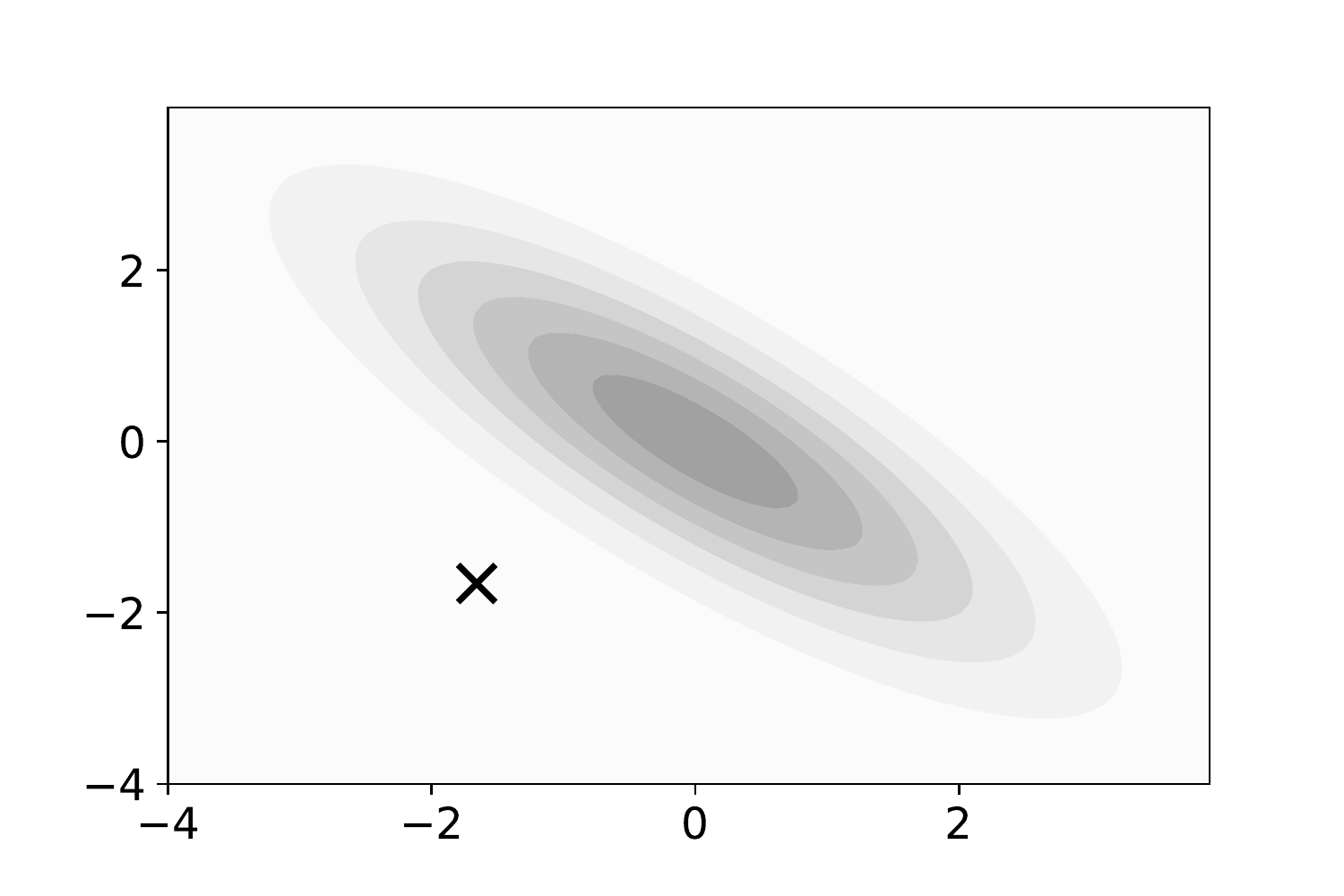} \\
  \includegraphics[scale = .35]{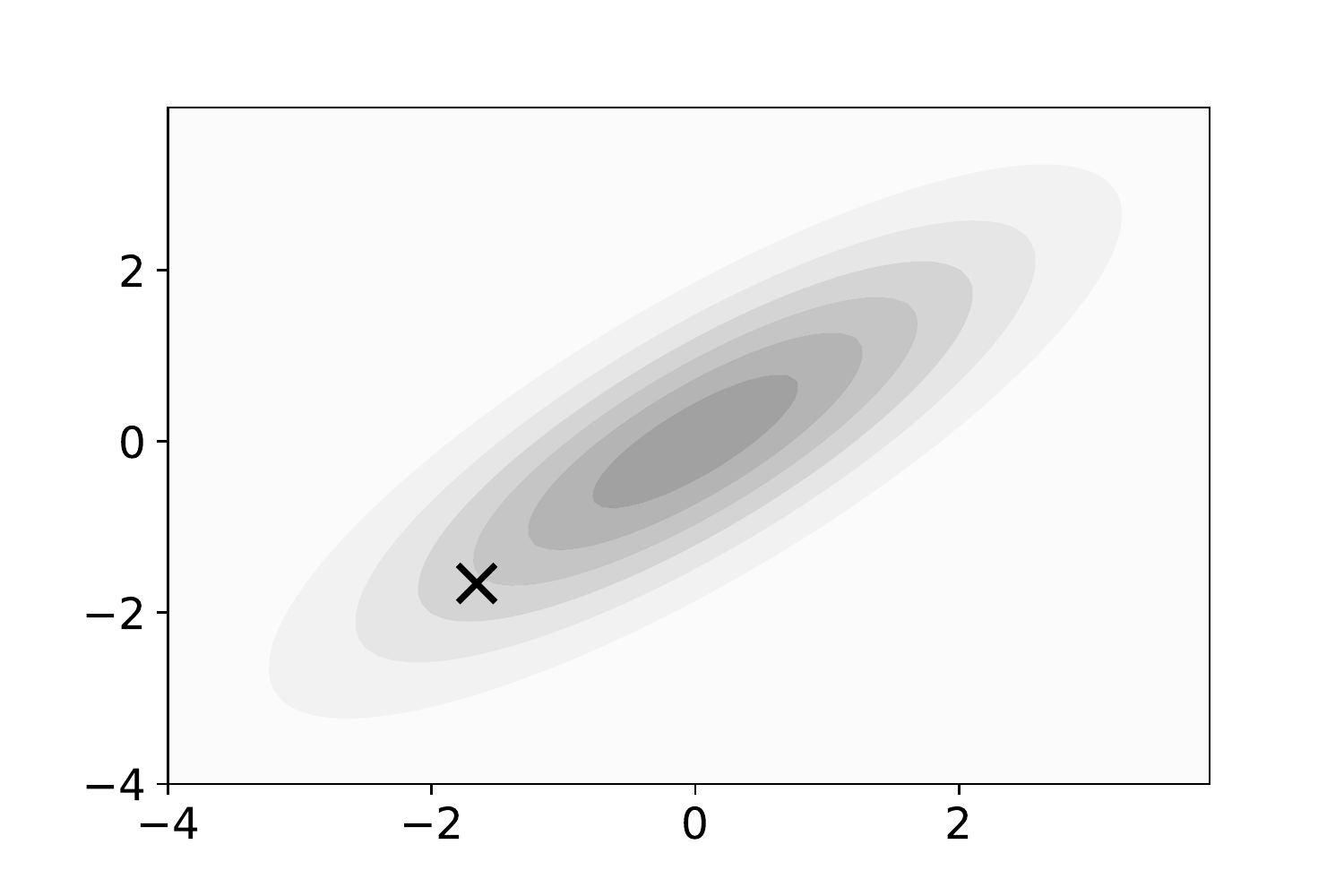}
  \caption{Assessment of calibration from $Z$-statistic and $S$-statistic. The contour plots represent the posterior distributions, and the symbol `$\times$' represents the solution.\\ 
 Top left: Both statistics decide that the solver is not  calibrated. Top right: The $S$-statistic decides that the solver
 is calibrated, while the $Z$-statistic does not. 
 Bottom: Both statistics decide that the solver is calibrated.}
  \label{F:ZSComparison}
\end{figure}

\subsubsection{$S$-statistic for BayesCG under the Krylov prior}
\label{S:BayesCGS}
We show that BayesCG under the Krylov prior is not calibrated, but that it is has similar performance to a calibrated solver under full posteriors and is optimistic under approximate posteriors.

\paragraph{Calibration of BayesCG under full Krylov posteriors.}
\tref{T:KrylovPrior} implies 
that the $S$-statistic for any solution $\xvec_*$ is equal to
\begin{equation*}
    \Sstat_m(\xvec_*) = \|\xvec_*-\xvec_m\|_\Amat^2 = \trace(\Amat\Gammat_m), \qquad 1 \leq m \leq \kry.
\end{equation*}
Thus, the $S$-statistic indicates that the size of Krylov posteriors is consistent with the errors, which is a desirable property of calibrated solvers.
However, BayesCG under the Krylov prior is not a calibrated solver
because  the traces of posterior covariances from  calibrated solvers
are distributed around the \emph{average} error instead of always being equal to the error.

\paragraph{Calibration of BayesCG under approximate posteriors.}

From~(\ref{Eq:StrakosTichy}) follows that $\trace(\Amat\widehat\Gammat_m)$ is concentrated around smaller values than the $S$-statistic; 
and the underestimate of the trace is equal to the Wasserstein distance between full and approximate Krylov posteriors in \tref{T:Wasserstein}. This underestimate points
to the optimism of BayesCG under approximate Krylov posteriors. 
This optimism is expected because
approximate posteriors model the uncertainty about $\xvec_*$ in a lower dimensional space than full  posteriors.

\section{Numerical experiments}
\label{S:Experiments}
We present numerical assessments of BayesCG calibration via the $Z$- and $S$-statistics. 

After describing the setup of the numerical experiments (\sref{S:ExpSetup}), 
we assess the calibration of three implementations of BayesCG:
(i) BayesCG with random search directions (\sref{S:ExpRand})---a solver known to be calibrated---so as to establish a baseline for comparisons with other versions of BayesCG;
(ii) BayesCG under the inverse prior (\sref{S:ExpInv});
and (iii) BayesCG under the Krylov prior (\sref{S:ExpKry}).
We choose the inverse prior and the Krylov priors because under each of these priors, the posterior mean from BayesCG coincides with the solution from CG.

\paragraph{Conclusions from all the experiments.}
Both, $Z$- and $S$ statistics indicate that BayesCG with random search directions is indeed a calibrated solver, and that BayesCG under the inverse prior is pessimistic.

The $S$-statistic indicates that 
BayesCG under full Krylov posteriors mimics a calibrated solver, and that
BayesCG under rank-50 approximate posteriors does as well, but not as much since it is 
slightly optimistic.  

However, among all versions, BayesCG under approximate Krylov posteriors is the only one that 
is computationally practical and that is competitive with CG.

\subsection{Experimental Setup}
\label{S:ExpSetup}
We describe the matrix $\Amat$ in the linear system (Section~\ref{s_exp3});
the setup of the $Z$- and $S$-statistic experiments (Section~\ref{s_exp1});
and the three BayesCG implementations (Section~\ref{s_exp2}).

\subsubsection{The matrix $\Amat$ in the linear system~(\ref{Eq:Axb})}\label{s_exp3}
The symmetric positive definite  matrix $\Amat\in\Rnn$ of dimension $n = 1806$
is a preconditoned version of the matrix BCSSTK14 from the Harwell-Boeing collection in \cite{bcsstk14}. Specifically,
\begin{equation*}
    \Amat = \Dmat^{-1/2}\Bmat\Dmat^{-1/2},\qquad \text{where} \quad
\Dmat\equiv \diag\begin{pmatrix}\Bmat_{11}& \cdots & \Bmat_{nn}\end{pmatrix}  
\end{equation*}
and $\Bmat$ is BCSSTK14.
Calibration is assessed at iterations $m = 10, 100, 300$.

\subsubsection{$Z$-statistic and $S$-statistic}\label{s_exp1}
The $Z$-statistic and $S$-statistic experiments are implemented as described in Algorithms \ref{A:ZStat} and~\ref{A:SStat}, respectively.
The calibration criteria for the $Z$-statistic are given in Remark~\ref{r_zeval}, and for the $S$-statistic in Remark~\ref{r_seval}.

We sample from Gaussian distributions 
by exploiting their stability. According to \lref{L:GaussStability} in Appendix~\ref{S:Aux},
if $\Zrv\sim\N(\zerovec,\Imat)$, and $\Fmat\Fmat^T = \Sigmat$ is a factorization of the covariance, then 
\begin{equation*}
    \Fmat\Zrv + \zvec = \Xrv \sim \N(\xvec,\Sigmat).
\end{equation*}
Samples $\Zrv\sim\N(\zerovec,\Imat)$ are generated with \texttt{randn}($n,1$)
in Matlab, and with \texttt{numpy.random.randn}($n,1$) in NumPy.

\paragraph{$Z$-statistic experiments.}
 We quantify the distance between the $Z$-statistic samples and the chi-squared distribution by applying the Kolmogorov-Smirnov statistic (\dref{D:KSTest})
to the empirical cumulative distribution function of the $Z$-statistic samples and the analytical cumulative distribution function of the chi-squared distribution.

The degree of freedom in the chi-squared distribution is chosen as the median numerical rank of the posterior covariances. Note that the numerical rank of $\Sigmat_m$ can differ from
\begin{equation*}
  \rank(\Sigmat_m) = \rank(\Sigmat_0) - m,
\end{equation*}
and choosing the median rank gives an integer equal to the rank of at least one posterior covariance.

In compliance with the Matlab function \texttt{rank} and the NumPy function \texttt{numpy.linalg.rank}, we compute the numerical rank of $\Sigmat_m$ as
\begin{equation}
\label{Eq:NumRank}
  \rank(\Sigmat_m) = \mathrm{cardinality} \{\sigma_i \ | \ \sigma_i > n\varepsilon \|\Sigmat_m\|_2\},
\end{equation}
where $\varepsilon$ is machine epsilon and $\sigma_i$, $1\leq i \leq n$, are the singular values of $\Sigmat_m$ \cite[Section 5.4.1]{GoVa13}. 

\subsubsection{Three BayesCG implementations}\label{s_exp2}

We use three versions of BayesCG: BayesCG with random search directions, BayesCG under the inverse prior, and BayesCG under the Krylov prior.

\paragraph{BayesCG with random search directions.}
The implementation in \aref{A:BayesCGC} in \appref{S:BayesCGC} computes
posterior covariances that do not depend on the solution $\xvec_*$.
This, in turn, requires search directions that do not depend on $\xvec_*$ \cite[Section 1.1]{CIOR20} which is achieved by starting with a random 
search direction $\svec_1=\uvec\sim\N(\zerovec,\Imat)$ instead of the initial residual
$\rvec_0\equiv\bvec_0-\Amat\xvec_0$.
The prior is $\N(\zerovec,\Amat^{-1})$.
 
 By design, this version of BayesCG is
 calibrated. However, it is also impractical due to its slow convergence,
see \figref{F:Convergence},
and an accurate solution is available only after
$n$ iterations. The random initial search direction~$\svec_1$ leads to uninformative $m$-dimensional subspaces,
so that the solver has to explore all of $\Rn$ before finding the solution.

\begin{figure}
  \centering
  \includegraphics[scale = .35]{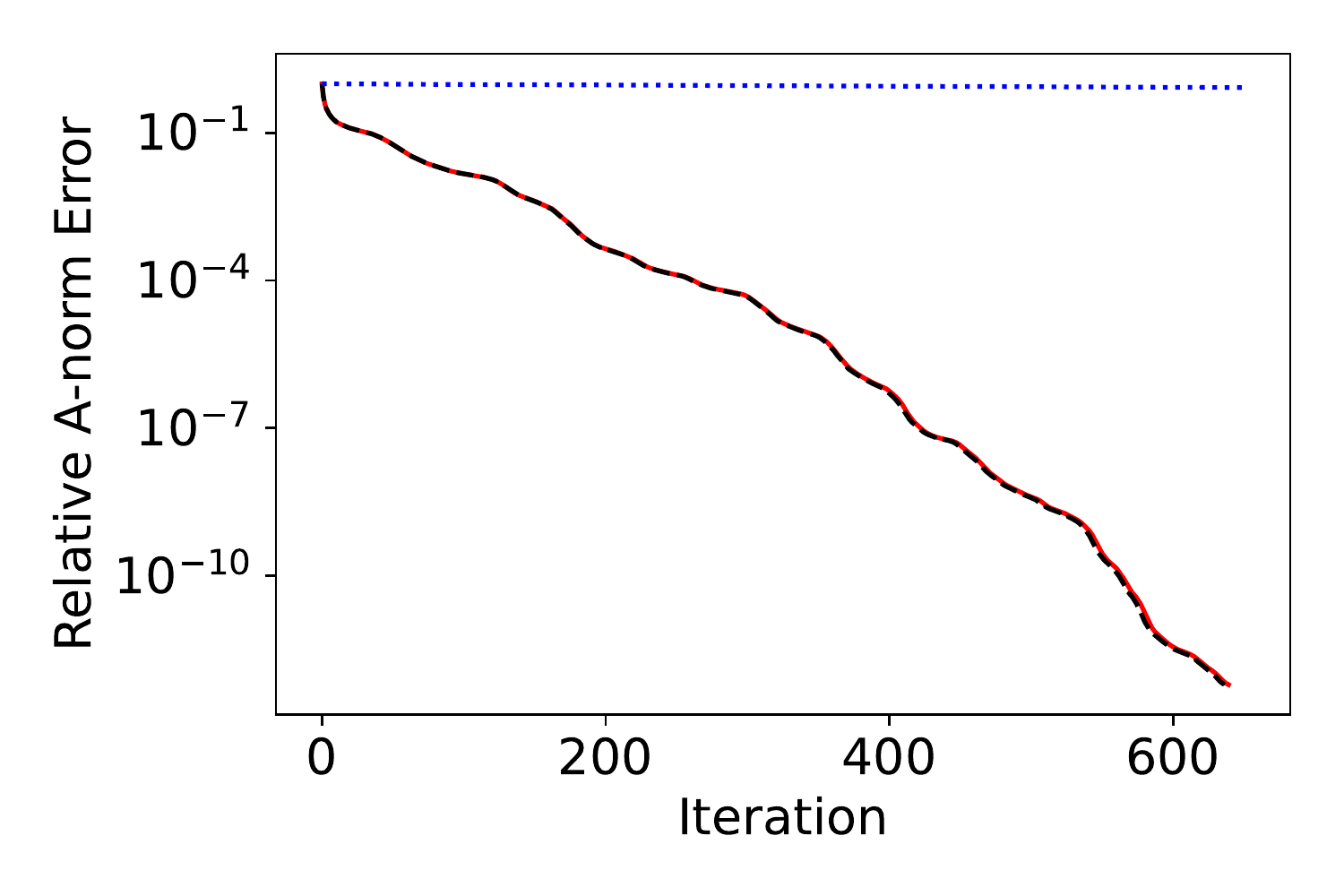}
  \includegraphics[scale = .35]{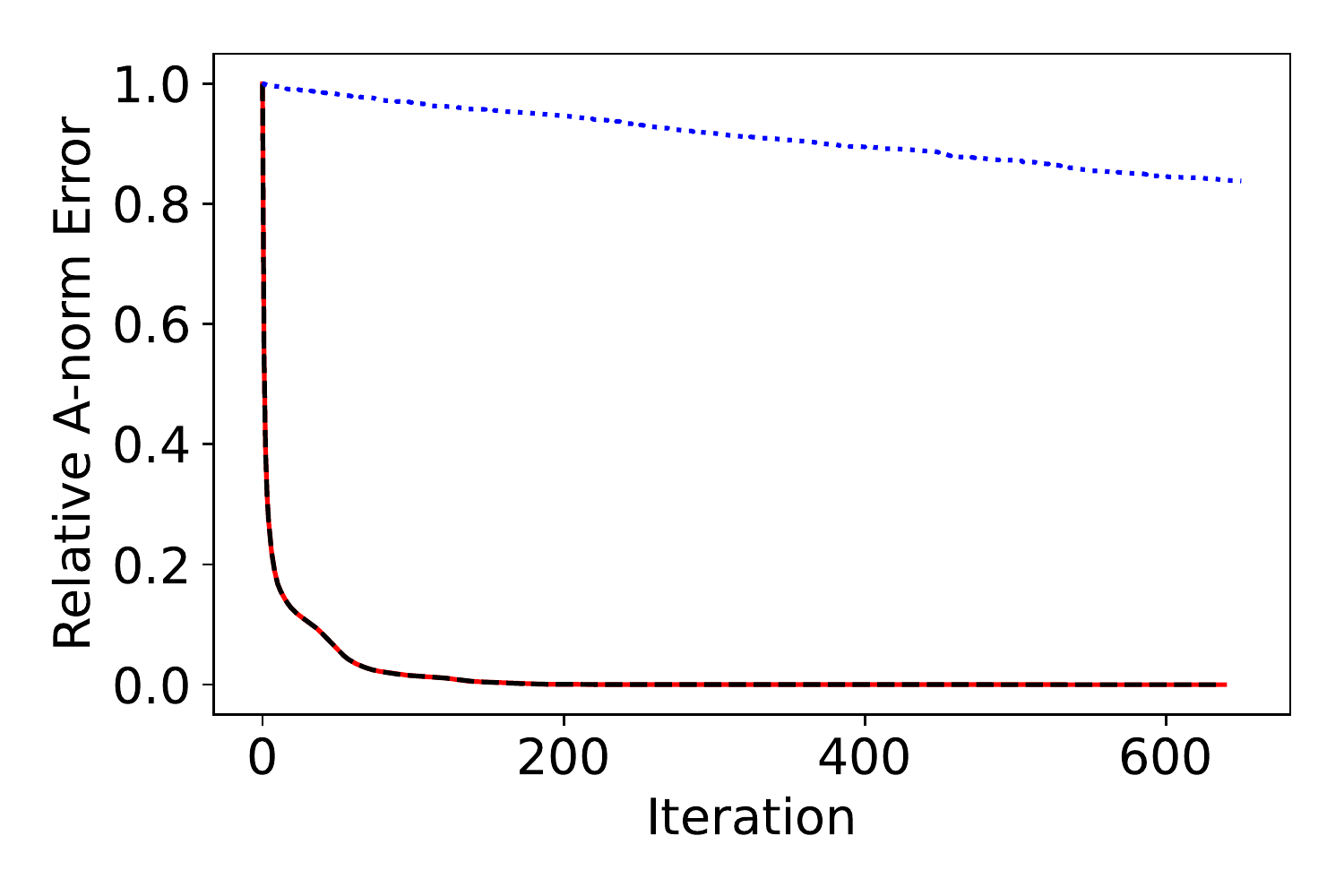}
  \caption{Relative error $\|\xvec_*-\xvec_m\|_\Amat^2/\|\xvec_*\|_\Amat^2$ for BayesCG under the inverse prior (solid line) and Krylov prior (dashed line), and  BayesCG with random search directions (dotted line). Vertical axis has a logarithmic scale (left plot) and a linear scale (right plot).}
  \label{F:Convergence}
\end{figure}

\paragraph{BayesCG under the inverse prior $\mu_0 \equiv \N(\zerovec,\Amat^{-1})$.}
The implementation in Algorithm~\ref{A:BayesCGF} in \appref{S:BayesCGF} 
is a modified version of \aref{A:BayesCG} for general priors that maintains the posterior covariances in factored form. 

\paragraph{BayesCG under the Krylov prior.}
For full posteriors, the modified Lanczos solver \aref{A:ALanczos} in  \appref{S:BayesCGK}
computes the full prior, which is then followed by
the direct computation of the posteriors
from the prior in \aref{A:BayesCGC}.

For approximate posteriors, \aref{A:BayesCGWithoutBayesCG}
in \appref{S:BayesCGK} computes rank-$d$ covariances at the same computational cost as $m+d$ iterations of CG.

In $Z$- and $S$-statistic experiments, solutions $\xvec_*$ are usually sampled from the prior distribution. We cannot sample solutions from the Krylov prior because it differs from solution to solution. Instead we sample solutions from the reference
distribution $\N(\zerovec,\Amat^{-1})$. This is a reasonable choice because 
the posterior means in
BayesCG under the inverse and Krylov priors coincide with the 
CG iterates \cite[Section 3]{Cockayne:BCG}.

\subsection{BayesCG with random search directions}
\label{S:ExpRand}

By design, BayesCG with random search directions is a calibrated solver.
The purpose is to establish a baseline for comparisons with BayesCG under the inverse and Krylov priors, and to demonstrate that the $Z$- and $S$-statistics perform as expected on a calibrated solver. 

\paragraph{Summary of experiments below.} 
Both, $Z$- and $S$-statistics strongly
confirm that BayesCG with random search directions is indeed a calibrated solver, thereby corroborating the statements in
\tref{T:ZStat} and \dref{D:SStat}. 

\begin{figure}
  \centering
  \includegraphics[scale = .35]{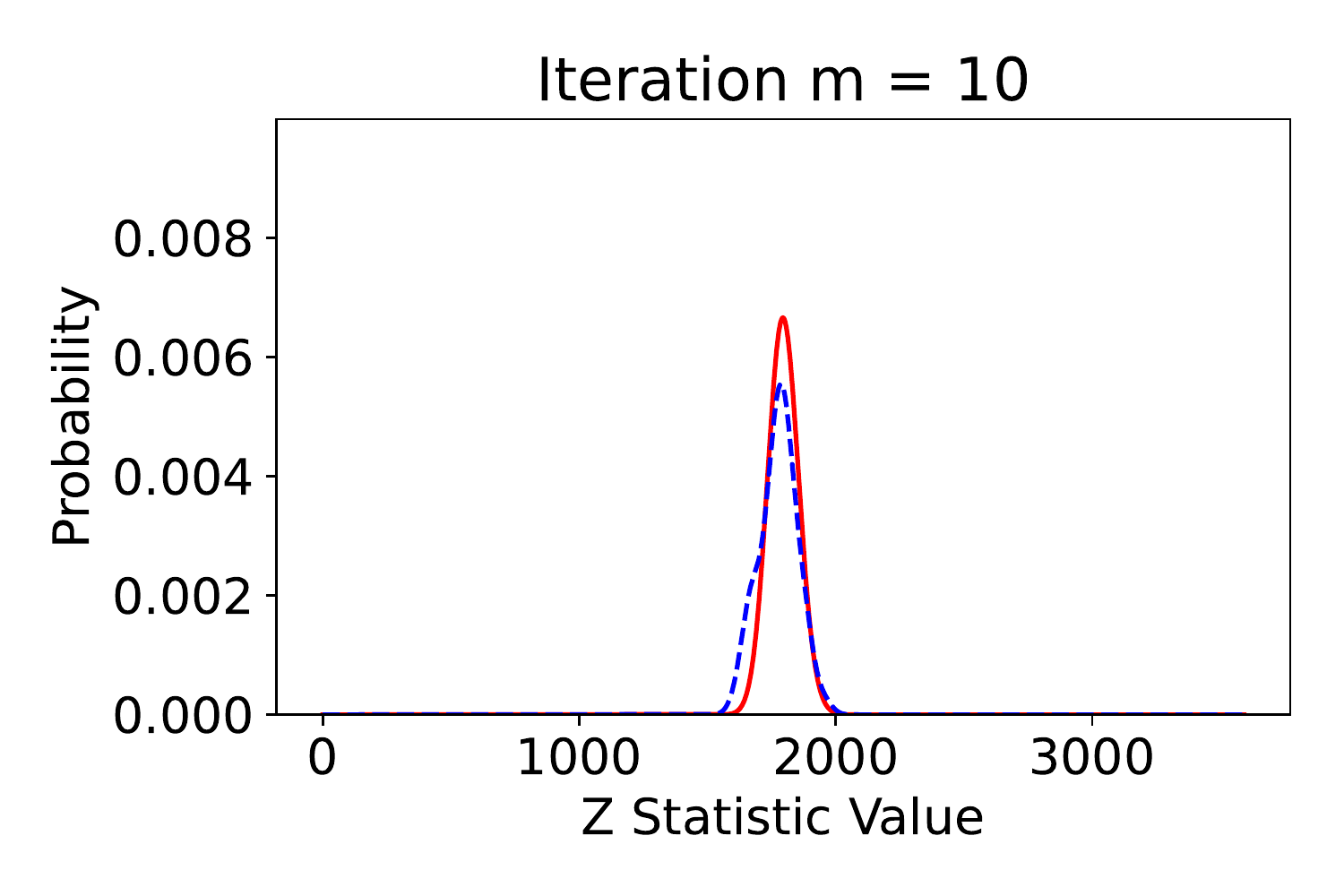}
  \includegraphics[scale = .35]{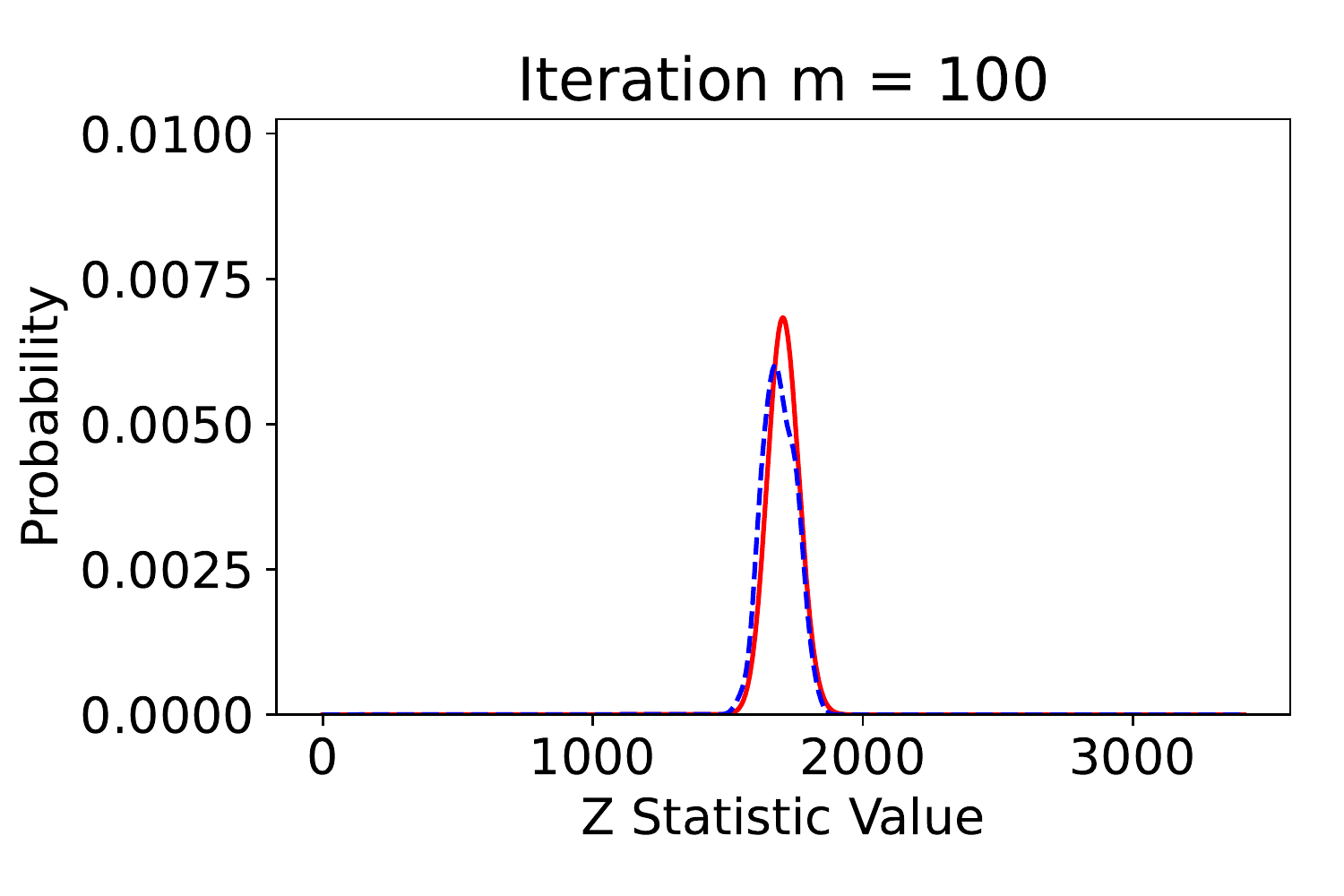} \\
  \includegraphics[scale = .35]{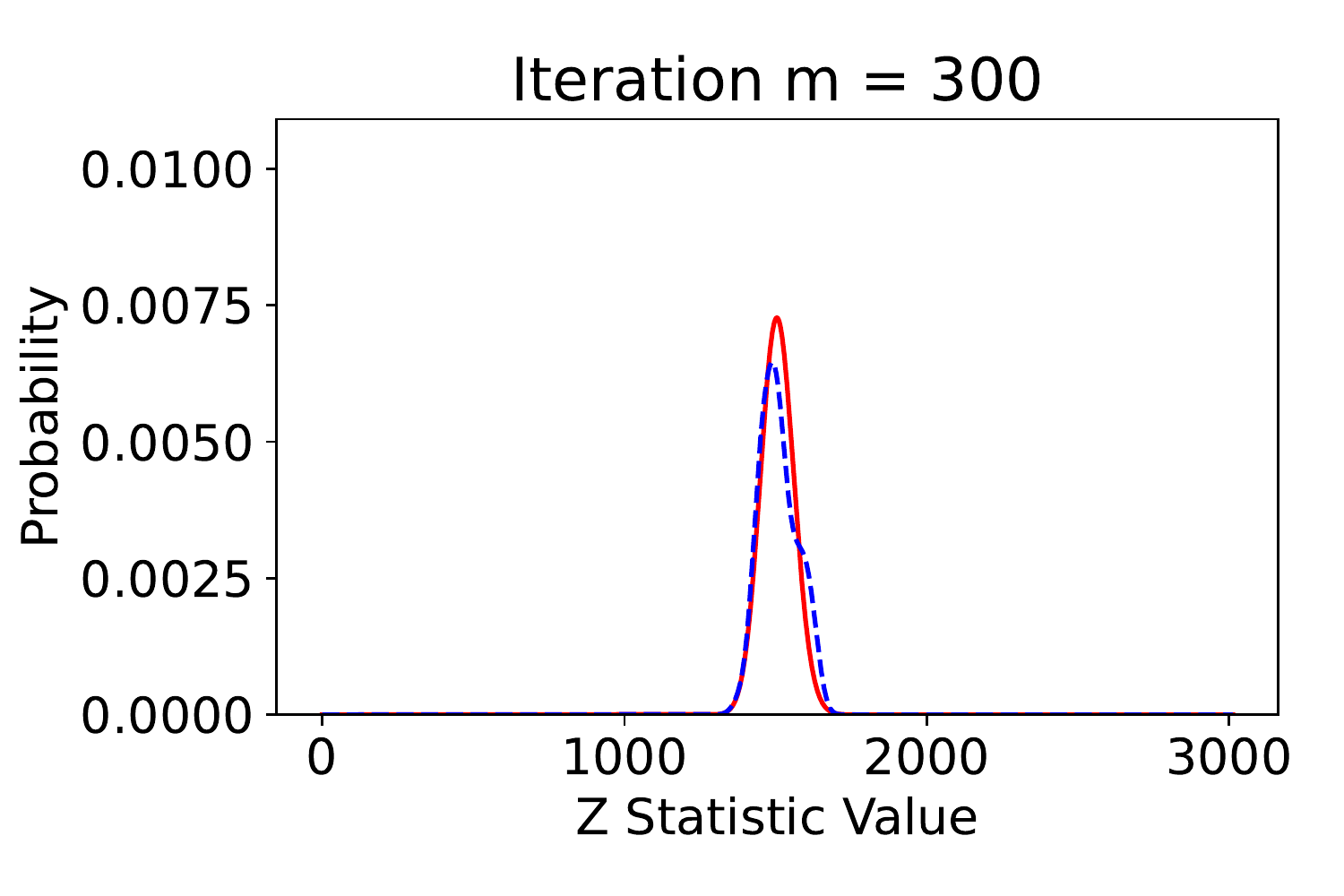}
  \caption{$Z$-statistic samples for BayesCG with random search directions after $m=10, 100, 300$ iterations. The solid curve represents the chi-squared distribution and the dashed curve
  the $Z$-statistic samples.}
  \label{F:ZRand}
\end{figure}

\begin{table}
    \centering
    \begin{tabular}{c|ccc}
    Iteration & $Z$-stat mean & $\chi^2$ mean & K-S statistic \\\hline
         $ 10.0 $ & $ 1.79 \times 10^{3} $ & $ 1.8 \times 10^{3} $ & $ 0.139 $ \\
$ 100.0 $ & $ 1.69 \times 10^{3} $ & $ 1.71 \times 10^{3} $ & $ 0.161 $ \\
$ 300.0 $ & $ 1.5 \times 10^{3} $ & $ 1.51 \times 10^{3} $ & $ 9.65 \times 10^{-2} $ \\
    \end{tabular}
    \caption{This table corresponds to \figref{F:ZRand}. For BayesCG with random search directions, it shows the
    $Z$-statistic sample means; the chi-squared distribution means; and the Kolmogorov-Smirnov statistic between the $Z$-statistic samples and the chi-squared distribution.}
    \label{Tab:ZRand}
\end{table}

\begin{figure}
  \centering
  \includegraphics[scale = .35]{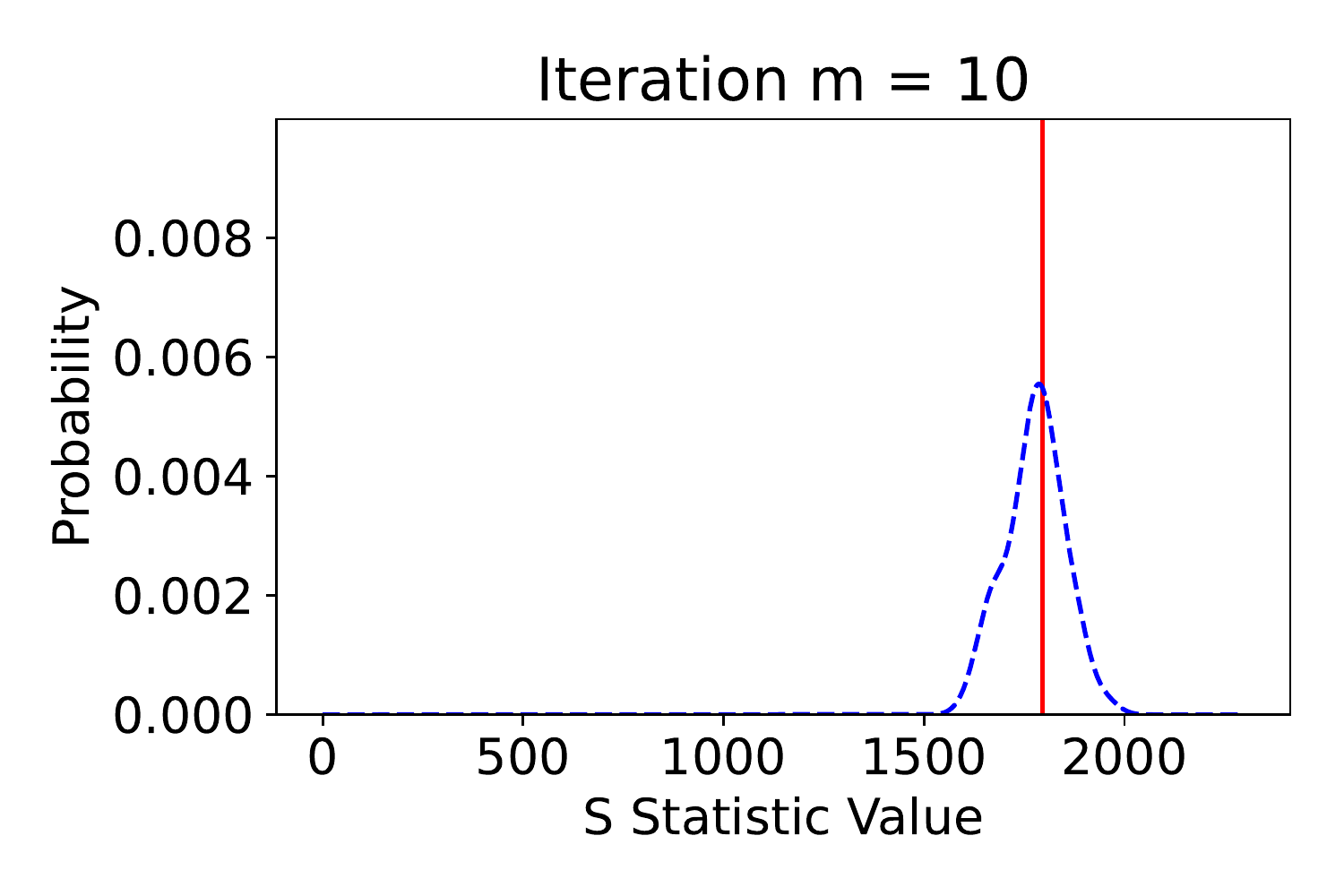} 
  \includegraphics[scale = .35]{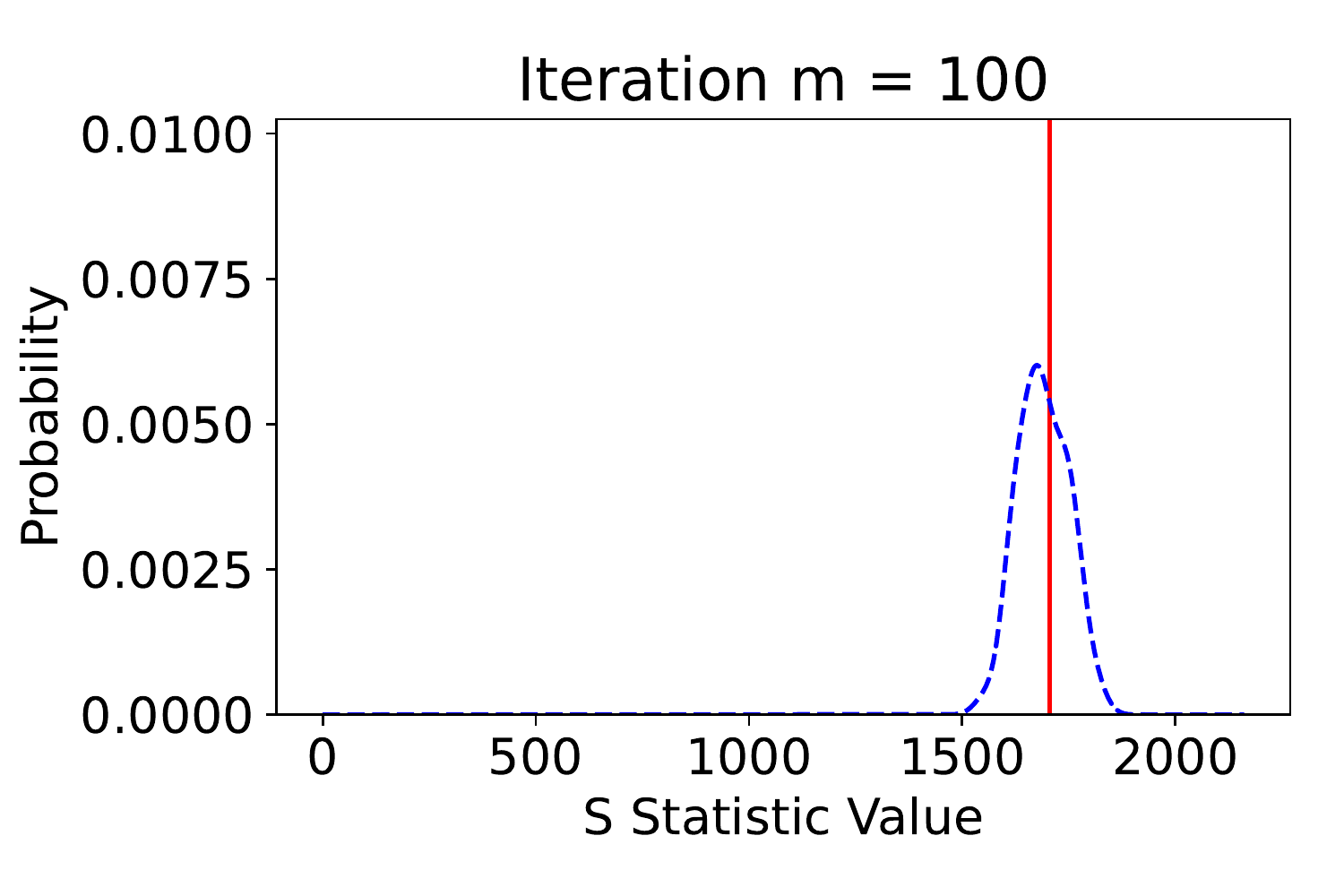} \\
  \includegraphics[scale = .35]{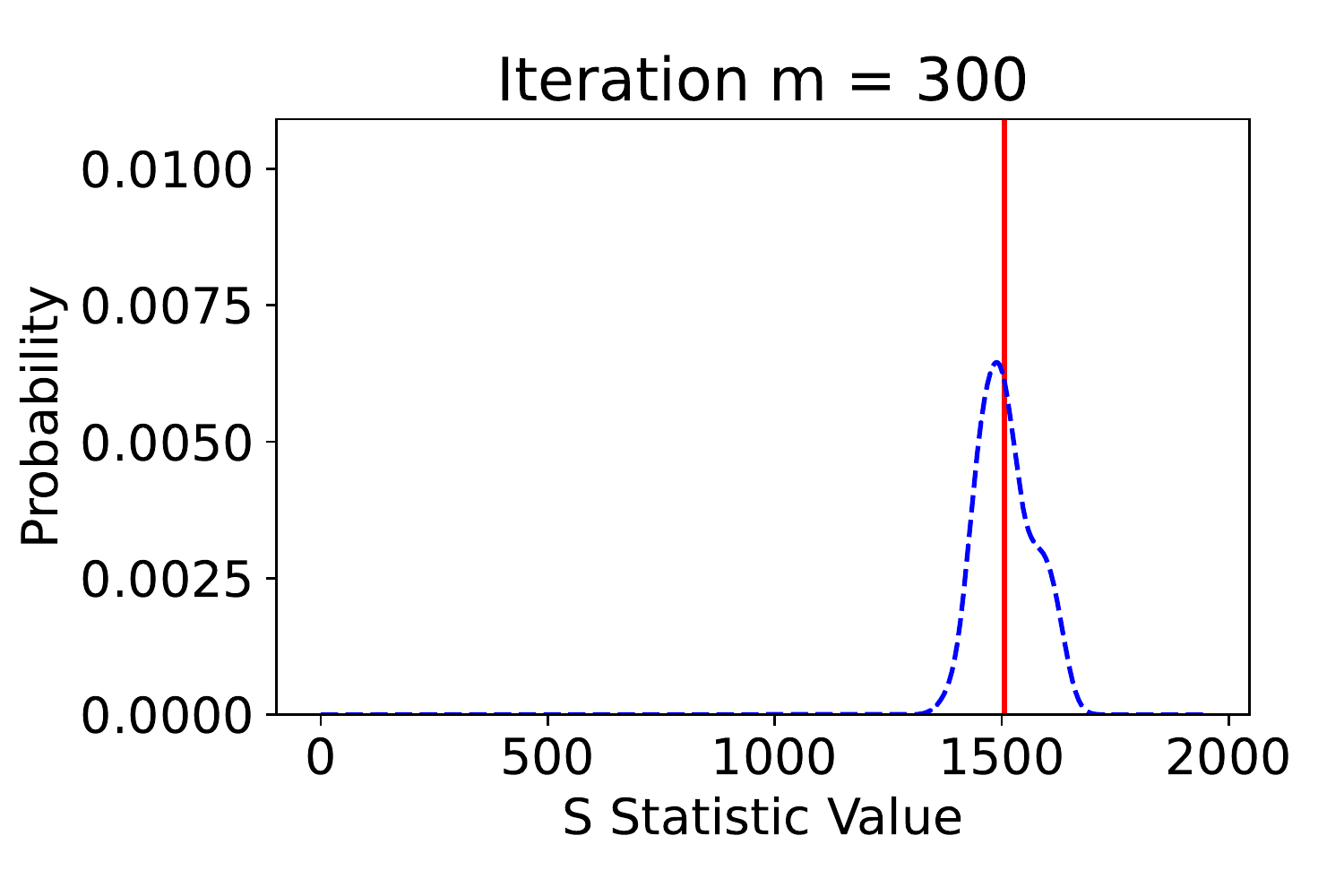}
  \caption{$S$-statistic samples and traces for BayesCG with random search directions
 after $m=10, 100, 300$ iterations. The solid curve represents the traces and the dashed curve the $S$-statistic samples.}
  \label{F:SRand}
\end{figure}

\begin{table}
    \centering
    \begin{tabular}{c|ccc}
    Iteration & $S$-stat mean & Trace mean & Trace standard deviation \\\hline
         $ 10.0 $ & $ 1.78 \times 10^{3} $ & $ 1.8 \times 10^{3} $ & $ 2.93 \times 10^{-12} $ \\
$ 100.0 $ & $ 1.69 \times 10^{3} $ & $ 1.71 \times 10^{3} $ & $ 2.29 \times 10^{-12} $ \\
$ 300.0 $ & $ 1.51 \times 10^{3} $ & $ 1.51 \times 10^{3} $ & $ 2.07 \times 10^{-12} $ \\
    \end{tabular}
    \caption{This table corresponds to \figref{F:SRand}. For BayesCG with random search directions, it shows the
    $S$-statistic sample means, the trace  means, and the trace standard deviations.}
    \label{Tab:SRand}
\end{table}

\paragraph{\figref{F:ZRand} and \tabref{Tab:ZRand}.}
The $Z$-statistic samples in \figref{F:ZRand} almost match
 the chi-squared distribution;
 and the Kolmogorov-Smirnov statistics in 
 \tabref{Tab:ZRand}
 are on the order of $10^{-1}$, meaning close to zero. This confirms that BayesCG with random search directions is indeed calibrated. 

\paragraph{\figref{F:SRand} and \tabref{Tab:SRand}.}
The traces in \figref{F:SRand} 
are tightly concentrated around the empirical mean of the $S$-statistic samples. \tabref{Tab:SRand} confirms
the strong clustering of the trace
and $S$-statistic sample means around $10^{-3}$, together with the very small deviation of the traces.
Thus, the area in which the posteriors are concentrated is consistent with the error, 
confirming again that BayesCG with random search directions is calibrated.

\subsection{BayesCG under the inverse prior $\mu_0 = \N(\zerovec,\Amat^{-1})$.}
\label{S:ExpInv}

\paragraph{Summary of experiments below.}
Both, $Z$- and $S$-statics indicate
that BayesCG under the inverse prior is pessimistic, and that the
pessimism increases with the iteration count. This is consistent with the experiments in \cite[Section 6.1]{Cockayne:BCG}.

\begin{figure}
  \centering
  \includegraphics[scale = .35]{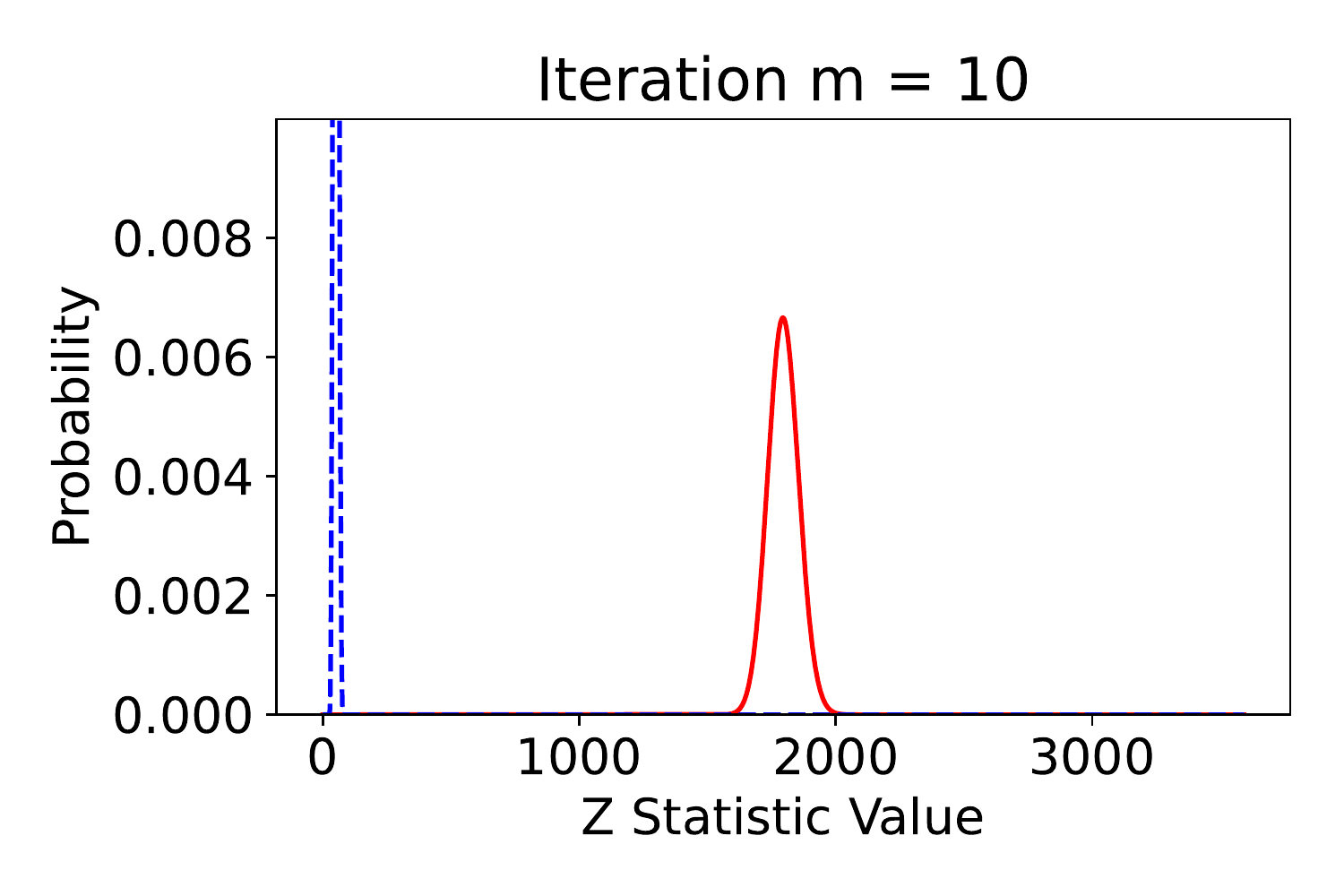}
  \includegraphics[scale = .35]{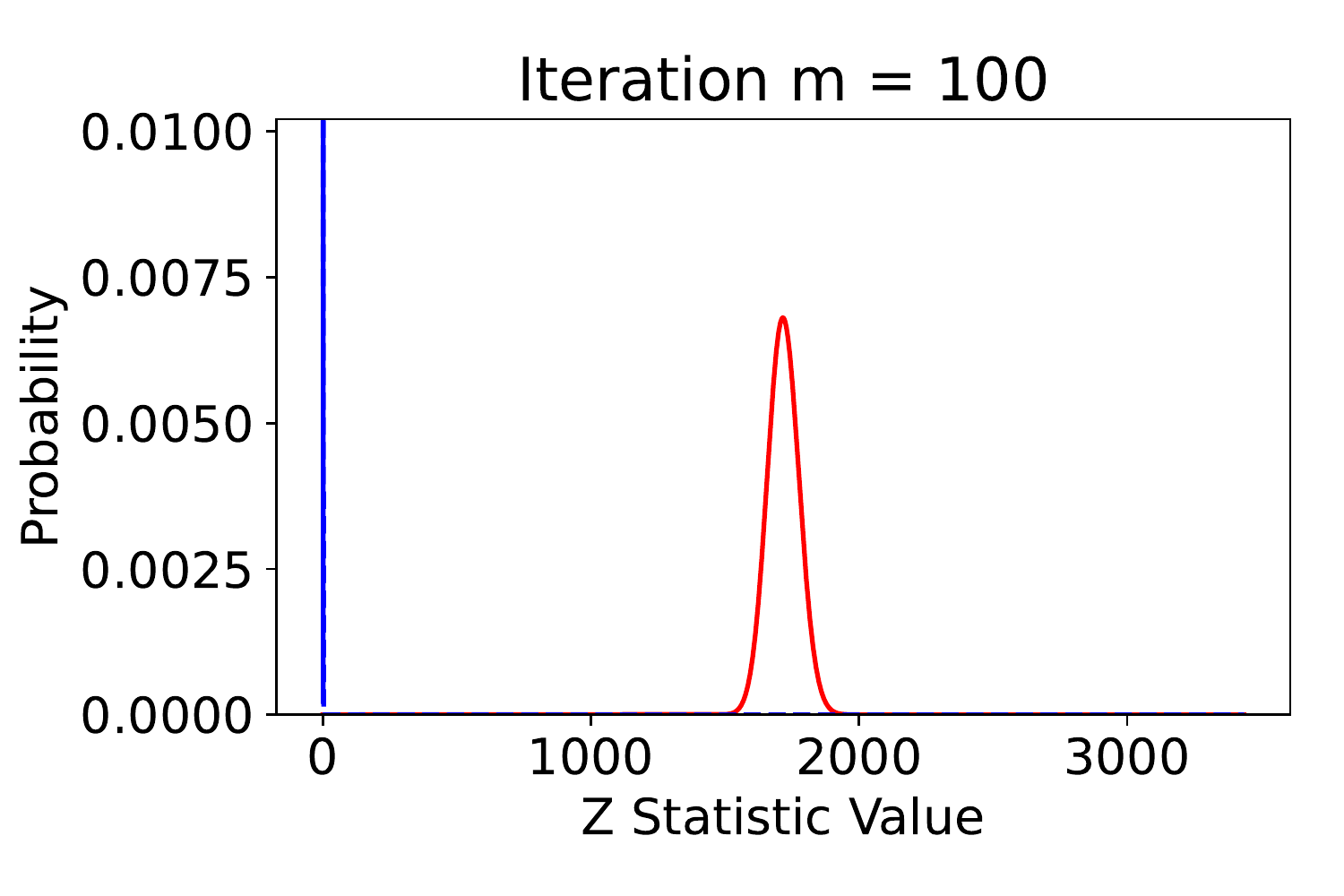} \\
  \includegraphics[scale = .35]{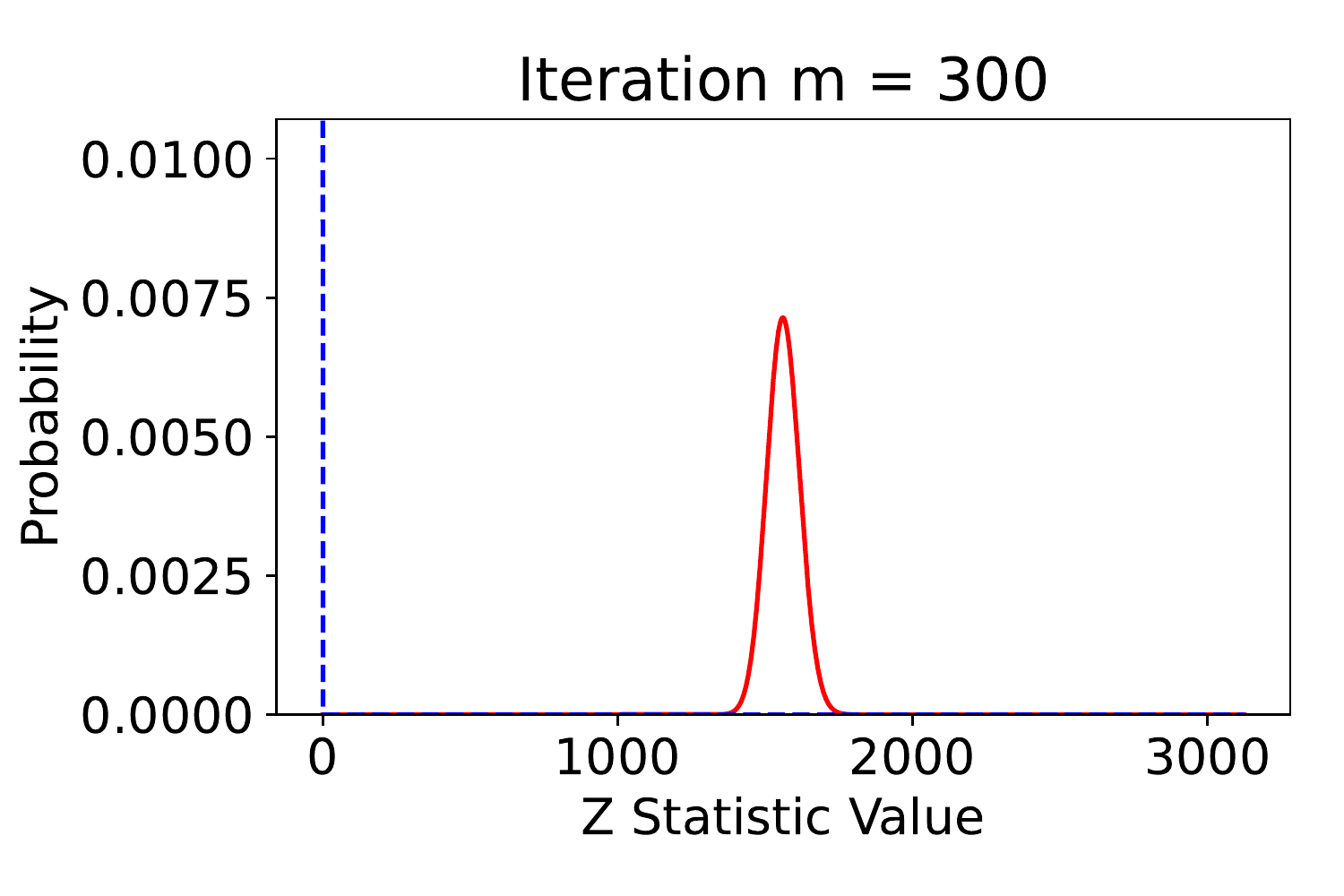}
  \caption{$Z$-statistic samples for BayesCG under the inverse prior after $m=10, 100, 300$ iterations. The solid curve represents the chi-squared distribution and the dashed curve the  $Z$-statistic samples.}
  \label{F:ZInv}
\end{figure}

\begin{table}
    \centering
    \begin{tabular}{c|ccc}
    Iteration  & $Z$-stat mean & $\chi^2$ mean & K-S statistic \\\hline
         $ 10.0 $ & $ 51.9 $ & $ 1.8 \times 10^{3} $ & $ 1.0 $ \\
$ 100.0 $ & $ 0.545 $ & $ 1.72 \times 10^{3} $ & $ 1.0 $ \\
$ 300.0 $ & $ 1.33 \times 10^{-5} $ & $ 1.56 \times 10^{3} $ & $ 1.0 $ \\
    \end{tabular}
    \caption{This table corresponds to \figref{F:ZInv}. For BayesCG under the inverse prior, it shows
    the $Z$-statistic sample means; the chi-squared distribution means; and Kolmogorov-Smirnov statistic between the $Z$-statistic samples and the chi-squared and distribution. }
    \label{Tab:ZInv}
\end{table}

\begin{figure}
  \centering
  \includegraphics[scale = .35]{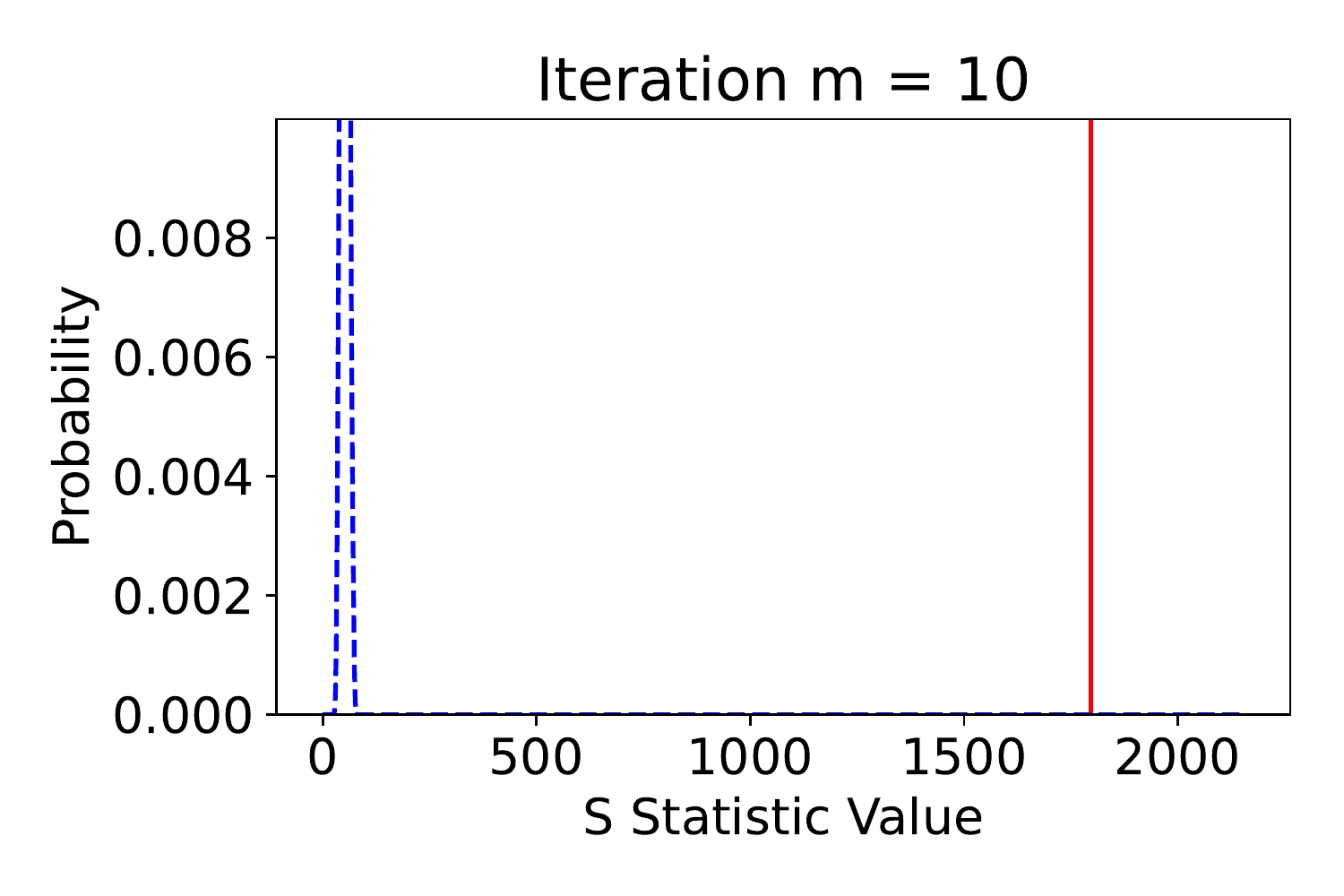} 
  \includegraphics[scale = .35]{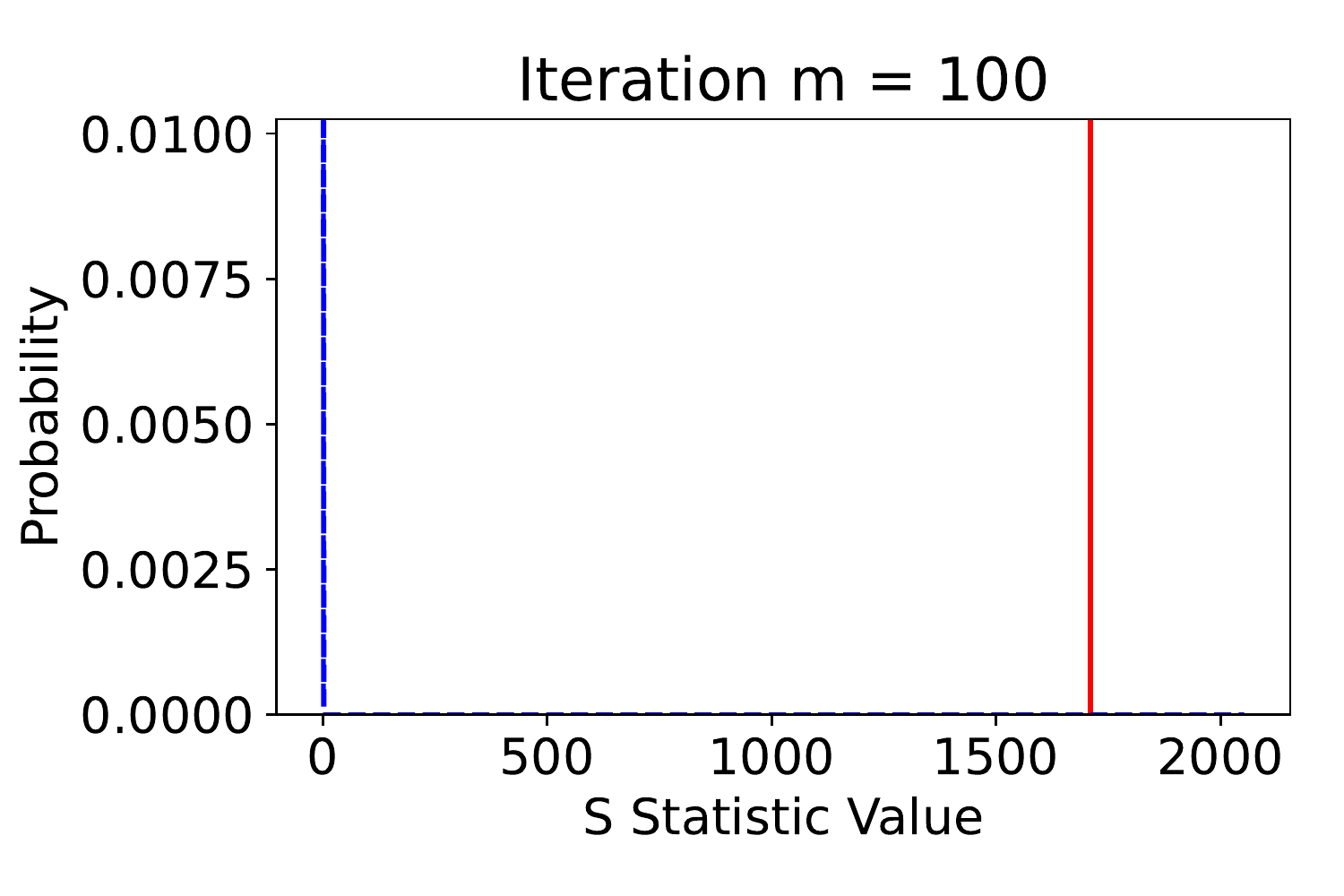} \\
  \includegraphics[scale = .35]{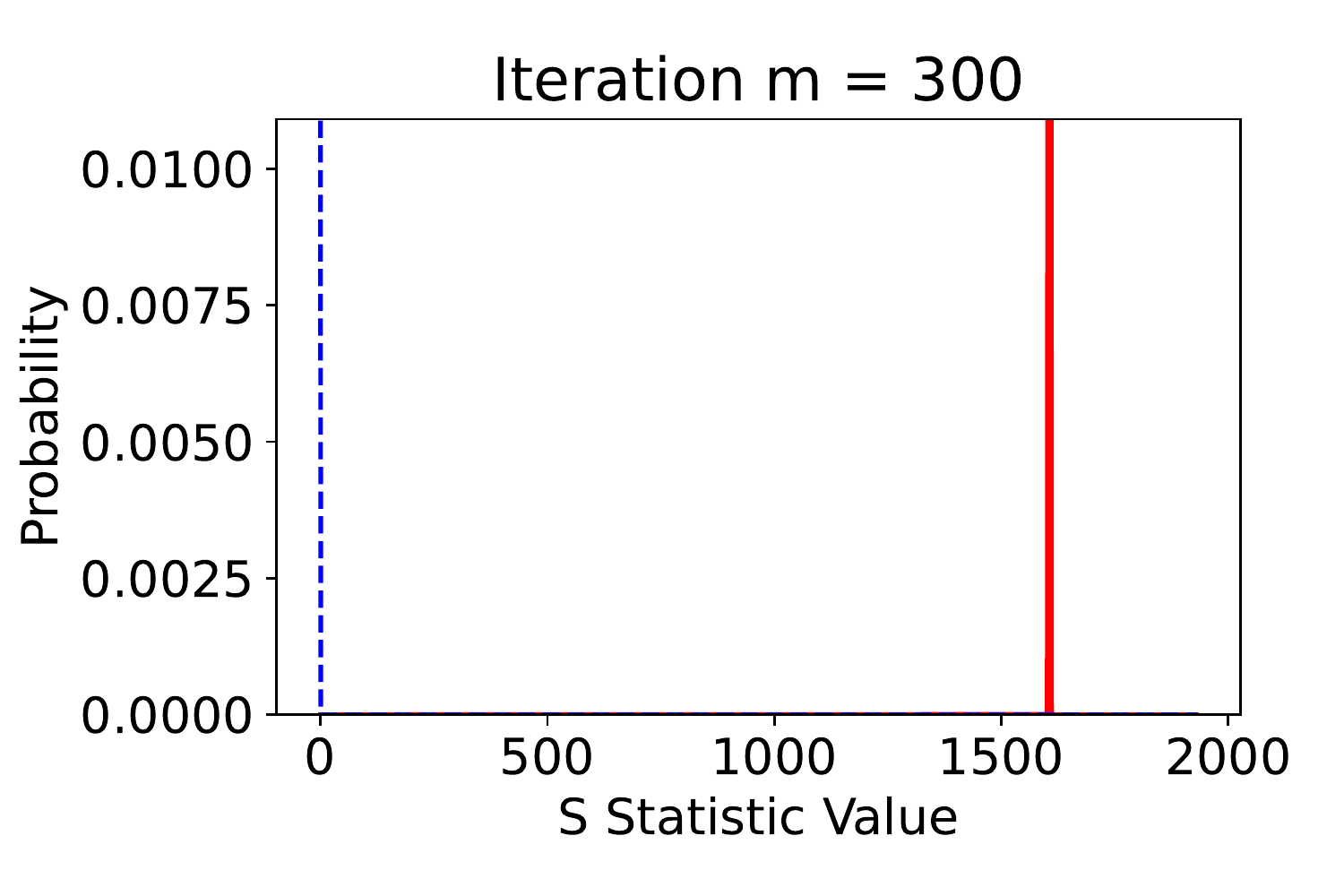} 
  \caption{$S$-statistic samples and traces for BayesCG under the inverse prior after $m=10, 100, 300$ iterations. The solid curve represents the traces and the dashed curve the $S$-statistic samples.}
  \label{F:SInv}
\end{figure}

\begin{table}
    \centering
    \begin{tabular}{c|ccc}
    Iteration & $S$-stat mean & Trace mean & Trace standard deviation \\\hline
$ 10.0 $ & $ 51.8 $ & $ 1.8 \times 10^{3} $ & $ 2.99 \times 10^{-12} $ \\
$ 100.0 $ & $ 0.574 $ & $ 1.71 \times 10^{3} $ & $ 0.164 $ \\
$ 300.0 $ & $ 3.57 \times 10^{-6} $ & $ 1.61 \times 10^{3} $ & $ 1.02 $ \\
    \end{tabular}
    \caption{This table corresponds to \figref{F:SInv}. For BayesCG under the inverse prior,
    it shows the $S$-statistic sample means, the trace means, and the trace standard deviations.}
    \label{Tab:SInv}
\end{table}

\paragraph{\figref{F:ZInv} and \tabref{Tab:ZInv}.}
 The  $Z$-statistic samples in 
 \figref{F:ZInv}
are concentrated around smaller values than the predicted chi-squared distribution. 
The Kol\--mogorov-Smirnov statistics in \tabref{Tab:ZInv}
are all equal to 1,
indicating no overlap between $Z$-statistic samples and chi-squared distribution. The first two columns of
\tabref{Tab:ZInv} show that
$Z$-statistic samples move further away from the chi-squared distribution as the iterations progress.  
Thus, BayesCG under the inverse prior is pessimistic, and the
pessimism increases with the iteration count. 

\paragraph{\figref{F:SInv} and \tabref{Tab:SInv}.}
The $S$-statistic samples in \figref{F:SInv}
are concentrated around smaller values than the traces.
\tabref{Tab:SInv} indicates trace values at $10^3$, while the $S$-statistic samples move towards zero as the iteration progress.
Thus the errors are much smaller than the area in which the posteriors are concentrated, meaning the posteriors overestimate the error. This again confirms the pessimism of BayesCG under the inverse prior.

\subsection{BayesCG under the Krylov prior}
\label{S:ExpKry}
We consider full posteriors (Section~\ref{S:ExpFull}),
and then rank-50 approximate posteriors (Section~\ref{S:ExpApprox}).

\subsubsection{Full Krylov posteriors}
\label{S:ExpFull}

\paragraph{Summary of experiments below.}
The $Z$-statistic indicates that BayesCG under full Krylov posteriors is somewhat optimistic,
while the $S$-statistic indicates resemblance to a calibrated solver.

\begin{figure}
  \centering
  \includegraphics[scale = .35]{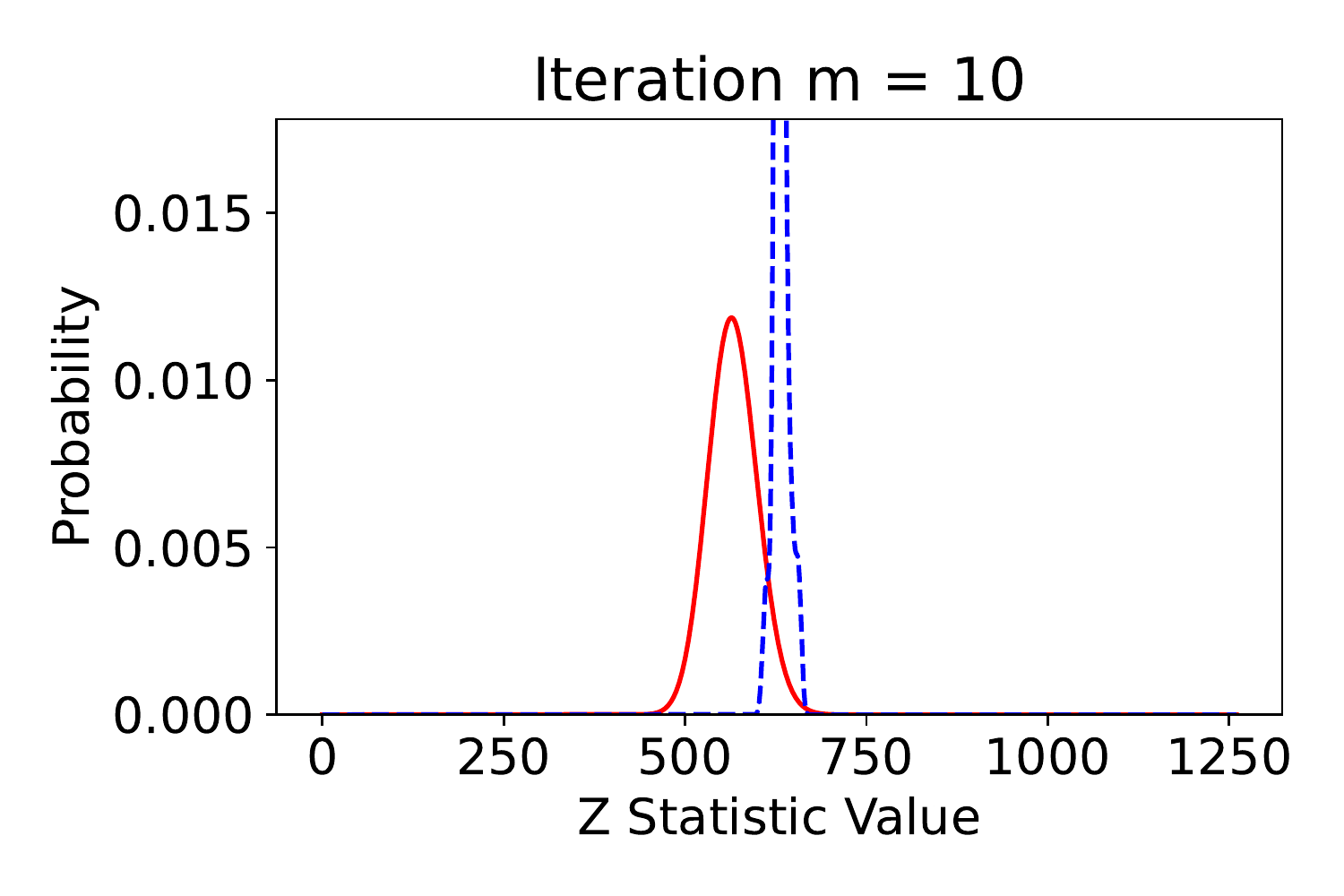}
  \includegraphics[scale = .35]{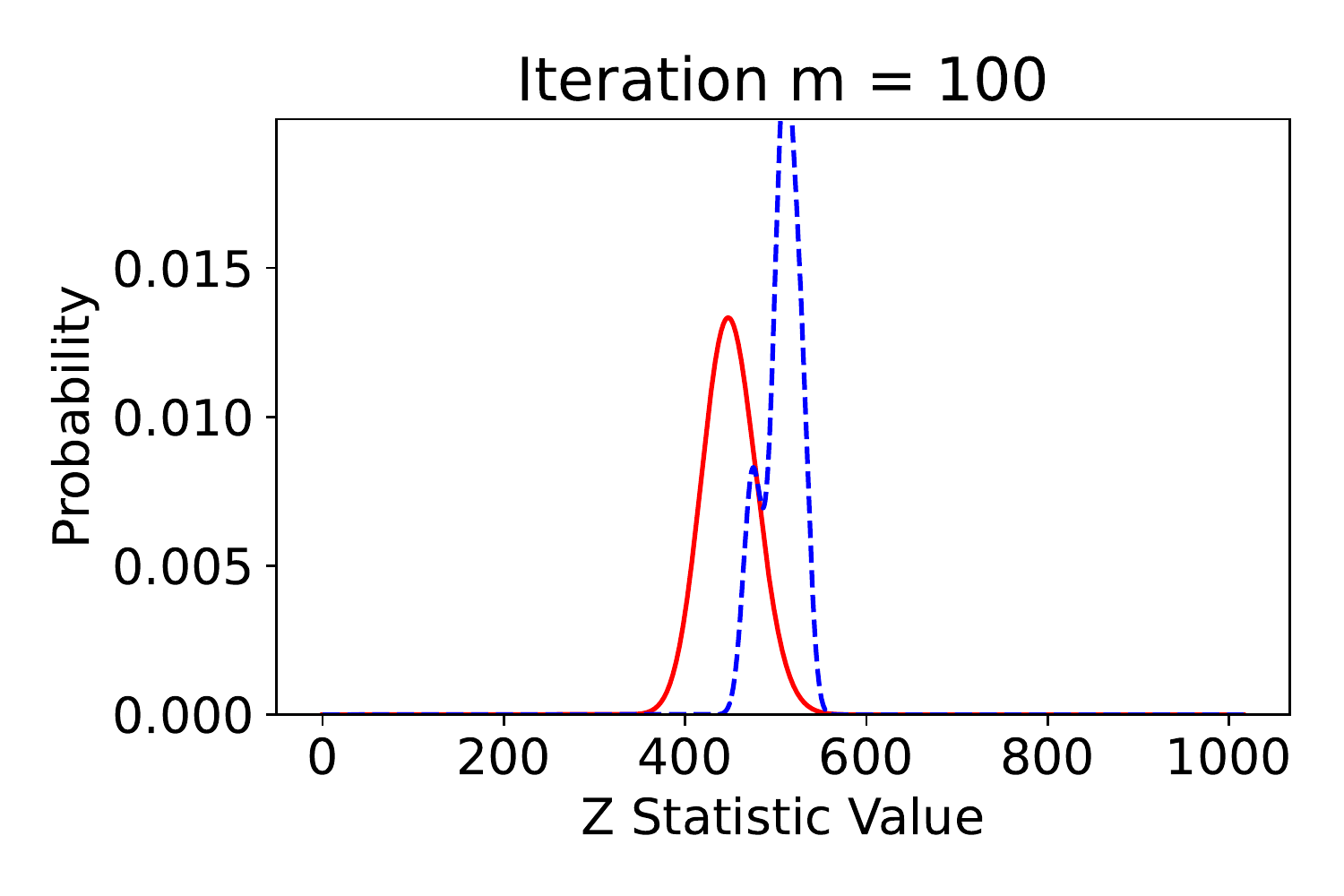} \\
  \includegraphics[scale = .35]{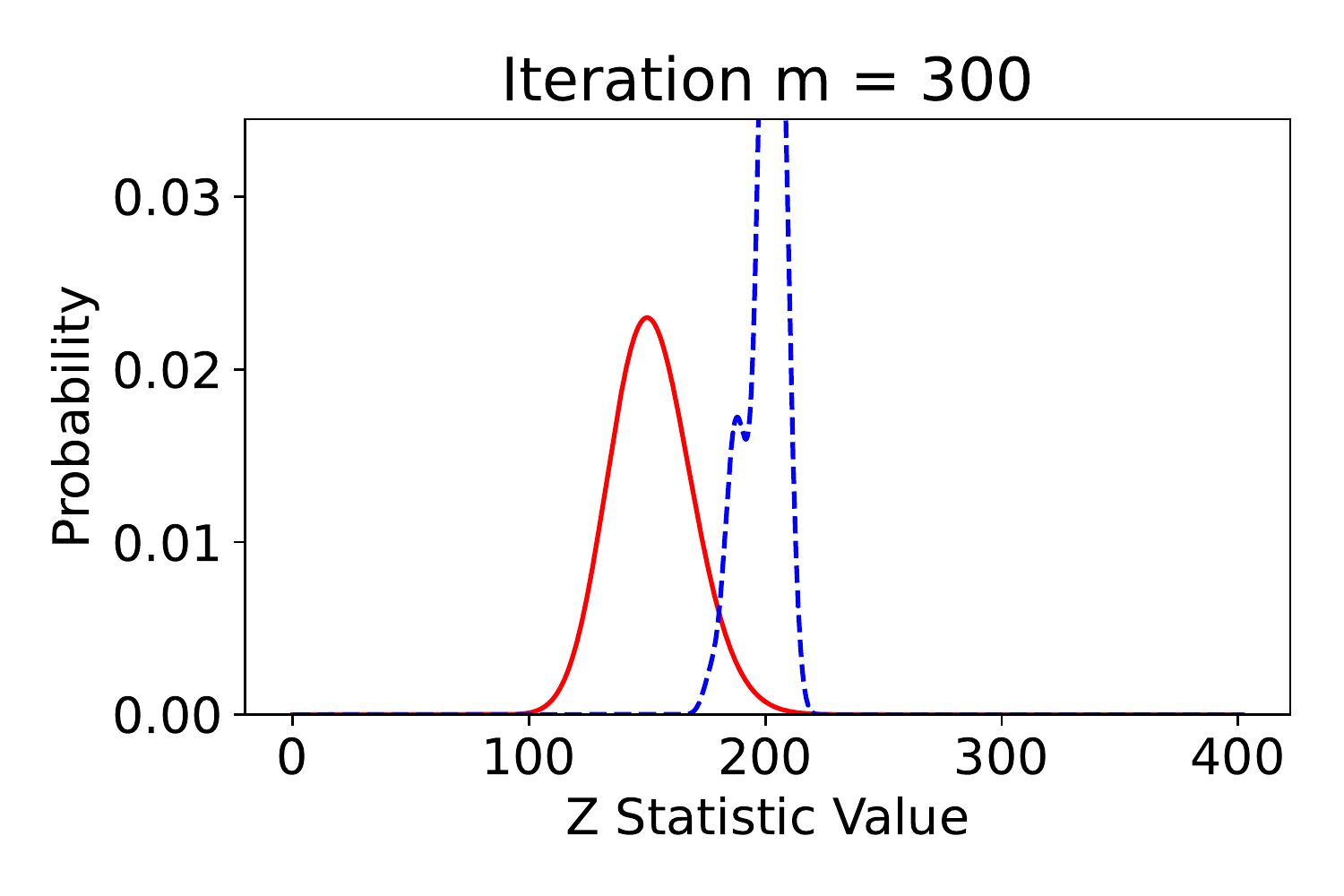}
  \caption{$Z$-statistic samples for BayesCG under the Krylov prior and full  posteriors at $m=10, 100, 300$ iterations. The solid curve represents the predicted chi-squared distribution and the dashed curve the $Z$-statistic samples.}
  \label{F:ZKrylov}
\end{figure}

\begin{table}
    \centering
    \begin{tabular}{c|ccc}
    Iteration & $Z$-stat mean & $\chi^2$ mean & K-S statistic \\\hline
$ 10.0 $ & $ 631.0 $ & $ 566.0 $ & $ 0.902 $ \\
$ 100.0 $ & $ 509.0 $ & $ 450.0 $ & $ 0.752 $ \\
$ 300.0 $ & $ 201.0 $ & $ 152.0 $ & $ 0.941 $ \\
    \end{tabular}
    \caption{This table corresponds to \figref{F:ZKrylov}. For BayesCG under the Krylov prior and full posteriors, it shows the
    $Z$-statistic sample means; the chi-squared distribution means; and the Kolmogorov-Smirnov statistic between the $Z$-statistic samples and the chi-squared distribution.}
    \label{Tab:ZKrylov}
\end{table}

\begin{figure}
  \centering
  \includegraphics[scale = .35]{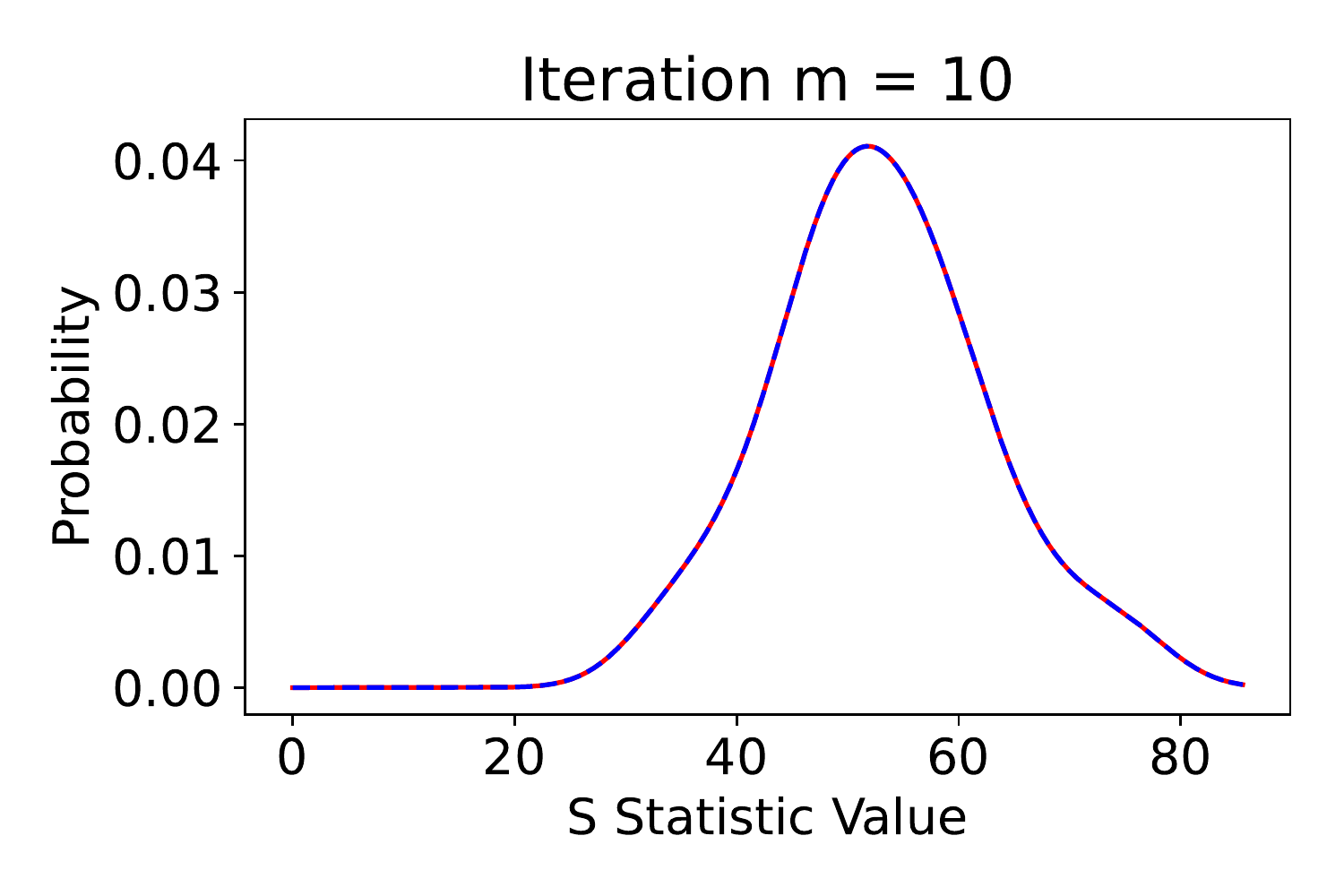} 
  \includegraphics[scale = .35]{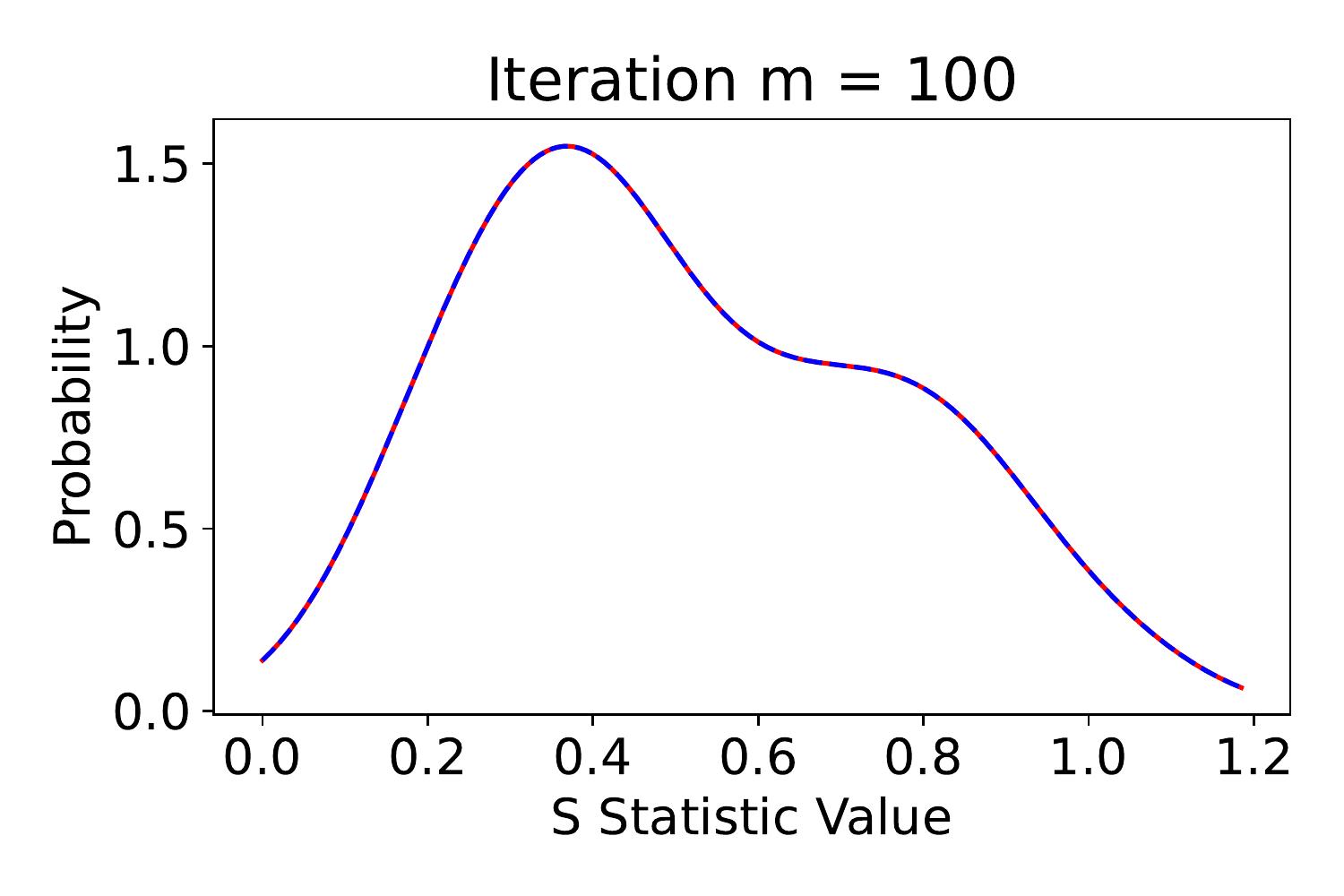} \\
  \includegraphics[scale = .35]{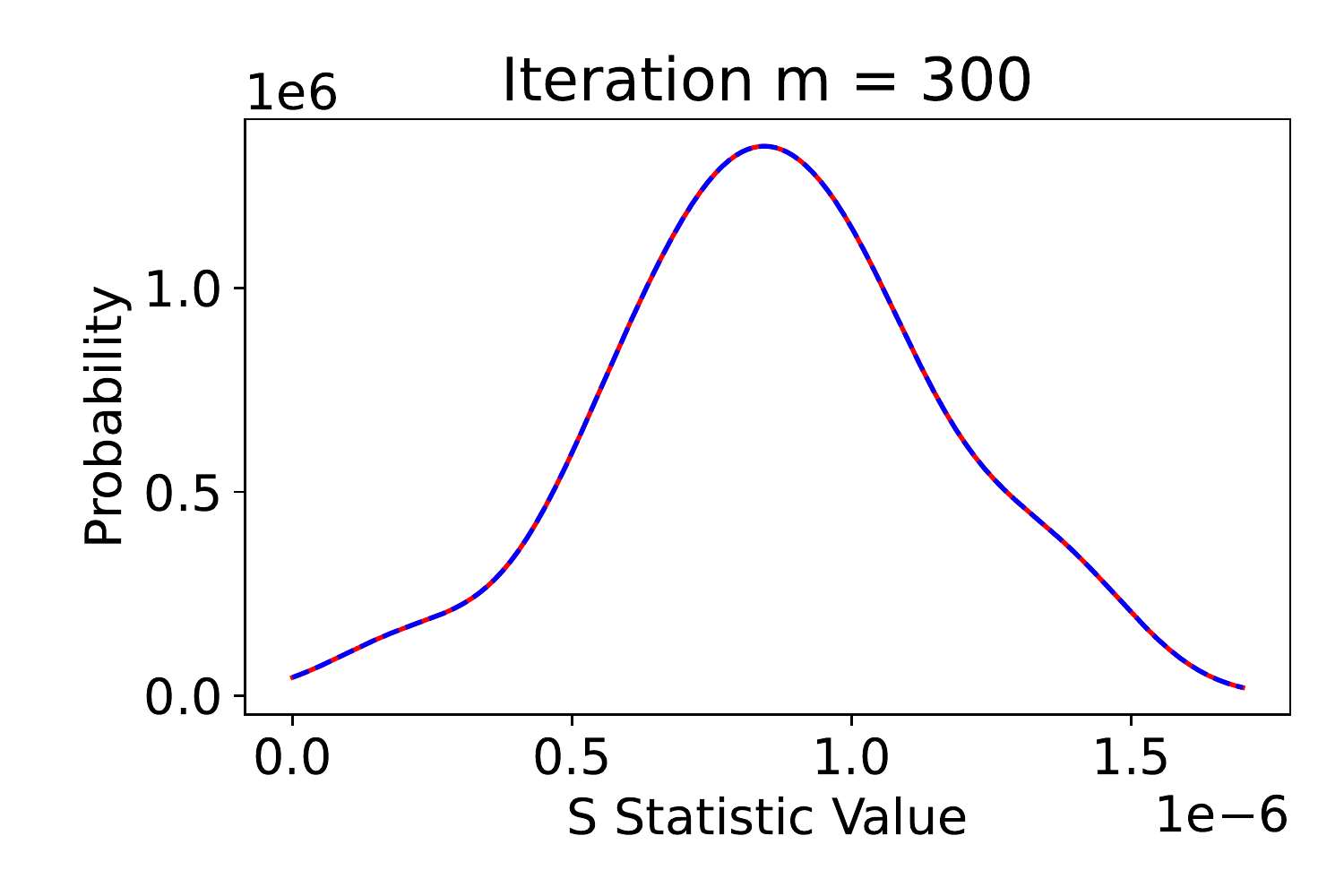} 
  \caption{$S$-statistic samples and traces for BayesCG under the Krylov prior and full posteriors at $m=10, 100, 300$ iterations. The solid curve represents the traces and the dashed curve the  $S$-statistic samples.}
  \label{F:SKrylov}
\end{figure}

\begin{table}
    \centering
    \begin{tabular}{c|ccc}
    Iteration & $S$-stat mean & Trace mean & Trace standard deviation \\\hline
$ 10.0 $ & $ 52.9 $ & $ 52.9 $ & $ 9.4 $ \\
$ 100.0 $ & $ 0.515 $ & $ 0.515 $ & $ 0.241 $ \\
$ 300.0 $ & $ 8.59 \times 10^{-7} $ & $ 8.59 \times 10^{-7} $ & $ 2.84 \times 10^{-7} $ \\
    \end{tabular}
    \caption{This table corresponds to \figref{F:SKrylov}. For BayesCG under the Krylov prior and full posteriors, it shows the
    $S$-statistic sample means, the trace means, and the trace standard deviations.}
    \label{Tab:SKrylov}
\end{table}

\paragraph{\figref{F:ZKrylov} and \tabref{Tab:ZKrylov}.}
The $Z$-statistic samples in
\figref{F:ZKrylov}
are concentrated at somewhat larger values than the predicted chi-squared distribution. The Kolmogorov-Smirnov statistics in \tabref{Tab:ZKrylov} are around .8 and .9, thus close to 1,
and indicate very little overlap between $Z$-statistic samples and chi-squared distribution. Thus,  BayesCG under full Krylov posteriors is somewhat optimistic. 

These numerical results differ from  \tref{T:KrylovZ}, which predicts $Z$-statistic samples equal to $\kry-m$.
A possible reason might be that 
the rank of the Krylov prior computed by \aref{A:ALanczos} is smaller than the exact rank. In exact arithmetic, $\rank(\Gammat_0) = \kry = n = 1806$. However, in finite precision, $\rank(\Gammat_0)$ is determined by the convergence tolerance which is set to $10^{-12}$,
resulting in $\rank(\Gammat_0) < \kry$.

\paragraph{\figref{F:SKrylov} and \tabref{Tab:SKrylov}.}
The  $S$-statistic samples in \figref{F:SKrylov}
match the traces extremely well, with \tabref{Tab:SKrylov} showing an agreement to 3 figures,
as predicted in \sref{S:BayesCGS},
Thus, the area in which the posteriors are concentrated is consistent with the error, as would be expected from a calibrated solver. 

However,  BayesCG under the Krylov prior does not behave exactly like a
calibrated solver, such as BayesCG  with random search directions
in \sref{S:ExpRand}, where all traces are concentrated at the empirical mean of the $S$-statistic samples. Thus, BayesCG under the Krylov prior is not calibrated but has a performance similar to that of a calibrated solver.

\subsubsection{Rank-50 approximate Krylov posteriors}
\label{S:ExpApprox}

\paragraph{Summary of the experiments below.} 
Both, $Z$- and $S$-statistic indicate that BayesCG under rank-50 approximate Krylov posteriors
is somewhat optimistic, and is not as close to a calibrated solver as BayesCG with full Krylov posteriors.
In contrast to the $Z$-statistic, the
respective $S$-statistic samples and traces for BayesCG under full and rank-50 posteriors are close.

\begin{figure}
  \centering
  \includegraphics[scale = .35]{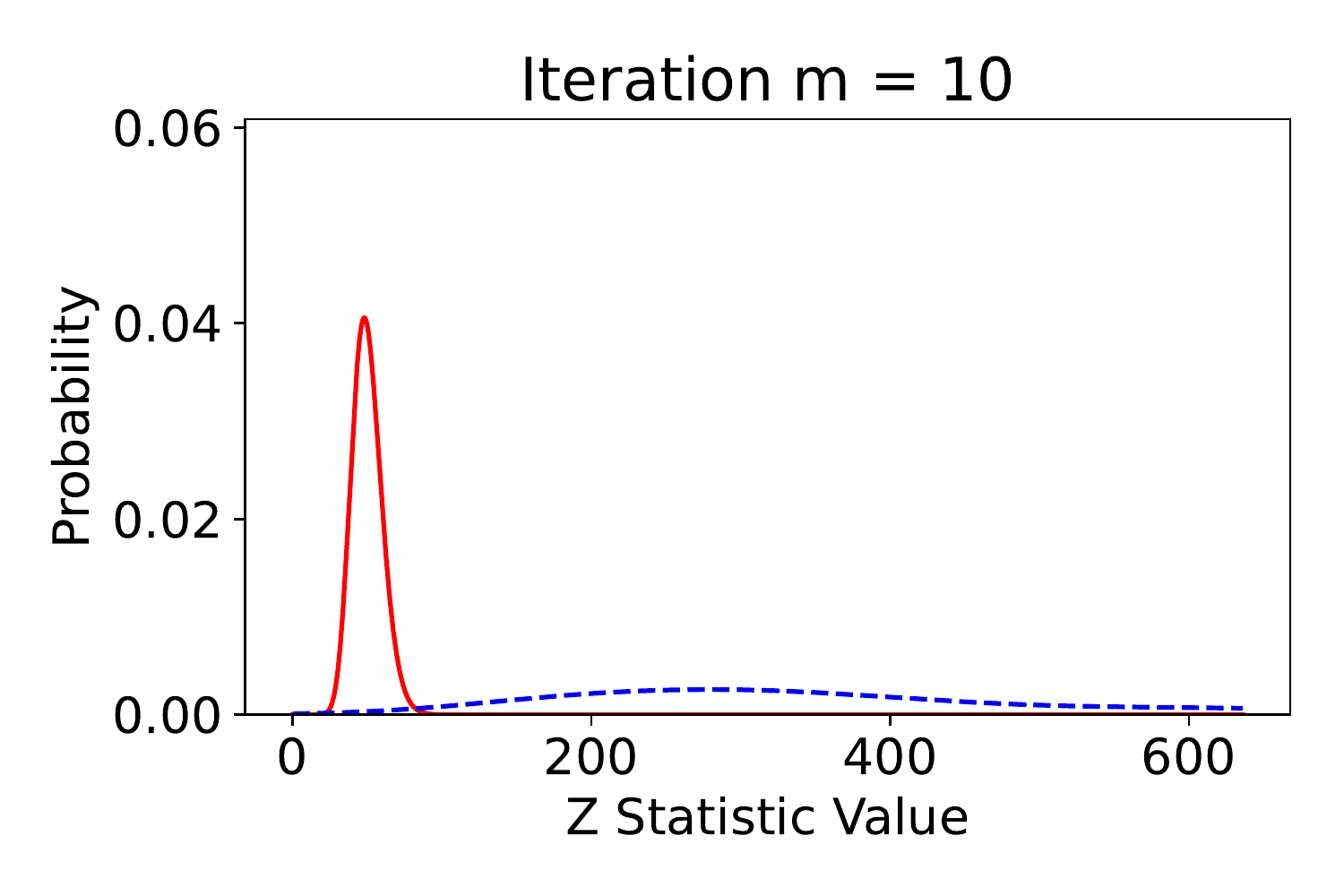}
  \includegraphics[scale = .35]{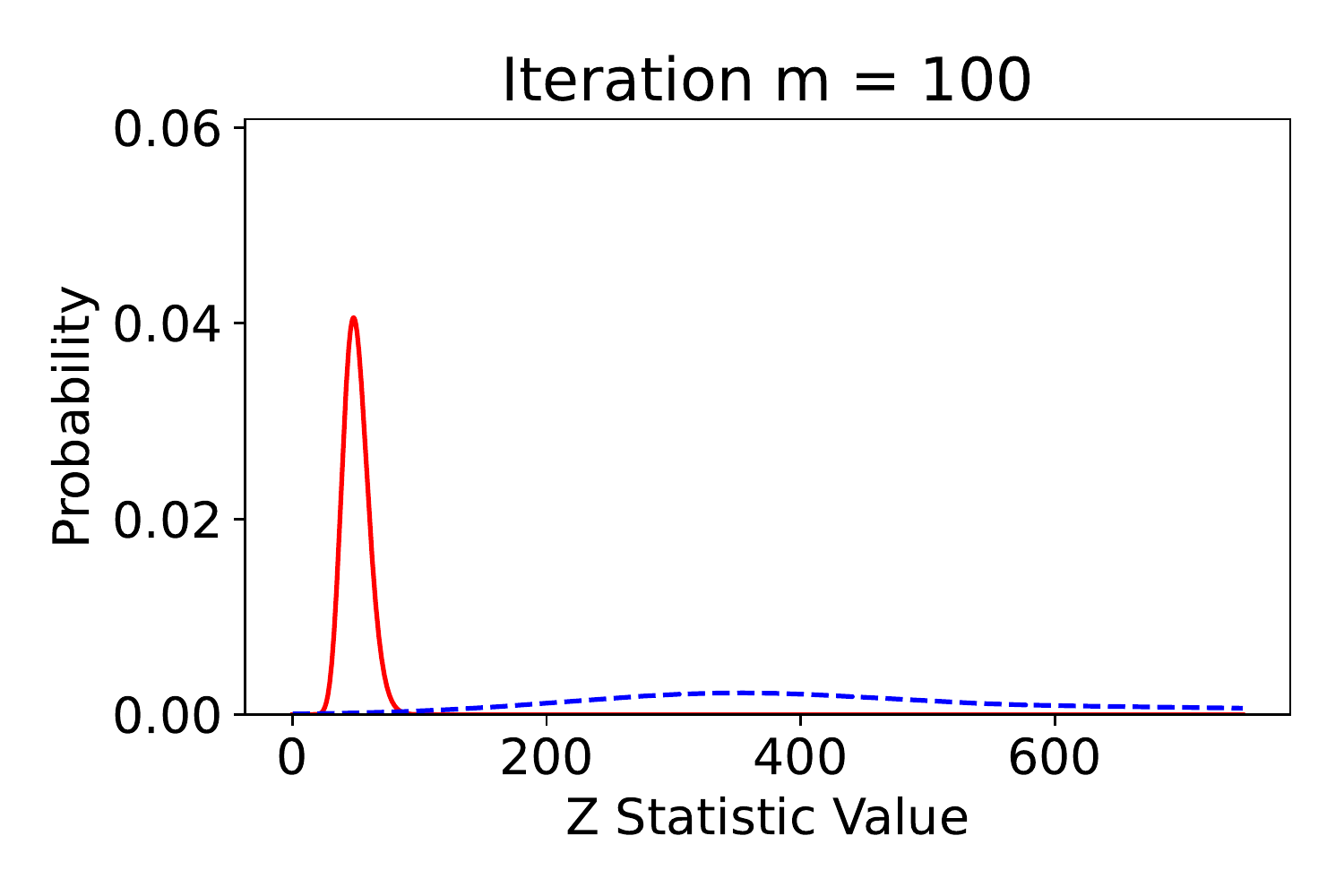} \\
  \includegraphics[scale = .35]{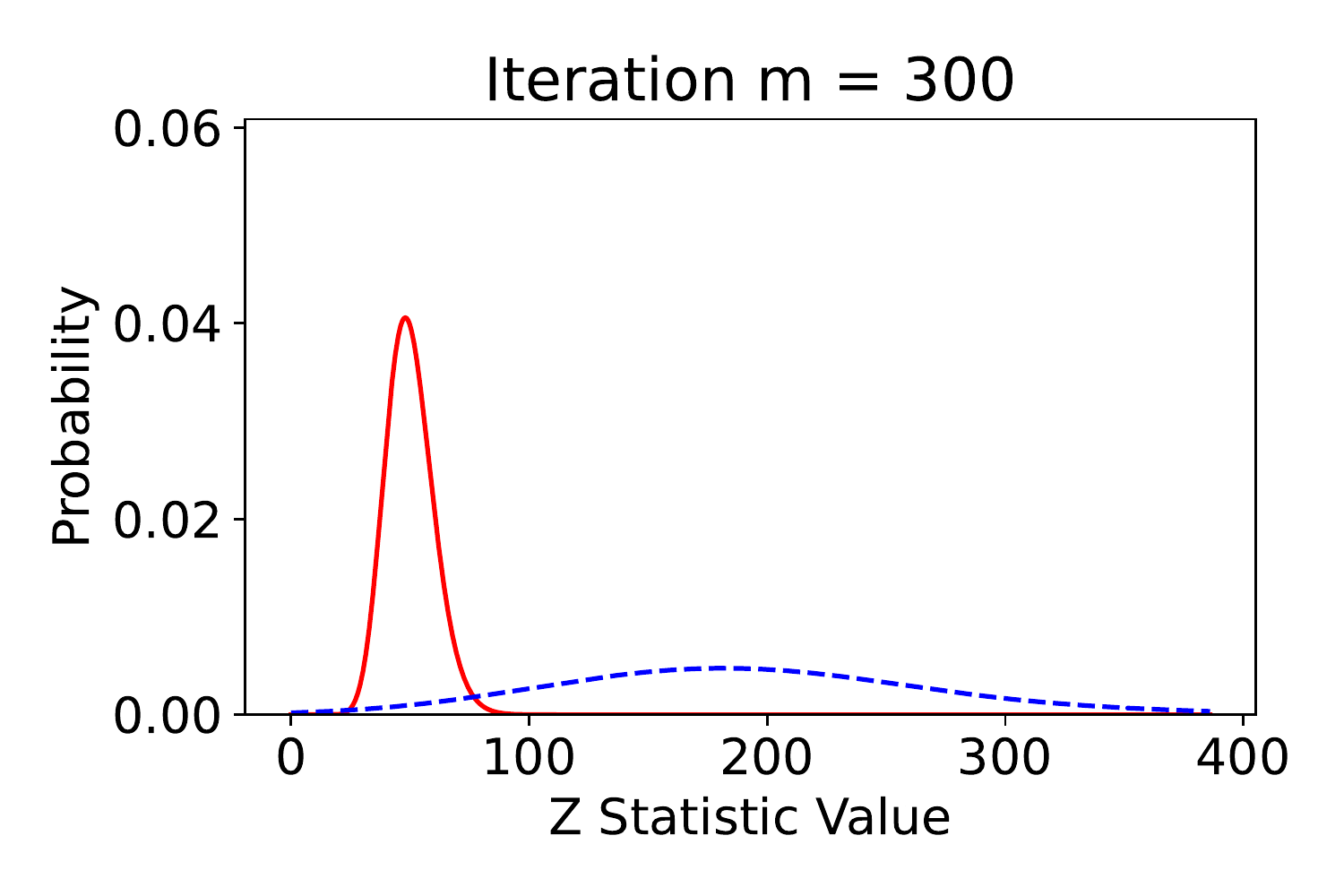}
  \caption{$Z$-statistic samples for BayesCG under rank-50 approximate Krylov posteriors
  at $m=10, 100, 300$ iterations. The solid curve represents the predicted chi-squared distribution and the dashed curve the $Z$-statistic samples.}
  \label{F:ZKrylovApprox}
\end{figure}

\begin{table}
    \centering
    \begin{tabular}{c|ccc}
    Iteration & $Z$-stat mean & $\chi^2$ mean & K-S statistic \\\hline
$ 10.0 $ & $ 319.0 $ & $ 50.0 $ & $ 1.0 $ \\
$ 100.0 $ & $ 375.0 $ & $ 50.0 $ & $ 1.0 $ \\
$ 300.0 $ & $ 194.0 $ & $ 50.0 $ & $ 1.0 $ \\
    \end{tabular}
    \caption{This table corresponds to \figref{F:ZKrylovApprox}. For BayesCG under rank-50 approximate Krylov posteriors, it shows the
    $Z$-statistic sample means; chi-squared distribution means; and Kolmogorov-Smirnov statistic between the $Z$-statistic samples and the chi-squared distribution.}
    \label{Tab:ZKrylovApprox}
\end{table}

\begin{figure}
  \centering
\includegraphics[scale = .35]{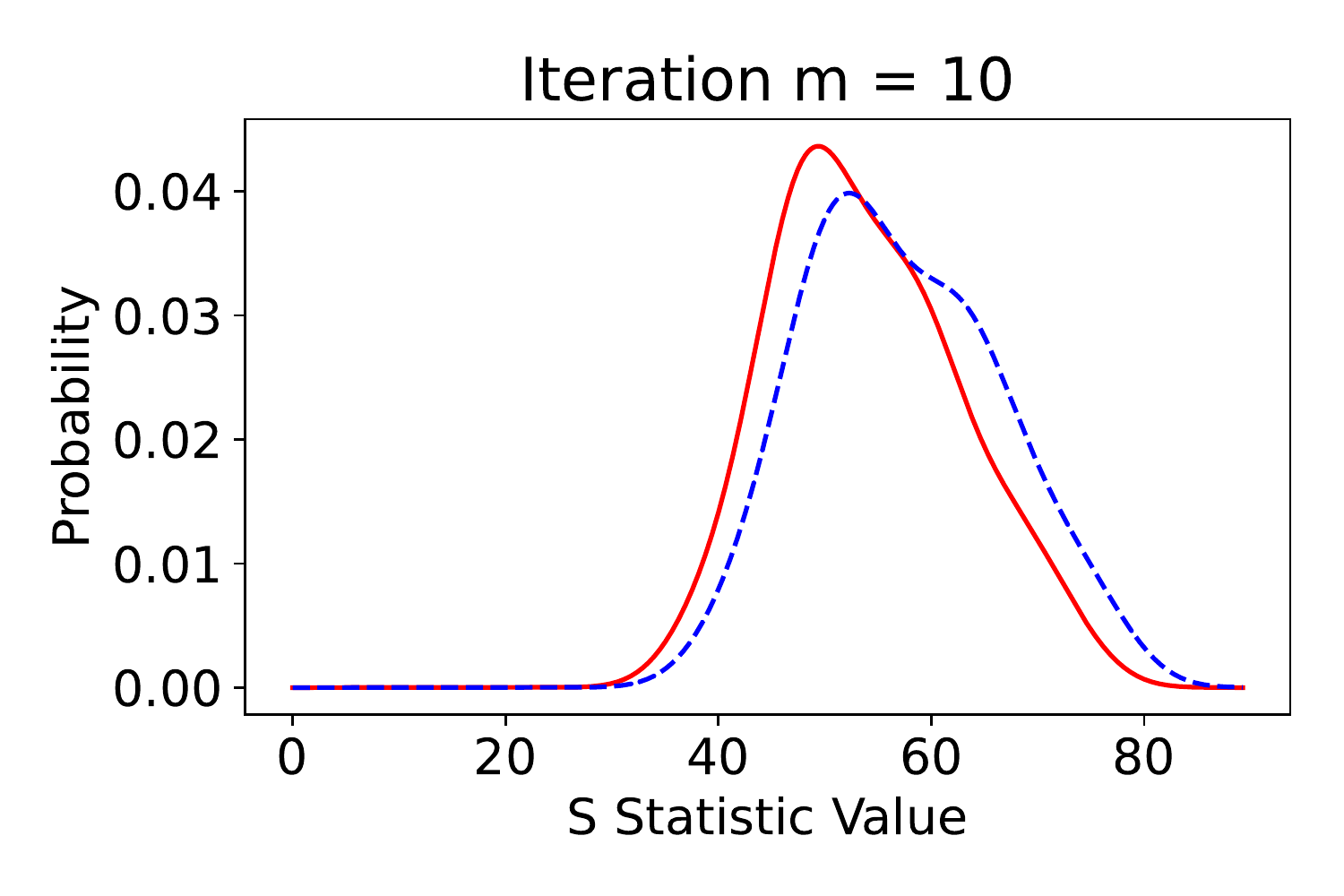} 
\includegraphics[scale = .35]{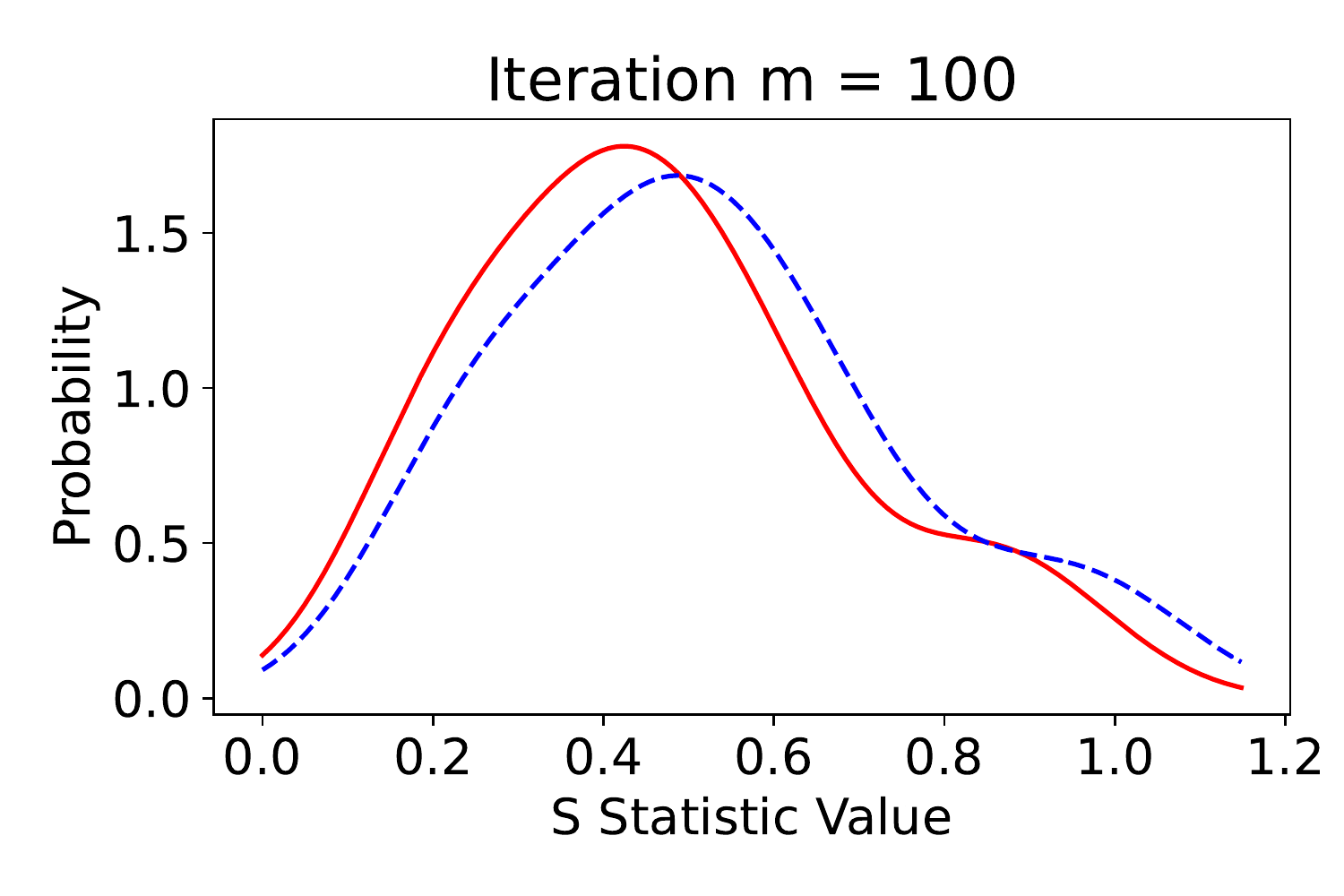} \\
\includegraphics[scale = .35]{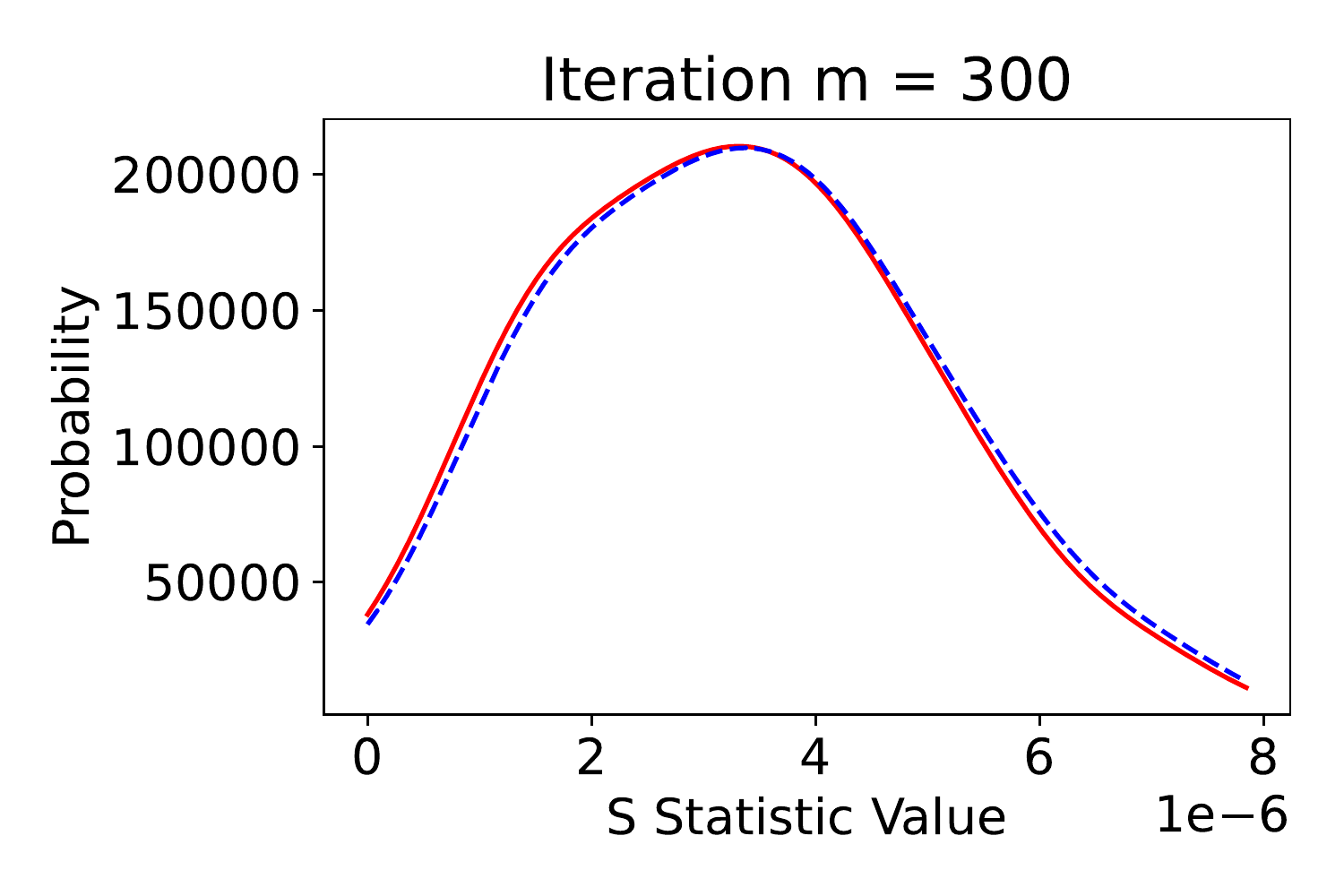}
  \caption{$S$-statistic samples and traces for BayesCG under rank-50 approximate Krylov posteriors at $m=10, 100, 300$ iterations. The solid curve represents the traces and the dashed curve the $S$-statistic samples.}
  \label{F:SKrylovApprox}
\end{figure}

\begin{table}
    \centering
    \begin{tabular}{c|ccc}
    Iteration & $S$-stat mean & Trace mean & Trace standard deviation \\\hline
$ 10.0 $ & $ 57.2 $ & $ 53.9 $ & $ 8.32 $ \\
$ 100.0 $ & $ 0.517 $ & $ 0.467 $ & $ 0.214 $ \\
$ 300.0 $ & $ 3.37 \times 10^{-6} $ & $ 3.29 \times 10^{-6} $ & $ 1.6 \times 10^{-6} $ \\
    \end{tabular}
    \caption{This table corresponds to \figref{F:SKrylovApprox}. For BayesCG under rank-50 approximate Krylov posteriors, it shows
    the $S$-statistic sample means, trace means, and trace standard deviations.}
    \label{Tab:SKrylovApprox}
\end{table}

\paragraph{\figref{F:ZKrylovApprox} and \tabref{Tab:ZKrylovApprox}.}
The $Z$-statistic samples in \figref{F:ZKrylovApprox} are concentrated around larger values than the predicted chi-squared distribution, which is steady at 50. 
All Kolmogorov-Smirnov statistics in \tabref{Tab:ZKrylovApprox} are equal to 1, indicating no overlap between $Z$-statistic samples and chi-squared distribution.
Thus, BayesCG under approximate Krylov posteriors is more optimistic than BayesCG under full posteriors.

\paragraph{\figref{F:SKrylovApprox} and \tabref{Tab:SKrylovApprox}.}
The traces in \figref{F:SKrylovApprox} 
are concentrated around slightly smaller values than the $S$-statistic samples, but
they all have the same order of magnitude, as shown in \tabref{Tab:SKrylovApprox}.
 This means, the errors are slightly larger than the area in which the posteriors are concentrated; and the posteriors slightly underestimate the errors.

A comparison of Tables \ref{Tab:SKrylov}
%\end{table} 
and~\ref{Tab:SKrylovApprox} shows
that the $S$-statistic samples and traces, respectively, for full and rank-50 posteriors are close. From the point of view of the $S$-statistic, 
BayesCG under approximate Krylov posteriors is somewhat optimistic, and close to being a calibrated solver but not as close as BayesCG under full Krylov posteriors.

\appendix

\section{Auxiliary Results}
\label{S:Aux}
We present auxiliary results required for proofs in other sections.

The stability of Gaussian distributions implies that a linear transformation
of a Gaussian random variable remains Gaussian.

\begin{lemma}[Stability of Gaussian Distributions {\cite[Section 1.2]{Muirhead}}]
  \label{L:GaussStability}
  Let $\Xrv\sim\N(\xvec,\Sigmat)$ be a Gaussian random variable
  with mean $\xvec\in\Rn$ and covariance $\Sigmat\in\Rnn$. If $\yvec\in\Rn$ and $\Fmat\in\Rnn$, then 
  \begin{equation*}
    \Zrv = \yvec+\Fmat\Xrv\sim \N(\yvec+\Fmat\xvec, \Fmat \Sigmat \Fmat^T).
  \end{equation*}
\end{lemma}

The conjugacy of Gaussian distributions implies that the distribution of a Gaussian random variable conditioned on information that linearly depends on the random variable is a Gaussian distribution.

\begin{lemma}[Conjugacy of Gaussian Distributions {\cite[Section 6.1]{Ouellette},
\cite[Corollary 6.21]{Stuart:BayesInverse}}]
  \label{L:GaussConjugacy}
  Let $\Xrv\sim\N(\xvec,\Sigmat_x)$ and $\Yrv\sim\N(\yvec,\Sigmat_y)$. The jointly Gaussian random variable $\begin{bmatrix}\Xrv^T & \Yrv^T\end{bmatrix}^T$  has the distribution
  \begin{equation*}
    \begin{bmatrix}
      \Xrv \\ \Yrv
    \end{bmatrix} \sim \N\left(\begin{bmatrix} \xvec\\ \yvec \end{bmatrix},\begin{bmatrix}\Sigmat_x & \Sigmat_{xy} \\ \Sigmat_{xy}^T & \Sigmat_y\end{bmatrix} \right),
  \end{equation*}
  where $\Sigmat_{xy} \equiv \Cov(\Xrv,\Yrv) = \Exp[(\Xrv-\xvec)(\Yrv-\yvec)^T]$ and the conditional distribution of $\Xrv$ given $\Yrv$ is
  \begin{equation*}
    (\Xrv \ | \ \Yrv) \sim \N(\overbrace{\xvec + \Sigmat_{xy}\Sigmat_y^\dagger(\Yrv - \yvec)}^{\text{mean}}, \ \overbrace{\Sigmat_x - \Sigmat_{xy}\Sigmat_y^\dagger\Sigmat_{xy}^T}^{\text{covariance}}).
  \end{equation*}
\end{lemma}

We show how to transform a $\Bmat$-orthogonal matrix into an orthogonal
matrix.

\begin{lemma}
  \label{L:BOrth}
  Let $\Bmat\in\Rnn$ be symmetric positive definite, and 
  let $\Hmat\in\Rnn$ be a $\Bmat$-orthogonal matrix
  with $\Hmat^T\Bmat\Hmat = \Hmat\Bmat\Hmat^T = \Imat$. Then
  \begin{equation*}
    \Umat \equiv \Bmat^{1/2}\Hmat
  \end{equation*}
  is an orthogonal matrix with $\Umat^T\Umat= \Umat\Umat^T=\Imat$.
\end{lemma}

\begin{proof}
The symmetry of $\Bmat$ and the $\Bmat$-orthogonality of $\Hmat$ imply
  \begin{equation*}
    \Umat^T\Umat = \Hmat^T\Bmat\Hmat = \Imat.
  \end{equation*}
From the orthonormality of the columns of $\Umat$, and the fact  that $\Umat$ is square follows that $\Umat$ is an orthogonal matrix \cite[Definition 2.1.3]{HornJohnson85}.
  \qed \end{proof}

\begin{definition}{{\cite[Section 7.3]{HornJohnson85}}}
  \label{D:MoorePenrose}
  The \emph{thin singular value decomposition} of the rank-$p$ matrix $\Gmat\in\Real^{m\times n}$ is
  \begin{equation*}
    \Gmat = \Umat\Dmat\Wmat^T,
  \end{equation*}
  where $\Umat\in\Real^{m\times p}$ and $\Wmat\in\Real^{n\times p}$ are matrices with orthonormal columns and $\Dmat\in\Real^{p\times p}$ is a diagonal matrix with positive diagonal elements. The \emph{Moore-Penrose inverse} of $\Gmat$ is
  \begin{equation*}
    \Gmat^\dagger = \Wmat\Dmat^{-1}\Umat^T.
  \end{equation*}
\end{definition}

If a matrix has full column-rank or full row-rank, then its Moore-Penrose can be expressed in terms of the matrix itself. Furthermore, the 
Moore-Penrose
inverse of a product is equal to the product of the Moore-Penrose inverses,
provided
the first matrix has full column-rank and the second matrix has full row-rank.

\begin{lemma}[{\cite[Corollary 1.4.2]{CM09}}]
  \label{L:PseudoProduct}
  Let $\Gmat\in\Real^{m\times n}$ and $\Jmat\in\Real^{n\times p}$ have full column and row rank respectively, so $\rank(\Gmat)=\rank(\Jmat)=n$.
  The Moore-Penrose inverses of $\Gmat$ and $\Jmat$ are
  \begin{equation*}
    \Gmat^\dagger = (\Gmat^T\Gmat)^{-1}\Gmat^T \quad \text{and} \quad \Jmat^\dagger = \Jmat^T(\Jmat\Jmat^T)^{-1}
  \end{equation*}
  respectively, and the Moore-Penrose inverse of the product equals
  \begin{equation*}
    (\Gmat\Jmat)^\dagger = \Jmat^\dagger\Gmat^\dagger.
  \end{equation*}
\end{lemma}

Below is an explicit expression for the mean of a quadratic form of Gaussians.

\begin{lemma}[{\cite[Sections 3.2b.1--3.2b.3]{Mathai}}]
  \label{L:QuadExp}
  Let $\Zrv\sim\N(\xvec_z,\Sigmat_z)$ be a Gaussian
  random variable in $\Rn$, and  $\Bmat\in\Rnn$ be symmetric positive definite. 
  The mean  of $\Zrv^T\Bmat\Zrv$ is
  \begin{align*}
    \Exp[\Zrv^T\Bmat\Zrv] &= \trace(\Bmat\Sigmat_z) + \xvec_z^T\Bmat\xvec_z.
  \end{align*}
\end{lemma}

We show that the squared Euclidean norm of a Gaussian random variable with an orthogonal projector as its covariance matrix is distributed according to a chi-squared distribution.

\begin{lemma}
  \label{L:ChiProjector}
  Let $\Zrv\sim\N(\zerovec,\Pmat)$ be a Gaussian random variable in $\Rn$. If the covariance matrix $\Pmat$ is an orthogonal projector, that is, if $\Pmat^2 = \Pmat$ and $\Pmat = \Pmat^T$, then
  \begin{equation*}
  \|\Xrv\|_2^2=  (\Xrv^T\Xrv) \sim \chi_p^2,
  \end{equation*}
  where $p = \rank(\Pmat)$.
\end{lemma}

\begin{proof}
We express the projector in terms of orthonormal matrices and then use the invariance of the 2-norm under orthogonal matrices and the stability of Gaussians.
  
  Since $\Pmat$ is an orthogonal projector, there exists $\Umat_1\in\Real^{n\times p}$ such that $\Umat_1\Umat_1^T = \Pmat$ and $\Umat_1^T\Umat = \Imat_p$.
Choose $\Umat_2\in\Real^{n\times (n-p)}$ so that $\Umat = \begin{bmatrix} \Umat_1 & \Umat_2 \end{bmatrix}$ is an orthogonal matrix. Thus,
  \begin{equation}
    \label{Eq:ChiProjector}
    \Xrv^T\Xrv = \Xrv^T\Umat\Umat^T\Xrv = \Xrv^T\Umat_1\Umat_1^T\Xrv + \Xrv^T\Umat_2\Umat_2^T\Xrv.
  \end{equation}
 Lemma~\ref{L:GaussStability} implies that $\Yrv = \Umat_1^T\Xrv$ is distributed according to a Gaussian distribution with mean $\zerovec$ and covariance $\Umat_1^T\Umat_1\Umat_1^T\Umat = \Imat_p$. Similarly, $\Zrv = \Umat_2^T\Xrv$ is distributed according to a Gaussian distribution with mean $\zerovec$ and covariance $\Umat_2^T\Umat_1\Umat_1^T\Umat_2 = \zerovec$, thus $\Zrv = \zerovec$.

  Substituting $\Yrv$ and $\Zrv$ into \eref{Eq:ChiProjector} gives
$\Xrv^T\Xrv = \Yrv^T\Yrv + \zerovec^T\zerovec$. 
From $\Yrv\sim\N(\zerovec,\Imat_p)$ follows $(\Xrv^T\Xrv)\sim\chi_p^2$.
  \qed \end{proof}

\begin{lemma}
\label{L:ExpS}
If $\Amat\in\Rnn$ is symmetric positive definite, and
 $\Mrv\sim\N(\xvec_\mu\Sigmat_\mu)$ and $\Nrv\sim\N(\xvec_\nu,\Sigmat_\nu)$ are independent random variables in $\Rn$, then
\begin{equation*}
    \Exp[\|\Mrv-\Nrv\|_\Amat^2] = \|\xvec_\mu-\xvec_\nu\|_\Amat^2 + \trace(\Amat\Sigmat_\mu) + \trace(\Amat\Sigmat_\nu).
\end{equation*}
\end{lemma}

\begin{proof}
The random variable $\Mrv-\Nrv$ has mean
$\Exp[\Mrv-\Nrv] = \xvec_\mu-\xvec_\nu$,
and covariance 
\begin{align*}
  \Sigmat_{\Mrv-\Nrv}
  %\equiv
 % \Exp\left[\left((\Mrv-\xvec_\mu)-(\Nrv-\xvec_\nu)\right)\left((\Mrv-\xvec_\mu)-(\Nrv-\xvec_\nu)\right)^T\right]\\
 & \equiv\Cov(\Mrv-\Nrv,\Mrv-\Nrv) \notag\\
&= \Cov(\Mrv,\Mrv) + \Cov(\Nrv,\Nrv) -\Cov(\Mrv,\Nrv) - \Cov(\Nrv,\Mrv)\notag\\
&= \Cov(\Mrv,\Mrv) + \Cov(\Nrv,\Nrv) =
\Sigmat_{\mu}+\Sigmat_{\nu},\label{Eq:SCov}
\end{align*}
where the covariances $\Cov(\Mrv,\Nrv) =\Cov(\Nrv,\Mrv) = 0$ because $\Mrv$ and $\Nrv$ are independent. Now apply \lref{L:QuadExp} to $M-N$.
\qed
\end{proof}

\section{Algorithms}
\label{S:Algs}
We present algorithms for the modified Lanczos method (Section \ref{S:ALanczos}), BayesCG with random search directions (Section \ref{S:BayesCGC}), 
BayesCG with covariances in factored form
(Section \ref{S:BayesCGF}), and BayesCG under the Krylov prior
(Section~\ref{S:BayesCGK}).

\subsection{Modified Lanczos method}
\label{S:ALanczos}

The Lanczos method 
\cite[Algorithm 6.15]{Saad}
produces an orthonormal basis for the Krylov space $\Kcal_\kry(\Amat,\vvec_1)$, while  the modified version in \aref{A:ALanczos}  produces an $\Amat$-orthonormal basis.

\begin{algorithm}
\caption{Modified Lanczos Method}
\label{A:ALanczos}
\begin{algorithmic}[1]
  \State \textbf{Input:} spd $\Amat\in\Rnn$, $\vvec_1\in\Rn$, basis dimension $m$, convergence tolerance $\varepsilon$
  \State $\vvec_0 = \zerovec \in \Rn$
  \State $i = 1$
  \State $\beta = (\vvec_i^T\Amat\vvec_i)^{1/2}$
  \State $\vvec_{i} = \vvec_i / \beta$
  \While{ $i \leq m$}
  \State $\wvec = \Amat\vvec_i-\beta\vvec_{i-1}$
  \State $\alpha = \wvec^T\Amat\vvec_i$
  \State $\wvec = w - \alpha\vvec_i$
  \State $\wvec = \wvec - \sum_{j=1}^i \vvec_j\vvec_j^T\Amat\wvec$ \Comment{Reorthogonalize $\wvec$} \label{A:ALanczos:Reorth1}
  \State $\wvec = \wvec - \sum_{j=1}^i \vvec_j\vvec_j^T\Amat\wvec$ \label{A:ALanczos:Reorth2}
  \State $\beta = (\wvec^T\Amat\wvec)^{1/2}$
  \If{ $\beta < \varepsilon$}
  \State Exit while loop
  \EndIf
  \State $i = i+1$
  \State $\vvec_i = \wvec / \beta$
  \EndWhile
  \State $m = i-1$ \Comment{Number of basis vectors}
\State \textbf{Output:} $\{\vvec_1,\vvec_2,\ldots,\vvec_m\}$\Comment{$\Amat$-orthonormal basis of $\Kcal_m(\Amat,\vvec_1)$}
\end{algorithmic}
\end{algorithm}

\aref{A:ALanczos} reorthogonalizes the basis vectors $\vvec_i$ with Classical Gram-Schmidt performed twice, see Lines~\ref{A:ALanczos:Reorth1} and~\ref{A:ALanczos:Reorth2}. This reorthogonalization technique can be implemented efficiently and produces vectors that are orthogonal to machine precision \cite{Giraud:CGS2Exp,Giraud:CGS2Theory}.

\subsection{BayesCG with random search directions}
\label{S:BayesCGC}

The version of BayesCG 
in \aref{A:BayesCGC}
is designed to be calibrated because the  search directions do not depend on $\xvec_*$, hence the posteriors do not depend on $\xvec_*$ either \cite[Section 1.1]{CIOR20}.

After  sampling an initial random search direction $\svec_1\sim\N(\zerovec,\Imat)$,
\aref{A:BayesCGC}
computes an $\Amat\Sigmat_0\Amat$-orthonormal basis for the Krylov space $\Kcal_m(\Amat\Sigmat_0\Amat,\svec_1)$ with  \aref{A:ALanczos}. Then \aref{A:BayesCGC} computes the BayesCG posteriors directly with \eref{Eq:XmTheory} and~\eref{Eq:SigmTheory} from \tref{T:BayesCG}. 
The numerical experiments in \sref{S:Experiments} run \aref{A:BayesCGC} with the inverse prior $\mu_0 = \N(\zerovec,\Amat^{-1})$.

\begin{algorithm}
\caption{BayesCG with random search directions}
\label{A:BayesCGC}
\begin{algorithmic}[1]
  \State{\textbf{Inputs}: spd $\Amat\in\Rnn$, $\bvec\in\Rn$, prior $\mu_0 = \N(\xvec_0,\Sigmat_0)$, iteration count $m$}
  \State{$  {\rvec}_0 =  {\bvec}- {\Amat\xvec}_0$} \Comment{Initial residual}
  \State{Sample $\svec_1$ from $\N(\zerovec,\Imat)$} \Comment{Initial search direction}
  \State Compute columns of $\Smat$ with \aref{A:ALanczos}
  \State $\Lammat_m = \Smat_m^T\Amat\Sigmat_0\Amat\Smat_m$ \Comment{$\Lammat_m$ is diagonal}
  \State $\xvec_m = \xvec_0 + \Sigmat_0\Amat\Smat_m\Lammat_m^{-1}\Smat_m^T\rvec_0$ \Comment{Compute posterior mean with \eref{Eq:XmTheory}}
  \State $\Sigmat_m = \Sigmat_0 - \Sigmat_0\Amat\Smat_m\Lammat_m^{-1}\Smat_m^T\Amat\Sigmat_0$ \Comment{Compute posterior covariance with \eref{Eq:SigmTheory}}
  \State \textbf{Output:} {$\mu_m=\mathcal{N}(\xvec_m, \Sigmat_m)$}
\end{algorithmic}
\end{algorithm}

\subsection{BayesCG with covariances in factored form}
\label{S:BayesCGF}

Algorithm~\ref{A:BayesCGF} takes as input a general
prior covariance $\Sigmat_0$ in factored form, and subsequently maintains the posterior
covariances $\Sigmat_m$ in factored form as well.
Theorem~\ref{T:PosteriorFactor}
presents the correctness proof for
Algorithm~\ref{A:BayesCGF}.

\begin{theorem}
  \label{T:PosteriorFactor}
Under the conditions of Theorem~\ref{T:BayesCG},
if $\Sigmat_0 = \Fmat_0\Fmat_0^T$ for 
  $\Fmat_0\in\Real^{n\times \ell}$ and some $m \leq \ell\leq n$, then
  $\Sigmat_m = \Fmat_m\Fmat_m^T$ with
  \begin{equation*}
    \Fmat_m = \Fmat_0\left(\Imat - \Fmat_0^T\Amat\Smat_m(\Smat_m^T\Amat\Fmat_0\Fmat_0^T\Amat\Smat_m)^{-1}\Smat_m\Amat\Fmat_0\right)\in\Real^{n\times \ell},\qquad 1\leq m\leq n.
  \end{equation*}
\end{theorem}

\begin{proof}
Fix $m$. Substituting $\Sigmat_0 = \Fmat_0\Fmat_0^T$ into \eref{Eq:SigmTheory} and factoring out $\Fmat_0$ on the left and $\Fmat_0^T$ on the right gives
$\Sigmat_m = \Fmat_0\Pmat\Fmat_0^T$
  where 
  \begin{align*}
\Pmat &\equiv \Imat - \Fmat_0^T\Amat\Smat_m(\Smat_m^T\Amat\Fmat_0\Fmat_0^T\Amat\Smat_m)^{-1}\Smat_m\Amat\Fmat_0\\
&=(\Imat - \Qmat(\Qmat^T\Qmat)^{-1}\Qmat^T)
\qquad\text{where}\quad \Qmat \equiv \Fmat_0^T\Amat\Smat_m.
\end{align*}
Show that $\Pmat$ is a projector,
  \begin{align*}
    \Pmat^2 &= \Imat - 2\Qmat(\Qmat^T\Qmat)^{-1}\Qmat^T + \Qmat(\Qmat^T\Qmat)^{-1}\Qmat^T\Qmat(\Qmat^T\Qmat)^{-1}\Qmat^T \\
            &= \Imat - \Qmat(\Qmat^T\Qmat)^{-1}\Qmat^T = \Pmat.
  \end{align*}
Hence $\Sigmat_m = \Fmat_0\Pmat\Fmat_0^T = \Fmat_0\Pmat\Pmat\Fmat_0^T = \Fmat_m\Fmat_m^T$.
   \end{proof}

\begin{algorithm}
\caption{BayesCG with covariances in factored form}
\label{A:BayesCGF}
\begin{algorithmic}[1]
  \State \textbf{Input:} spd $\Amat\in\Rnn$, $\bvec\in\Rn$, $\xvec_0\in\Rn$, $\Fmat_0\in\Real^{n\times \ell}$ \Comment{need $\xvec_*-\xvec_0 \in\range(\Sigmat_0)$}
\State{$  \rvec_0 =  \bvec- \Amat\xvec_0$} 
\State{$ \svec_1 =  \rvec_0$} 
\State{$\Pmat = \zerovec \in \Rnn$}
\State{$m=0$} 
\While{not converged}
\State{$m = m+1$}
\State{$\Pmat(:,m) = \Fmat_0^T\Amat\svec_m$} \Comment{Save column $m$ of $\Pmat$}
\State{$\qvec = \Fmat_0\Pmat(:,m)$} \Comment{Compute $\qvec = \Sigmat_0\Amat\svec_m$}
\State{$\eta_m = \svec_m^T\Amat\qvec$}
\State{$\Pmat(:,m) = \Pmat(:,m)\big/\eta_m$} \Comment{Normalize column $m$ of $\Pmat$}
\State{$\alpha_m = \left( \rvec_{m-1}^T  \rvec_{m-1}\right)\big/\eta_m$}
\State{$ \xvec_m =  \xvec_{m-1} + \alpha_m  \qvec $} 
\State{$ \rvec_m =  \rvec_{m-1} - \alpha_m \Amat\qvec$}
\State{$\beta_m = \left( \rvec_m^T \rvec_m\right)\big/\left( \rvec_{m-1}^T\rvec_{m-1}\right)$}
\State{$ \svec_{m+1} =  \rvec_m+\beta_m  \svec_m$} 
\EndWhile
\State{$\Pmat = \Pmat(:,1:m)$} \Comment{Discard unused columns of $\Pmat$}
\State{$\Fmat_m = \Fmat_0(\Imat- \Pmat\Pmat^T)$}
\State \textbf{Output:} $\xvec_m$, $\Fmat_m$ \Comment{Final posterior}
\end{algorithmic}
\end{algorithm}

\subsection{BayesCG under the Krylov prior}
\label{S:BayesCGK}
We present algorithms for BayesCG under full Krylov posteriors (Section~\ref{s_FK}) and under approximate Krylov posteriors (Section~\ref{s_AK}).

\subsubsection{Full Krylov posteriors}\label{s_FK}

 \aref{A:BayesCGKFull} computes the following:
a matrix $\Vmat$ whose columns are an $\Amat$-orthonormal basis for $\mathcal{K}_g(\Amat,\rvec_0)$;
 the diagonal matrix
$\Phimat$ in \eref{Eq:PhiDef}; and the posterior mean $\xvec_m$ in \eref{Eq:DirectPosterior}.
The output consists of the posterior mean $\xvec_m$, and the factors $\Vmat_{m+1:\kry}$ and $\Phimat_{m+1:\kry}$ for the posterior covariance.

\begin{algorithm}
\caption{BayesCG under the Krylov prior with full posteriors}
\label{A:BayesCGKFull}
\begin{algorithmic}[1]
  \State{\textbf{Inputs}: spd $\Amat\in\Rnn$, $\bvec\in\Rn$, $\xvec_0\in\Rn$, iteration count $m$}
  \State{$  {\rvec}_0 =  {\bvec}- {\Amat\xvec}_0$} \Comment{Initial residual}
  \State{$ {\vvec}_1 =  {\rvec}_0$} \Comment{Initial search direction}
  \State Compute columns of $\Vmat$ with \aref{A:ALanczos}
  \State $\Phimat = \mathrm{diag}((\Vmat^T\rvec_0)^2)$ \Comment{Compute $\Phimat$ with \eref{Eq:PhiDef}}
  \State $\xvec_m = \xvec_0 + \Vmat_{1:m}\Vmat_{1:m}^T\rvec_0$ \Comment{Compute posterior mean with \eref{Eq:DirectPosterior}}
  \State \textbf{Output:} {$\xvec_m$, $\Vmat_{m+1:\kry}$, $\Phimat_{m+1:\kry}$}
\end{algorithmic}
\end{algorithm}

\subsubsection{Approximate Krylov posteriors}\label{s_AK}
\aref{A:BayesCGWithoutBayesCG} computes
rank-$d$ approximate Krylov posteriors in
two main steps: 
(i) posterior mean and iterates
$\xvec_m$ in Lines 5-14; and
(ii) factorization of the posterior covariance $\widehat{\Gammat}_m$ in Lines 16-26.

\begin{algorithm}
\caption{BayesCG under the Krylov prior \cite[Algorithm 3.1]{RICO21}}
\label{A:BayesCGWithoutBayesCG}
\begin{algorithmic}[1]
  \State{\textbf{Inputs}: spd $\Amat\in\Rnn$, $\bvec\in\Rn$, $\xvec_0\in\Rn$, iteration count $m$, posterior rank $d$}
  \State{$  {\rvec}_0 =  {\bvec}- {\Amat\xvec}_0$} \Comment{Initial residual}
  \State{$ {\vvec}_1 =  {\rvec}_0$} \Comment{Initial search direction}
  \State{$i = 0$} \Comment{Initial iteration counter}
  \While{$i < m$} \Comment{CG recursions for posterior means}
  \State{$i = i+1$} \Comment{Increment iteration count}
  \State{$ \eta_i = \vvec_i^T\Amat\vvec_i$}
  \State{$\gamma_i = (\rvec_{i-1}^T\rvec_{i-1})\big/\eta_i$} \Comment{Next step size}
  \State{$ {\xvec}_{i} =  {\xvec}_{i-1} +   \gamma_i \vvec_i $} \Comment{Next iterate}
  \State{$ {\rvec}_{i} =  {\rvec}_{i-1}- \gamma_i \Amat\vvec_i$} \Comment{Next residual}
  \State{$\delta_i = (\rvec_i^T\rvec_i)\big/(\rvec_{i-1}^T\rvec_{i-1})$}
  \State{$\vvec_{i+1} = \rvec_i + \delta_i\vvec_i$} \Comment{Next search direction}
  \EndWhile
  \State{$d = \min\{d,\kry - m\}$} \Comment{Compute full rank posterior if $d > \kry-m$}
  \State{$\Vmat_{m+1:m+d} = \vzero_{n\times d}$} \Comment{Initialize approximate posterior matrices}
  \State{$\Phimat_{m+1:m+d} = \vzero_{d\times d}$}
  \For{$j=m+1:m+d$} \Comment{$d$ additional iterations for posterior covariance}
  \State{$\eta_{j} = \vvec_j^T\Amat\vvec_j$}
  \State{$\gamma_j = (\rvec_{j-1}^T\rvec_{j-1})\big/\eta_j$}
  \State{$\Vmat(:,j) = \vvec_{j}\big/\sqrt{\eta_j}$} \Comment{Next column of $\Vmat_{m+1,m+d}$}
  \State{$\Phimat(j,j) = \gamma_{j} \|\rvec_{j-1}\|_2^2$} \Comment{Next diagonal element of $\Phimat_{m+1,m+d}$}  
  \State{$\rvec_j = \rvec_{j-1} - \gamma_j \Amat\vvec_j$}
  \State{$\delta_j = (\rvec_j^T\rvec_j)\big/(\rvec_{j-1}^T\rvec_{j-1})$}
  \State{$\vvec_{j+1} = \rvec_j + \delta_j\vvec_j$}
  \Comment{Next un-normalized column  of $\Vmat_{m+1,m+d}$}
  \EndFor
  \State \textbf{Output:} {$\xvec_m$, $\Vmat_{m+1:m+d}$, $\Phimat_{m+1:m+d}$}
\end{algorithmic}
\end{algorithm}

\bibliography{BCGSources}

\begin{thebibliography}{10}

\bibitem{bcsstk14}
{BCSSTK14: BCS Structural Engineering Matrices (linear equations) Roof of the
  Omni Coliseum, Atlanta}.

\bibitem{Bartels}
S.~Bartels, J.~Cockayne, I.~C.~F. Ipsen, and P.~Hennig.
\newblock Probabilistic linear solvers: a unifying view.
\newblock {\em Stat. Comput.}, 29(6):1249--1263, 2019.

\bibitem{Berger1985}
J.~O. Berger.
\newblock {\em Statistical Decision Theory and {B}ayesian Analysis}.
\newblock Springer New York, 1985.

\bibitem{Berljafa}
M.~Berljafa and S.~G\"{u}ttel.
\newblock Generalized rational {K}rylov decompositions with an application to
  rational approximation.
\newblock {\em SIAM J. Matrix Anal. Appl.}, 36(2):894--916, 2015.

\bibitem{CM09}
S.~L. Campbell and C.~D. Meyer.
\newblock {\em Generalized inverses of linear transformations}, volume~56 of
  {\em Classics in Applied Mathematics}.
\newblock Society for Industrial and Applied Mathematics (SIAM), Philadelphia,
  PA, 2009.

\bibitem{CGOS21}
J.~Cockayne, M.~M. Graham, C.~J. Oates, and T.~J. Sullivan.
\newblock Testing whether a learning procedure is calibrated, 2021.
\newblock arXiv:2012.12670.

\bibitem{CIOR20}
J.~Cockayne, I.~C.~F. Ipsen, C.~J. Oates, and T.~W. Reid.
\newblock Probabilistic iterative methods for linear systems.
\newblock {\em J. Mach. Learn. Res.}, 22 (232):1--34, 2021.

\bibitem{Cockayne:BCG}
J.~Cockayne, C.~J. Oates, I.~C.~F. Ipsen, and M.~Girolami.
\newblock A {B}ayesian conjugate gradient method (with discussion).
\newblock {\em Bayesian Anal.}, 14(3):937--1012, 2019.
\newblock Includes 6 discussions and a rejoinder from the authors.

\bibitem{BCG:Supp}
J.~Cockayne, C.~J. Oates, I.~C.~F. Ipsen, and M.~Girolami.
\newblock Supplementary material for `{A} {B}ayesian conjugate-gradient
  method'.
\newblock {\em Bayesian Anal.}, 2019.

\bibitem{Cockayne:BPNM}
J.~Cockayne, C.~J. Oates, T.~J. Sullivan, and M.~Girolami.
\newblock Bayesian probabilistic numerical methods.
\newblock {\em SIAM Rev.}, 61(4):756--789, 2019.

\bibitem{Fanaskov21}
Vladimir Fanaskov.
\newblock Uncertainty calibration for probabilistic projection methods.
\newblock {\em Stat. Comput.}, 31(5):Paper No. 56, 17, 2021.

\bibitem{Gelbrich90}
M.~Gelbrich.
\newblock On a formula for the {$L^2$} {W}asserstein metric between measures on
  {E}uclidean and {H}ilbert spaces.
\newblock {\em Math. Nachr.}, 147:185--203, 1990.

\bibitem{Giraud:CGS2Exp}
L.~Giraud, J.~Langou, and M.~Rozloznik.
\newblock The loss of orthogonality in the {G}ram-{S}chmidt orthogonalization
  process.
\newblock {\em Comput. Math. Appl.}, 50(7):1069--1075, 2005.

\bibitem{Giraud:CGS2Theory}
Luc Giraud, Julien Langou, Miroslav Rozlo\v{z}n\'{\i}k, and Jasper van~den
  Eshof.
\newblock Rounding error analysis of the classical {G}ram-{S}chmidt
  orthogonalization process.
\newblock {\em Numer. Math.}, 101(1):87--100, 2005.

\bibitem{GoVa13}
Gene~H. Golub and Charles~F. Van~Loan.
\newblock {\em Matrix Computations}.
\newblock The Johns Hopkins University Press, Baltimore, 4th edition, 2013.

\bibitem{HBH20}
J.~Hart, B.~{van Bloemen Waanders}, and R.~Herzog.
\newblock Hyperdifferential sensitivity analysis of uncertain parameters in
  {PDE}-constrained optimization.
\newblock {\em Int. J. for Uncertain. Quantif.}, 10(3):225--248, 2020.

\bibitem{Hennig}
P.~Hennig.
\newblock Probabilistic interpretation of linear solvers.
\newblock {\em SIAM J. Optim.}, 25(1):234--260, 2015.

\bibitem{HOG15}
P.~Hennig, M.~A. Osborne, and M.~Girolami.
\newblock Probabilistic numerics and uncertainty in computations.
\newblock {\em Proc. R. Soc. A.}, 471(2179):20150142, 17, 2015.

\bibitem{Hestenes}
Magnus~R. Hestenes and Eduard Stiefel.
\newblock Methods of conjugate gradients for solving linear systems.
\newblock {\em J. Research Nat. Bur. Standards}, 49:409--436, 1952.

\bibitem{Higham08}
N.~J. Higham.
\newblock {\em Functions of matrices. {Theory} and computation}.
\newblock Society for Industrial and Applied Mathematics (SIAM), Philadelphia,
  PA, 2008.

\bibitem{HornJohnson85}
R.~A. Horn and C.~R. Johnson.
\newblock {\em Matrix Analysis}.
\newblock Cambridge University Press, 1985.

\bibitem{JWHT21}
G.~James, D.~Witten, T.~Hastie, and R.~Tibshirani.
\newblock {\em An introduction to statistical learning}, volume 112.
\newblock Springer, second edition, 2021.

\bibitem{Kaltenbach12}
H.-M. Kaltenbach.
\newblock {\em A concise guide to statistics}.
\newblock Springer Briefs in Statistics. Springer, Heidelberg, 2012.

\bibitem{KLMU20}
D.~Kressner, J.~Latz, S.~Massei, and E.~Ullmann.
\newblock Certified and fast computations with shallow covariance kernels,
  2020.
\newblock arXiv:200109187.

\bibitem{Liesen}
J.~Liesen and Z.~Strakos.
\newblock {\em Krylov Subspace Methods: Principles and Analysis}.
\newblock Oxford University Press, 2013.

\bibitem{Mathai}
A.~M. Mathai and S.~B. Provost.
\newblock {\em Quadratic forms in random variables: {Theory} and applications}.
\newblock Dekker, 1992.

\bibitem{Muirhead}
R.~J. Muirhead.
\newblock {\em Aspects of multivariate statistical theory}.
\newblock John Wiley \& Sons, Inc., New York, 1982.
\newblock Wiley Series in Probability and Mathematical Statistics.

\bibitem{NW06}
J.~Nocedal and S.~J. Wright.
\newblock {\em Numerical optimization}.
\newblock Springer Series in Operations Research and Financial Engineering.
  Springer, New York, second edition, 2006.

\bibitem{Oates}
C.~J. Oates and T.~J. Sullivan.
\newblock A modern retrospective on probabilistic numerics.
\newblock {\em Stat. Comput.}, 29(6):1335--1351, 2019.

\bibitem{Ouellette}
D.~V. Ouellette.
\newblock Schur complements and statistics.
\newblock {\em Linear Algebra Appl.}, 36:187--295, 1981.

\bibitem{PZSHG12}
N.~Petra, H.~Zhu, G.~Stadler, T.J.R. Hughes, and O.~Ghattas.
\newblock An inexact {G}auss-{N}ewton method for inversion of basal sliding and
  rheology parameters in a nonlinear {S}tokes ice sheet model.
\newblock {\em J. Glaciology}, 58(211):889–903, 2012.

\bibitem{RICO21}
T.~W. Reid, I.~C.~F. Ipsen, J.~Cockayne, and C.~J. Oates.
\newblock {BayesCG} as an uncertainty aware version of {CG}, 2022.

\bibitem{Ross07}
S.~M. Ross.
\newblock {\em Introduction to probability models}.
\newblock Academic Press, Inc., Boston, MA, ninth edition, 2007.

\bibitem{Saad}
Y.~Saad.
\newblock {\em Iterative Methods for Sparse Linear Systems}.
\newblock SIAM, 2003.

\bibitem{SHB21}
A.~K. Saibaba, J.~Hart, and B.~{van Bloemen Waanders}.
\newblock Randomized algorithms for generalized singular value decomposition
  with application to sensitivity analysis.
\newblock {\em Numer. Linear Algebra Appl.}, page e2364, 2021.

\bibitem{StrakosTichy}
Z.~Strako\v{s} and P.~Tich\'{y}.
\newblock On error estimation in the conjugate gradient method and why it works
  in finite precision computations.
\newblock {\em Electron. Trans. Numer. Anal.}, 13:56--80, 2002.

\bibitem{Stuart:BayesInverse}
A.~M. Stuart.
\newblock Inverse problems: a {B}ayesian perspective.
\newblock {\em Acta Numer.}, 19:451--559, 2010.

\bibitem{Villani09}
C.~Villani.
\newblock {\em Optimal Transport, Old and New}, volume 338 of {\em Grundlehren
  der Mathematischen Wissenschaften [Fundamental Principles of Mathematical
  Sciences]}.
\newblock Springer-Verlag, Berlin, 2009.

\bibitem{WH20}
J.~Wenger and P.~Hennig.
\newblock Probabilistic linear solvers for machine learning.
\newblock In H.~Larochelle, M.~Ranzato, R.~Hadsell, M.F. Balcan, and H.~Lin,
  editors, {\em Advances in Neural Information Processing Systems}, volume~33,
  pages 6731--6742. Curran Associates, Inc., 2020.

\end{thebibliography}

\end{document}